\documentclass[11pt,reqno]{amsart}
\usepackage[utf8]{inputenc}
\usepackage[T1]{fontenc}
\usepackage{lmodern}
\usepackage[english]{babel}
\usepackage{amsmath,a4wide}
\usepackage{xfrac}
\usepackage{esint}
\usepackage{stmaryrd,mathrsfs,bm,amsthm,mathtools,yfonts,amssymb,color,braket,booktabs,graphicx,graphics,amsfonts}
\usepackage{latexsym,microtype,indentfirst,hyperref}
\usepackage{xcolor}
\usepackage{courier}

\newtheorem{theorem}{Theorem}[section]
\newtheorem{corollary}[theorem]{Corollary}

\newtheorem{lemma}[theorem]{Lemma}
\newtheorem{proposition}[theorem]{Proposition}

\theoremstyle{definition}
\newtheorem{definition}[theorem]{Definition}
\newtheorem{remark}[theorem]{Remark}

\newtheorem{assumption}[theorem]{Assumption}

\newtheorem{theoremletter}{Theorem}

\numberwithin{equation}{section}
\numberwithin{subsection}{section}

\newcommand{\Na}{\mathbb{N}} 
\newcommand{\R}{\mathbb{R}} 

\newcommand{\eps}{\varepsilon} 
\newcommand{\spt}{\mathrm{spt}} 
\newcommand{\dist}{\mathrm{dist}} 
\newcommand{\Ha}{\mathcal{H}} 
\newcommand{\Leb}{\mathcal{L}} 

\newcommand{\pa}{\partial}

\newcommand{\mres}{\mathbin{\vrule height 1.6ex depth 0pt width 
0.13ex\vrule height 0.13ex depth 0pt width 1.3ex}}

\newcommand{\weakstar}{\stackrel{*}{\rightharpoonup}}

\newcommand{\abs}[1]{\lvert#1\rvert} 

\newcommand{\V}{\mathbf{V}} 
\newcommand{\IV}{\mathbf{IV}} 
\newcommand{\var}{\mathbf{var}} 
\newcommand{\E}{\mathcal{E}} 
\newcommand{\cA}{\mathcal{A}}
\newcommand{\bE}{\mathbf{E}}

\DeclareMathOperator*{\esssup}{ess\,sup}

\title[
canonical multi--phase Brakke flows]{On the existence of canonical \\ multi--phase Brakke flows} 
\date{\today}

\author[S. Stuvard]{Salvatore Stuvard}
\address{Dipartimento di Matematica, Universit\`{a} degli Studi di Milano, Via Saldini 50, I-20133 Milano (MI), Italy}
\email{salvatore.stuvard@unimi.it}

\author[Y. Tonegawa]{Yoshihiro Tonegawa}
\address{Department of Mathematics, Tokyo Institute of Technology, 2-12-1 Ookayama, Meguro-ku, Tokyo 152-8551, Japan}
\email{tonegawa@math.titech.ac.jp}

\begin{document}

\begin{abstract}
This paper establishes the global-in-time existence of a multi-phase mean curvature flow, evolving from an arbitrary closed rectifiable initial datum, which is a Brakke flow and a BV solution at the same time. In particular, we prove the validity of an explicit identity concerning the change of volume of the evolving grains, showing that their boundaries move according to the generalized mean curvature vector of the Brakke flow. As a consequence of the results recently established by Fischer et al. in \cite{fischer}, under suitable assumptions on the initial datum, such additional property resolves the non-uniqueness issue of Brakke flows.   

\end{abstract}

\maketitle

\tableofcontents

\section{Introduction} \label{s:introduction}

Arising as the gradient flow of the area functional, the \emph{mean curvature flow} (henceforth abbreviated as MCF) is arguably the most fundamental geometric flow involving extrinsic curvatures. The unknown of MCF is a one-parameter family $\{\Gamma(t)\}_{t \geq 0}$ of surfaces in Euclidean space (or in an ambient Riemannian manifold) such that the normal velocity of the flow equals the mean curvature vector at each point for every time $t$. The initial value problem for MCF starting with a smooth closed surface $\Gamma(0)=\Gamma_0$ is locally well-posed in time, until the appearance of singularities such as shrinking or neck pinching. Numerous frameworks
of generalized solutions past singularities have been proposed: we mention, among others, the Brakke flows \cite{Brakke}, level set flows \cite{CGG,ES}, BV solutions \cite{Luckhaus,Laux1} and $L^2$ flows \cite{Mugnai,Bertini}. Existence of these possibly different generalized solutions to the 
MCF as well as 
their relations
have been studied intensively in the past 40 years or so. 

\smallskip

The aim of the present paper is to establish the global-in-time 
existence of a \emph{``canonical multi-phase'' Brakke flow} evolving from an arbitrary rectifiable initial datum. 
The attribute \emph{``multi-phase''} here refers to the fact that the evolving surfaces are, in fact, \emph{boundaries} of finitely many, but at least two, labelled open subsets
of $\R^{n+1}$ (henceforth referred to also as \emph{``grains''} or \emph{``phases''}). The MCF evolution of such objects is strongly motivated by materials science, as it describes the motion and growth of crystallites in polycrystalline materials; see e.g.\,\cite{Mullins}. While the literature concerning the two-phase MCF is rich, fewer works have been dedicated to the more general case of multi-phase MCF with at least three grains, despite its relevance in the modeling of physical processes governed by surface tension type energies.

\smallskip

In the present paper, we work with an arbitrary (but finite) number of grains. The solution we construct consists of two objects: the flow of the evolving grains 
and a Brakke flow, intertwined as follows. The Brakke flow -- a measure-theoretic generalization of MCF, particularly suited to describe the evolution of surfaces through singularities (see Definition \ref{d:Brakke flow}) -- is 
essentially supported on the topological boundary of the grains, and it keeps track of multiplicities.
Additionally, the mean curvature of the Brakke flow determines the distributional velocity at which the reduced boundary of each grain moves. As a result of the latter property, the change of volume of each grain between two instants of time can be recovered by integrating the mean curvature over the reduced boundary, a property certainly expected for a smooth MCF but quite non-trivial in a context where singularities and multiplicities occur. The attribute \emph{``canonical''} refers to this very precise interplay between the Brakke flow and the evolution of the grains. Note that Brakke flows are non-unique in general due to the nature of the formulation, but the existence of the grains prevents
redundant non-uniqueness such as sudden vanishing, 
for example. 
In the absence of higher multiplicities
of the Brakke flow, we show that the collection of the grains (or, more precisely, the collection of their indicator functions) constitutes
a \emph{BV solution} of the MCF (see Definition \ref{d:BV flow}). In this case, the grain
boundaries are, in fact, 
a smooth MCF almost everywhere in space
and time by Brakke's partial regularity theory
\cite{Brakke,Kasai-Tone,Ton-2}. 

\smallskip

In certain instances, the additional 
BV characterization may lead to the uniqueness of the canonical Brakke flow. 
For instance, in ambient dimension $n+1=2$, the recent work of Fischer et al. \cite{fischer} shows that, as soon as a \emph{strong solution} of the network flow exists (see \cite[Definition 16]{fischer} for the definition of strong solution) then any BV solution must coincide with it at least for all times until the first topology changes occur. One conclusion derived from \cite{fischer} and the
present paper is then that, if $n=1$
and if a regular network flow starting from the given initial datum exists, then the canonical Brakke flow 
constructed in the present paper from the same initial datum necessarily
coincides with that regular flow until the first topology changes; see \cite[Theorem 19]{fischer} for the precise statement (also see \cite{Laux2} for a similar uniqueness result in $n=2$).  
In fact, the work \cite{fischer} inspired the study carried out in the present paper.

\smallskip

In more precise and technical terms, the highlight of the main 
results of the present paper may be stated as follows (see the complete
statements in Section \ref{Mainresult}).
\begin{theoremletter}\label{t:main}
Let $E_{0,1},\ldots,
E_{0,N}\subset\mathbb R^{n+1}$ be mutually disjoint non-empty
open sets with $N\geq 2$ such that $\Gamma_0:=\mathbb R^{n+1}\setminus\cup_{i=1}^N E_{0,i}$ is
countably $n$-rectifiable. Assume that the $n$-dimensional Hausdorff measure of $\Gamma_0$ is finite or grows at most exponentially fast at infinity. Then,
there exist a Brakke flow $\{V_t\}_{t\geq 0}$ (see Definition \ref{d:Brakke flow} below) as well as one-parameter families $\{E_i (t)\}_{t \geq 0}$ ($i \in \{1,\ldots,N\}$) of open sets, with $\{E_1(t),\ldots,E_N(t)\}$ mutually disjoint for each $t \ge 0$, such that
\[
\|V_0\| = \Ha^n\mres_{\Gamma_0}\,, \qquad E_i (0) = E_{0,i} \quad \mbox{for every $i=1,\ldots,N$}\,,
\]
and satisfying
the following 
properties. Writing $\Gamma(t):=\mathbb R^{n+1}
\setminus \cup_{i=1}^N E_i(t)$:
\begin{itemize}
    \item[(a)] $\mathcal H^{n-1+\delta}(\Gamma(t)\triangle\,{\rm spt}\,\|V_t\|)=0$ for any $\delta>0$ and for a.e.~$t\geq 0$. 
\item[(b)] For each $i=1,\ldots,N$ and 
for any arbitrary test function $\phi = \phi (x,t)$, we have, in the sense of distributions on $\left[0,\infty\right)$, that
\begin{equation}\label{vchange}
    \frac{d}{dt}\int_{E_i(t)}\phi\,dx 
    = \int_{\partial^*E_i(t)} \phi\,h\cdot \nu_i\,
    d\mathcal H^n+\int_{E_i(t)}\frac{\partial\phi}{\partial t}\,dx\,.
\end{equation}
\item[(c)] If the Brakke flow is locally a unit density flow, then, locally, we have 
\begin{equation}
    \mathcal H^n(\Gamma(t)\triangle \cup_{i=1}^N \partial^* E_i(t))=0
    \mbox{ and }\|V_t\|=\mathcal H^n\mres_{\cup_{i=1}^N\partial^* E_i(t)}
\end{equation}
for a.e.~$t$.  
\end{itemize}
\end{theoremletter}

In the above statement, $A\triangle B$ denotes the symmetric difference of two sets
$A$ and $B$, and ${\rm spt}\,\|V_t\|$ is the support of
the weight measure $\|V_t\|$ of the varifold $V_t$. 
The symbol $h$ denotes the generalized mean curvature vector of $V_t$,
and $\nu_i$ is the outer unit normal vector field to the reduced boundary 
$\partial^* E_i(t)$ of $E_i(t)$. Since $\Gamma(t)=\cup_{i=1}^N \partial E_i(t)$ for all $t>0$ (see Theorem \ref{main3}(3)), 
the claim (a) shows that the support of the
Brakke flow coincides -- up to a lower dimensional set -- with the union of the topological boundaries of the grains for a.e.~$t>0$. The claim (b) states that each reduced boundary $\partial^* E_i (t)$ is a solution to the MCF in the integral sense specified in \eqref{vchange}: that is, the generalized velocity of $\partial^* E_i (t)$ -- defined as the distributional time derivative of the indicator function of $E_i (t)$ -- is precisely $h \cdot \nu_i\, \Ha^n \mres_{\partial^* E_i (t)}$.

When integrated, formula \eqref{vchange} provides, as a byproduct, the change of the $(n+1)$-dimensional volume of each grain $E_i(t)$ in any bounded open set $U$:
\begin{equation}\label{vchange-s}
    \mathcal L^{n+1}(U\cap E_i(t_2))-\mathcal L^{n+1}(U\cap E_i(t_1))=\int_{t_1}^{t_2}\int_{U\cap\partial^*
    E_i(t)}h\cdot\nu_i\,
    d\mathcal H^n dt\,.
\end{equation}
We emphasize the following point in particular: the formulae \eqref{vchange} and \eqref{vchange-s}
hold true even if there is a ``higher
multiplicity portion'' of $\|V_t\|$ on $\partial^* E_i(t)$, or some ``interior boundary'' $\partial E_i(t)\setminus\partial^* E_i(t)$. As far as the authors are aware of, for generalized MCF,
\eqref{vchange-s} had never been established prior to the present paper, even for the two-phase MCF. It is worth mentioning that, in the latter context, the existence of a square-integrable generalized velocity advecting the common boundary of the two phases had been proved before, in the setting of sharp interface limits of minimizers of the Allen-Cahn action functional in \cite{Mugnai}, and of solutions to the Allen-Cahn equation in \cite{Hensel-Laux_AC}. A fundamental new achievement of the present paper is, besides the multi-phase setting, the fact that the generalized velocity advecting each phase is precisely the dot product of the generalized mean curvature of the underlying varifold with the appropriate unit normal. Even though we have $\|V_0\|=\mathcal H^n\mres_{\Gamma_0}$, and thus the unit density condition is satisfied at the initial time of the Brakke flow, it is not possible to exclude, in the very general framework under consideration, the occurrence of
higher multiplicities at a later time. Despite all 
those possible singular behaviors, \eqref{vchange}
and \eqref{vchange-s} are guaranteed time-globally.
The claim (c) states that, if the higher 
multiplicity of 
$\|V_t\|$ does not occur for $\mathcal H^n$-a.e.~$x$ and
a.e.~$t$ locally in space-time, then the reduced
boundary measure and the Brakke flow may be
identified with one another
in that region, and 
we may say that the $N$-tuple $\chi=\left(\chi_{E_1} , \ldots , \chi_{E_N} \right)$ is a BV solution to MCF (see Definition \ref{d:BV flow}).
We can guarantee the existence of a unit density flow for some initial time interval $[0,T_0)$ if we additionally 
assume a suitable density ratio upper bound on $\Gamma_0$
(see Theorem \ref{main5}). Such assumption would still allow for an initial datum $\Gamma_0$ which consists of a union of
Lipschitz curves joined by triple junctions in $n=1$ and 
Lipschitz bubble clusters with tetrahederal singularities in $n=2$,
for example.

\smallskip

For general Brakke flows, 
there is no clear pathway leading from the characterization of Brakke flow, which 
consists of a variational inequality dictating an upper bound on the rate of change of the mass of the evolving surfaces, 
to formula \eqref{vchange}, even under 
the unit density assumption. In fact, as mentioned already, 
by the partial regularity theorem for unit density
Brakke flows \cite{Brakke,Kasai-Tone,Ton-2}, in this latter case it is known that $\Gamma(t)$ is a $C^\infty$ MCF in a space-time neighborhood of $(x_0,t_0)$ for a.e.~$t_0$ and for $\Ha^n$-a.e.~$x_0 \in \Gamma(t_0)$: nonetheless, this alone is not sufficient to guarantee that $\chi=(\chi_{E_1},\ldots,\chi_{E_N})$ is a BV solution. The same
remark goes for the opposite implication, i.e., from BV solution to 
Brakke flow. Since there is no known partial regularity theory for general BV solutions, these two notions appear far from being 
equivalent in any case. 
In the present paper, instead, the Brakke flow arises as the limit of a suitable time-discrete approximation scheme, analogous to that introduced by Kim and the second-named author in \cite{KimTone}, and we prove \eqref{vchange} by showing that an analogous identity holds approximately true for the approximating flows, with vanishing errors in the limit. In order to gain enough control on the change of volumes in the approximation scheme and consequently obtain good estimates on the error terms, we will need to implement an appropriate modification to 
the construction of the time-discrete
approximate flows devised in \cite{KimTone}:
the details of such modification will 
be explained thoroughly later; see Section \ref{s:Lipschitz} and Appendix \ref{appendix:Lipschitz}. 

\smallskip

We next discuss closely related works, particularly on the
aspect of existence of generalized MCF. For two-phase 
MCF, the level-set method \cite{CGG,ES} provides a 
general existence and uniqueness result even past the time after singularities appear.
On the other hand, the level-set may develop a
non-trivial interior, a phenomenon called ``fattening'',
due to the singular behavior of the MCF.
Also the uniqueness 
of the level-set solutions depends essentially 
on the maximum principle and it cannot handle general
multi-phase MCF of more than two phases.

For the general multi-phase problem, 
it is natural to consider an initial datum $\Gamma_0$
with singularities to start with.  For example, in dimension $n=1$, a typical $\Gamma_0$ in a three-phase problem is a union of
curves meeting at triple junctions. In the parametric 
setting, Bronsard and Reitich \cite{Bronsard} first showed 
the short-time existence of a unique solution for $C^{2,\alpha}$ initial datum.
Since then, there have been numerous studies (mostly for $n=1$ but also for higher dimensions \cite{Freire,Depner}), and we refer the reader to
the survey \cite{Mantegazza} for the references on the parametric approaches. 
Due to the nature of the solutions and the need to heavily employ PDE techniques, these existence results do not extend beyond the time of 
topological changes. With a non-parametric approach and for the existence of MCF with regular triple junctions, 
one can
adopt the elliptic regularization \cite{Ilm1} for the
class of flat chains
with coefficients in a finite group, see \cite{SW}.

 Luckhaus and Sturzenhecker \cite{Luckhaus} introduced
the formulation of BV solution of the two-phase MCF, which can be
extended naturally to the multi-phase MCF. Their existence result is 
conditional, 
in the sense that a BV solution is shown to exist under the assumption that 
the time-discrete approximate solutions converge to their limit
without loss of surface energy. Laux and Otto \cite{Laux1} proved that 
a sequence arising from the thresholding scheme of Merriman, Bence and Osher converges conditionally to a BV solution using the interpretation in terms of minimizing movements due to Esedo\={g}lu and Otto \cite{Esedoglu}, again under an 
assumption similar to \cite{Luckhaus} (see also 
\cite{Laux-Simon} for a similar convergence result of
the parabolic Allen-Cahn system). The BV solution of \cite{Luckhaus} was partly motivated by the minimizing movements scheme of Almgren-Taylor-Wang \cite{ATW}, 
and the multi-phase version has been studied recently by Bellettini and 
Kholmatov \cite{Bellettini}. 

On the side of Brakke flows, Ilmanen \cite{Ilm_AC}
proved the existence of a rectifiable Brakke flow arising as a limit 
of solutions to a parabolic Allen-Cahn equation, and the second-named author proved the
integrality of the Brakke flow \cite{Ton3}. These results are for two-phase
MCF, but the relation to the BV solution as formulated in \cite{Luckhaus}
remained obscure. In a different but 
related problem -- the Allen-Cahn action functional --, Mugnai and R\"{o}ger
\cite{Mugnai} introduced a notion of $L^2$ flow to describe a weak 
formulation of MCF with additional $L^2$ forcing term for
$n=1,\,2$. The existence 
of the $L^2$ flow depends on the result of R\"{o}ger and Sch\"{a}tzle \cite{R-S} which solved
one of De Giorgi's conjectures. The work \cite{Mugnai}
essentially contains the result that the limit phase boundary of the parabolic Allen-Cahn equation satisfies an analogous equation to \eqref{vchange} (see \cite[Proposition 4.5]{Mugnai}). The solution constructed in the present paper is, in fact, also an $L^2$ flow in the sense of Mugnai-R\"oger, and, even though the approach leading to \eqref{vchange} is different, some properties of generalized velocities of $L^2$ flows established in \cite{Mugnai} are used in the present paper as well.

Finally, we mention again that a global-in-time  existence
result of a multi-phase Brakke flow which is equipped with 
moving grain boundaries was given by Kim and the second-named author in \cite{KimTone}, reworking the pioneering paper by 
Brakke \cite{Brakke} within a different formulation. The grains in \cite{KimTone}
move continuously with respect to the Lebesgue measure, but the problem concerning the validity of an exact identity involving the volume change was not addressed in there. Previous works by the authors of the present paper (see \cite{ST19,stu-tone2}), in which certain Brakke flows are constructed with an approximation scheme analogous to that introduced in \cite{KimTone}, could be reworked so that the additional conclusions concerning the interplay between the flow of the grains and the Brakke flow can be drawn in those contexts as well: in particular, it is possible to have the Brakke flow with prescribed boundary constructed in \cite{ST19} satisfy the formulae \eqref{vchange} and \eqref{vchange-s} (see Section \ref{fixed}).

\medskip

\noindent\textbf{Acknowledgments.} S.S. is partially supported by the \textit{GNAMPA - Gruppo Nazionale per l'Analisi Matematica, la Probabilità e le loro Applicazioni} of INdAM; Y.T. is partially supported by JSPS 18H03670, 19H00639, 17H01092.

\section{Notation and main results} \label{s:notation}
\subsection{Basic notation}
We shall use the same notation adopted in \cite[Section 2]{KimTone}. In particular, the ambient space we will be working in is the Euclidean space $\R^{n+1}$, and $\R^+$ will denote the interval $\left[0,\infty\right)$. Coordinates $(x,t)$ are set in the product space $\R^{n+1} \times \R^+$, and $t$ will be thought of and referred to as ``time''. For a subset $A$ of Euclidean space, ${\rm clos}\,A$, ${\rm int}\,A$, $\pa A$ and ${\rm conv}\,A$ will denote the closure, interior, boundary and convex hull of $A$, respectively. If $A \subset \R^{n+1}$ is (Borel) measurable, $\Leb^{n+1} (A)$ or $|A|$ will denote the Lebesgue measure of $A$, whereas $\Ha^k (A)$ denotes the $k$-dimensional Hausdorff measure of $A$. When $x \in \R^{n+1}$ and $r>0$, $U_r(x)$ and $B_r(x)$ denote the open ball and the closed ball centered at $x$ with radius $r$, respectively. More generally, if $k$ is an integer then $U^k_r(x)$ and $B^k_r(x)$ will denote open and closed balls in $\R^k$, and $\omega_k := \Leb^k (U^k_1(0))$.\\

A positive Radon measure $\mu$ on $\mathbb R^{n+1}$ (or in ``space-time'' $\R^{n+1} \times \R^+$) is always also regarded as a positive linear functional on the space $C_c(\R^{n+1})$ of continuous and compactly supported functions on $\R^{n+1}$, with the pairing denoted $\mu(\phi)$ for $\phi \in C_c(\R^{n+1})$. The restriction of $\mu$ to a Borel set $A$ is denoted $\mu\, \mres_A$, so that $(\mu \,\mres_A)(E) := \mu(A \cap E)$ for any $E \subset \R^{n+1}$. The support of $\mu$ is denoted $\spt\,\mu$, and it is the closed set defined by
\[
\spt\,\mu := \left\lbrace x \in \R^{n+1} \, \colon \, \mu(B_r(x)) > 0 \mbox{ for every $r > 0$} \right\rbrace\,.
\]
The upper and lower $k$-dimensional densities of a Radon measure $\mu$ at $x \in \R^{n+1}$ are
\[
\Theta^{*k}(\mu,x) := \limsup_{r \to 0^+} \frac{\mu(B_r(x))}{\omega_k\, r^k} \,, \qquad \Theta^k_*(\mu,x) := \liminf_{r \to 0^+} \frac{\mu(B_r(x))}{\omega_k\, r^k}\,,
\]
respectively. If $\Theta^{*k}(\mu,x) = \Theta^k_*(\mu,x)$ then the common value is denoted $\Theta^k(\mu,x)$, and is called the $k$-dimensional density of $\mu$ at $x$. For $1 \leq p \leq \infty$, the space of $p$-integrable (resp.~locally $p$-integrable) functions with respect to $\mu$ is denoted $L^p(\mu)$ (resp.~$L^p_{{\rm loc}}(\mu)$). If $U\subset \R^{n+1}$ is an open set, $L^p(\Leb^{n+1}\mres_U)$ and $L^p_{{\rm loc}}(\Leb^{n+1}\mres_U)$ are simply written $L^p(U)$ and $L^p_{{\rm loc}}(U)$. For a signed or vector-valued measure $\mu$, $|\mu|$ denotes its total variation. \\

Given an open set $U \subset \R^{n+1}$, we say that a function $f \in L^1(U)$ has bounded variation in $U$, written $f \in {\rm BV}(U)$, if 
\[
\sup\left\lbrace \int_U f \, {\rm div}\,g \, dx \, \colon \, g \in C^1_c(U;\R^{n+1}) \mbox{ with } \|g\|_{C^0} \leq 1 \right\rbrace < \infty\,.
\]
If $f\in {\rm BV}(U)$ then there exists an $\R^{n+1}$-valued Radon measure on $U$, which we will call the measure derivative of $f$ and denoted $\nabla f$, such that
\[
\int_U f\,{\rm div}\,g\, dx = - \int_U g \cdot d\nabla f \qquad \mbox{for all $g \in C^1_c(U;\R^{n+1})$}\,.
\]
We say that $f \in {\rm BV}_{{\rm loc}}(U)$ if $f \in {\rm BV}(U')$ for all $U' \Subset U$.\\

For a set $E \subset \R^{n+1}$, $\chi_E$ is the characteristic (or indicator) function of $E$, defined by $\chi_E(x)=1$ if $x\in E$ and $\chi_E(x)=0$ otherwise. We say that $E$ has locally finite perimeter in $\R^{n+1}$ if $\chi_E \in {\rm BV}_{{\rm loc}}(\R^{n+1})$. When $E$ is a set of locally finite perimeter, then the measure derivative $\nabla \chi_E$ is the associated Gauss-Green measure, and its total variation $|\nabla \chi_E|$ is the perimeter measure; by De Giorgi's structure theorem, $| \nabla \chi_E| = \Ha^n \mres_{\pa^* E}$, where $\pa^* E$ is the reduced boundary of $E$, and $\nabla\chi_E = - \nu_E\, |\nabla \chi_E| = - \nu_E\, \Ha^n \mres_{\pa^* E}$, where $\nu_E$ is the outer pointing unit normal vector field to $\pa^*E$.\\

A subset $\Gamma \subset \R^{n+1}$ is countably $k$-rectifiable if it is $\Ha^k$-measurable and it admits a covering
\[
\Gamma \subset Z \cup \bigcup_{h \in \mathbb N} f_h (\R^k)\,,
\]
where $\Ha^k(Z)=0$ and $f_h \colon \R^k \to \R^{n+1}$ is Lipschitz. A countably $k$-rectifiable set $\Gamma$ is (locally) $\Ha^k$-rectifiable if, moreover, $\Ha^k(\Gamma)$ is (locally) finite. Countably $k$-rectifiable sets $\Gamma$ are characterized by the existence of approximate tangent planes $\Ha^k$-almost everywhere (\cite[Theorem 11.6]{Simon}). In other words, $\Gamma$ is countably $k$-rectifiable if and only if there exists a \emph{positive} function $g \in L^1_{{\rm loc}}(\Ha^k \mres_{\Gamma})$ such that the following holds: for $\Ha^k$-a.e. $x \in \Gamma$ there exists a $k$-dimensional linear subspace of $\R^{n+1}$, denoted $T_x\Gamma$ or ${\rm Tan}(\Gamma, x)$ and referred to as the (approximate) tangent plane to $\Gamma$ at $x$ such that
\[
g(x+r\,\cdot) \, \Ha^k \mres_{\frac{\Gamma - \{x\}}{r}} \weakstar g(x) \, \Ha^k \mres_{T_x\Gamma} \qquad \mbox{in the sense of measures, as $r \to 0^+$}\,.
\]
A (positive) measure $\mu$ on $\R^{n+1}$ is said to be $k$-rectifiable if there are a countably $k$-rectifiable set $\Gamma$ and a positive function $g\in L^1_{{\rm loc}}(\Ha^k \mres_{\Gamma})$ such that $\mu=g\,\Ha^k \mres_{\Gamma}$. If $\mu$ is $k$-rectifiable, $\mu = g\,\Ha^k \mres_{\Gamma}$, then for $\mu$-a.e. $x$ the measures $\mu_{x,r}$ defined, for $r > 0$, by $\mu_{x,r}(A) := r^{-k}\,\mu(x+r\,A)$ satisfy
\[
\mu_{x,r} \weakstar g(x)\,\Ha^k \mres_{T_x\Gamma} \qquad \mbox{as $r \to 0^+$}\,,
\]
that is $g(x)\,\Ha^k \mres_{T_x\Gamma}$ is the tangent measure of $\mu$ at $x$. In this case the notation $T_x\mu$ may be used interchangeably with $T_x\Gamma$ to denote the approximate tangent plane to the $k$-rectifiable measure $\mu$ at $x$ for $\mu$-a.e.\,$x$. More generally, given a Radon measure $\mu$ on $\R^{n+1}$ and $x \in \R^{n+1}$, we say that $\mu$ has an approximate tangent $k$-plane at $x$ if there exist $g(x) \in \left(0, \infty\right)$ and a $k$-dimensional linear subspace $\pi \subset \R^{n+1}$ such that
\[
\mu_{x,r} \weakstar g(x)\,\Ha^k \mres_{\pi} \qquad \mbox{as $r \to 0^+$}\,.
\]
When this happens, the plane $\pi$ is unique, and it is denoted $T_x\mu$.\\

A $k$-dimensional varifold in $\R^{n+1}$ (\cite{Allard,Simon})
is a positive Radon measure $V$ on the space $\R^{n+1} \times {\bf G}(n+1,k)$, where ${\bf G}(n+1,k)$ is the Grassmannian of (unoriented) $k$-dimensional linear subspaces of $\R^{n+1}$. If $V$ is a $k$-varifold in $\R^{n+1}$, we write $V \in \mathbf{V}_k(\R^{n+1})$, and we let $\|V\|$ and $\delta V$ denote its weight and first variation, respectively. When $\delta V$ is locally bounded and absolutely continuous with respect to $\|V\|$, we let $h(\cdot,V) \in L^1_{{\rm loc}}(\|V\|;\R^{n+1})$ denote the generalized mean curvature vector of $V$, so that $\delta V = - h(\cdot, V) \,\|V\|$ in the sense of $\R^{n+1}$-valued measures on $\R^{n+1}$. If $\Gamma$ is countably $k$-rectifiable, and $\theta \in L^1_{{\rm loc}}(\Ha^k \mres_{\Gamma})$ is positive and integer-valued, we let $\var (\Gamma, \theta)$ denote the varifold $\var (\Gamma,\theta) = \theta\,\Ha^k\mres_{\Gamma} \otimes \delta_{T_{\cdot}\Gamma}$. When $V$ admits a representation $V=\var (\Gamma,\theta)$ as above, we say that $V$ is an integral $k$-varifold, and we write $V \in \mathbf{IV}_k(\R^{n+1})$. All above notions concerning measures (and varifolds) in Euclidean space $\R^{n+1}$ can be immediately localized to open sets $U \subset \R^{n+1}$.

\subsection{Three weak notions of MCF}

As anticipated in Section \ref{s:introduction}, in the last few decades several alternative notions of weak solution to the MCF have been proposed. In this subsection we briefly define and comment upon the three of interest in the present paper: Brakke flows, $L^2$ flows and ${\rm BV}$ flows. We begin with the notion of Brakke flow, introduced by Brakke in \cite{Brakke}.

\begin{definition}[Brakke flow]\label{d:Brakke flow}
Let $0 < T \leq \infty$, and let $U \subset \R^{n+1}$ be an open set. A \emph{$k$-dimensional (integral) Brakke flow} in $U$ is a one-parameter family of varifolds $\{V_t\}_{t\in\left[0,T\right)}$ in $U$ such that all of the following hold: 
\begin{enumerate}
\item[(a)]
For a.e.~$t\in \left[0,T\right)$, $V_t\in{\bf IV}_k(U)$;
\item[(b)] For a.e.~$t\in \left[0,T\right)$, $\delta V_t$ is locally bounded and absolutely continuous with respect to $\|V_t\|$;
\item[(c)] The generalized mean curvature $h(\cdot,V_t)$ (which exists for a.e.~$t$ by (b)) satisfies $h(\cdot,V_t) \in L^2_{{\rm loc}}(\|V_t\|;\R^{n+1})$, and for every compact set $K \subset U$ and for every $t < T$ it holds $\sup_{s \in \left[0,t\right]} \|V_s\| (K) < \infty$;
\item[(d)]
For all $0\leq t_1<t_2<T$ and $\phi\in C_c^1(U\times\left[0,T\right);\mathbb R^+)$, it holds
\begin{equation}
\label{brakineq}
\begin{split}
&\|V_{t_2}\|(\phi(\cdot,t_2)) - \|V_{t_1}\|(\phi(\cdot,t_1)) \\ 
&\qquad\leq \int_{t_1}^{t_2} \int_{U} \left\lbrace - \phi (x,t) \,\abs{h(x,V_t)}^2 + \nabla\phi(x,t) \cdot h(x,V_t) + \frac{\partial\phi}{\partial t}(x,t) \right\rbrace    d\|V_t\|(x)\,dt\,.
\end{split}
\end{equation}
\end{enumerate}
\end{definition}

The inequality in \eqref{brakineq} is typically referred to as \emph{Brakke's inequality}. It is not difficult to show that if $\{\Gamma(t)\}_{t\in \left[0,T\right)}$ is a flow of smooth submanifolds with mean curvature $h(x,t)=h(x,\Gamma(t))$ and normal velocity $v(x,t)$ then one has, for any $0\leq t_1 < t_2 < T$ and for every $\phi \in C^1_c(U \times \left[t_1,t_2\right])$, the identity 
\begin{equation} \label{smooth flow}
    \left.\int_{\Gamma(t)} \phi(x,t) \,d\Ha^k\right|_{t=t_1}^{t_2}
    = \int_{t_1}^{t_2} \int_{\Gamma(t)} \left\lbrace  - \phi\, v \cdot h + \nabla \phi \cdot v + \frac{\pa\phi}{\pa t} \right\rbrace \, d\Ha^k\,dt\,.
\end{equation}
In particular, if $\{\Gamma(t)\}_{t\in[0,T)}$ is a smooth MCF then setting $V_t := \var(\Gamma(t), 1)$ defines a Brakke flow for which \eqref{brakineq} is satisfied with an equality. Conversely, if $V_t = \var (\Gamma(t),1)$ with $\Gamma(t)$ smooth, then the inequality \eqref{brakineq} is sufficient to guarantee that $\{\Gamma(t)\}_{t\in[0,T)}$ is a classical solution to the MCF. For further details on the definition, the reader can consult the original work by Brakke in \cite{Brakke} or the more recent monograph \cite{Ton1}. Next, we define the notion of unit density flow. 
\begin{definition}
A Brakke flow is said to be a \emph{unit density flow} 
in $U\times [t_1,t_2]$ if
$\Theta^k(\|V_t\|,x)=1$ for $\|V_t\|$-a.e.~$x\in U$ and
$\mathcal L^1$-a.e.~$t\in [t_1,t_2]$. If $U=\mathbb R^{n+1}$, we may simply say that $\{V_t\}_{t\in[t_1,t_2]}$ is a unit density Brakke flow.
\end{definition}

The following definition of $L^2$ flow has been given by Mugnai and R\"oger in \cite{Mugnai}.

\begin{definition}[$L^2$ flow] \label{d:L2-flow}
Let $0 < T < \infty$, and let $U \subset \R^{n+1}$ be an open and bounded set. A one-parameter family $\{V_t\}_{t \in \left[0,T\right)}$ of varifolds in $U$ is a \emph{$k$-dimensional $L^2$ flow} if it satisfies (a)-(b) in Definition \ref{d:Brakke flow} as well as the following:
\begin{itemize}
    \item[(c')] The generalized mean curvature $h (\cdot, V_t)$ (which exists for a.e. $t \in \left[0,T\right)$ by (b)) satisfies $h(\cdot,V_t) \in L^2(\|V_t\|;\R^{n+1})$, and $d\mu:=d\|V_t\|dt$ is a Radon measure on $U \times \left( 0, T \right)$;
    \item[(d')] There exist a vector field $v \in L^2(\mu;\R^{n+1})$ and a positive constant $C$ such that
    \begin{itemize}
        \item[(d'1)] $v(x,t)\,\perp\,T_x\|V_t\|$ for $\mu$-a.e. $(x,t) \in U \times \left( 0, T \right)$;
        \item[(d'2)] For every $\phi \in C^1_c(U \times \left( 0,T \right))$ it holds
        \begin{equation} \label{e:L^2-flow def}
            \left| \int_0^T \int_{U} \frac{\partial\phi}{\partial t}(x,t) +  \nabla\phi(x,t) \cdot v(x,t) \, d\|V_t\|(x) \, dt \right| \leq C \, \|\phi\|_{C^0}\,.
        \end{equation}
    \end{itemize}
\end{itemize}
Any function $v\in L^2(\mu;\R^{n+1})$ satisfying (d') above is called a \emph{generalized velocity vector} for the flow. 
\end{definition}

The definition can be easily motivated by considering \eqref{smooth flow} once again: the latter, indeed, implies \eqref{e:L^2-flow def} if both $h$ and $v$ are in $L^2(\mu;\R^{n+1})$. $L^2$ flows can then be described as flows of generalized surfaces with generalized mean curvature and velocity vectors in $L^2$. 

\medskip

Finally, the following definition of ${\rm BV}$ flow is related to a motion of \emph{hypersurfaces} which are the boundaries of a finite family of sets of locally finite perimeter; see \cite{Luckhaus,Laux1}. Given $N \ge 2$, we say that $\{E_1,\ldots,E_N\}$ is an $\Leb^{n+1}$-partition of $\R^{n+1}$ if $E_i \subset \R^{n+1}$ for every $i$, they are pairwise disjoint, and $\Leb^{n+1}(\R^{n+1} \setminus \bigcup_{i=1}^N E_i) = 0$.

\begin{definition}[${\rm BV}$ flow] \label{d:BV flow}
Let $N \ge 2$ be an integer, and let $0 < T < \infty$. $N$ one-parameter families $\{E_i (t)\}_{t\in\left[0,T\right)}$ ($i=1,\ldots,N$) identify a ${\rm BV}$ solution for multi-phase MCF in $\R^{n+1}$ if all of the following hold:
\begin{itemize}
    \item[(a'')] For a.e.~$t \in \left[0,T\right)$, $\{E_1(t),\ldots,E_N(t)\}$ is an $\Leb^{n+1}$-partition of $\R^{n+1}$, $E_i(t)$ is a set of locally finite perimeter, and, setting $I_{i,j}(t):=\pa^*E_i (t)\cap \pa^*E_j(t)$ for $i \neq j$, 
    \begin{equation} \label{BV:finite energy}
        \esssup_{t \in \left[0,T\right)} \sum_{i,j=1,\,i \neq j}^N \Ha^n (I_{i,j}(t)) < \infty\,;
    \end{equation}
    \item[(b'')] There exist scalar functions $v_1,\ldots,v_N$ such that
    \begin{equation} \label{scalar velocities are in L2}
        \int_0^T \int_{\pa^*E_i(t)} |v_i(x,t)|^2\,d\Ha^n(x)\,dt < \infty \qquad \mbox{for every $i$}\,,
    \end{equation}
    and with the property that
    \begin{equation} \label{vi are velocities}
        \begin{split}
 \left.\int_{E_i(t)} \phi(x,t) \, dx\right|_{t=t_1}^{t_2} = & \int_{t_1}^{t_2} \int_{E_i(t)} \frac{\partial\phi}{\partial t}(x,t) \, dx\,dt  \\
& + \int_{t_1}^{t_2} \int_{\pa^*E_i(t)} \phi(x,t) \, v_i(x,t)\,d\Ha^n(x)\,dt 
\end{split}
    \end{equation}
    for a.e.~$0 \le t_1 < t_2 < T$ and for all $\phi \in C^1_c(\R^{n+1} \times \left[0,T\right))$;
    \item[(c'')] Setting $\nu_i(x,t) = \nu_{E_i(t)}(x)$ for the outer unit normal to the reduced boundary of $E_i(t)$ at $x$, it holds 
    \begin{equation} \label{reflection condition}
        v_i(\cdot,t)\,\nu_i (\cdot,t) = v_j(\cdot, t)\,\nu_j(\cdot,t) \qquad \mbox{$\Ha^n$-a.e.~on $I_{i,j}(t)$, for a.e. $0 \leq t < T$}\,;
    \end{equation}
    \item[(d'')] The functions $v_i$ further satisfy
    \begin{equation} \label{BV formulation of mean curvature}
        \sum_{i\neq j} \int_0^T \int_{I_{i,j}(t)} {\rm div}\,g - (\nu_i \otimes \nu_i) \cdot \nabla g\,d\Ha^n\,dt = - \sum_{i \neq j} \int_0^T \int_{I_{i,j}(t)} v_i \, \nu_i \cdot g \,  d\Ha^n\,dt
    \end{equation}
    for all vector fields $g \in C^1_c(\R^{n+1}\times \left[0,T\right];\R^{n+1})$;
\item[(e)] The following inequality holds for 
a.e.~$0\leq t< T$:
\begin{equation}\label{d:e_diss}
    \sum_{i,j=1,\,i \neq j}^N \Ha^n(I_{i,j}(t)) + \sum_{i,j=1,\,i \neq j}^N \int_0^t\int_{I_{i,j}(s)} |v_i(x,s)|^2\,d\Ha^n(x)\,ds \leq \sum_{i,j=1,\,i \neq j}^N \Ha^n(I_{i,j}(0))\,.
\end{equation}
\end{itemize}
\end{definition}
The equality \eqref{vi are velocities} characterizes $v_i(\cdot, t)$ as the velocity of 
$\partial^* E_i(t)$, as anticipated in Section \ref{s:introduction}. 
By \cite[Proposition 29.4]{Maggi_book}, 
for each $i=1,\ldots, N$, we have
\begin{equation}\label{chap1}
 \mathcal H^n\Big(\partial^* E_i(t)
 \setminus \cup_{j=1,j\neq i}^N I_{i,j}(t)\Big)=0\,\mbox{ and }\,
 \mathcal H^n(I_{i,j}(t)\cap I_{i,j'}(t))=0
 \,\mbox{ for } j\neq j',
\end{equation}
thus \eqref{BV formulation of mean curvature}
formally characterizes $v_i\nu_i$ as the
mean curvature vector of $\cup_{i=1}^N
\partial^* E_i(t)$ on $I_{i,j}(t)$. Since
\begin{equation}\label{chap2}
\mathcal H^n(\cup_{i=1}^N\partial^* E_i(t))
 =\frac12\sum_{i,j=1,i\neq j}^N\mathcal H^n(I_{i,j}(t)),
\end{equation}
(e) corresponds precisely to the energy dissipation 
inequality for the hypersurface measures
of the MCF. 
\subsection{Main results} \label{Mainresult}

\begin{definition} \label{def:weight}
We will denote by $\Omega$ a function in $C^2(\R^{n+1})$ such that
\begin{equation}
    0<\Omega(x)\leq 1\,, \quad
    |\nabla\Omega(x)|\leq c_1\Omega(x)\,,\quad \|\nabla^2\Omega(x)\|
    \leq c_1\Omega(x)
\end{equation}
for all $x\in \mathbb R^{n+1}$, where $c_1 \geq 0$ is a constant and $\|\nabla^2\Omega(x)\|$ is the Hilbert-Schmidt norm of the Hessian matrix $\nabla^2\Omega(x)$.
\end{definition}
The function $\Omega$ is introduced as a weight function to treat unbounded MCF, which may have infinite $n$-dimensional Hausdorff measure in $\R^{n+1}$. A typical choice of 
$\Omega$ in this case can be $\Omega (x) = \exp(-\sqrt{1+|x|^2})$ with a suitable choice of
$c_1$. If one is interested in MCF with finite measure, one can choose $\Omega \equiv 1$, with $c_1=0$ in this case. With a function $\Omega$ as specified above, we consider an initial datum $\Gamma_0$ satisfying the following:
\begin{assumption} 
Suppose that $\Gamma_0\subset\mathbb R^{n+1}$ is a closed, countably $n$-rectifiable
set such that
\begin{equation}\label{ofini}
    \mathcal H^n\mres_\Omega(\Gamma_0):=\int_{\Gamma_0}\Omega(x)\,d\mathcal H^n(x)<\infty\,.
\end{equation}
Moreover, assume that there are $N \geq 2$ mutually disjoint non-empty 
open sets $E_{0,1},\ldots,E_{0,N}$ 
such that $\Gamma_0 = \mathbb R^{n+1}\setminus \cup_{i=1}^N E_{0,i}$. In particular, $\{E_{0,1},\ldots,E_{0,N}\}$ is an $\Leb^{n+1}$-partition of $\R^{n+1}$.
\label{as1}
\end{assumption}
By \eqref{ofini} and $\Omega>0$, $\Gamma_0$ has locally finite
$\mathcal H^n$-measure, and thus $\Gamma_0$ has no interior points. In particular, we have 
$$\Gamma_0=\cup_{i=1}^N
\partial E_{0,i}.$$
For a given $\Gamma_0$ as in Assumption \ref{as1}, the assignment of the partition $\R^{n+1} \setminus \Gamma_0 = \bigcup_{i=1}^N E_{0,i}$ is certainly 
non-unique. Each $E_{0,i}$ may not be connected, for example, and the different choice may result in different MCF. In general, 
$E_{0,i}$ has locally finite perimeter, and $\partial^* E_{0,i}\subset {\rm spt}\,|\nabla\chi_{E_{0,i}}|\subset\partial E_{0,i}$ for each $i=1,\ldots,N$. 

\medskip

The following Theorems \ref{main1}-\ref{main4} all hold under Assumption
\ref{as1}. First, we claim the existence of a Brakke flow starting with $\Gamma_0$ which is also an $L^2$ flow in any bounded domain of $\R^{n+1}$, and whose generalized velocity vector coincides with the generalized mean curvature vector of the evolving varifolds.
\begin{theorem}\label{main1}
There exists an $n$-dimensional Brakke flow $\{V_t\}_{t\geq 0}$ in $\R^{n+1}$ such that
\begin{itemize}
    \item[(1)] $\|V_0\|=\mathcal H^n\mres_{\Gamma_0}$,
    \item[(2)] if $\mathcal H^n(\cup_{i=1}^N(\partial E_{0,i}\setminus\partial^*
    E_{0,i}))=0$, then $\lim_{t\rightarrow 0+}\|V_t\|=\|V_0\|$,
  \item[(3)] $\|V_t\|(\Omega)\leq \mathcal H^n\mres_{\Omega}(\Gamma_0)\,
  \exp(c_1^2 t/2)$ and $\int_0^t\int_{\mathbb R^{n+1}}|h(\cdot,
  V_s)|^2\,\Omega\,d\|V_s\|ds<\infty$ for all $t> 0$.
  \item[(4)] If $\Ha^n (\Gamma_0) < \infty$, and thus if one can choose $\Omega \equiv 1$ in Assumption \ref{as1}, then 
  $\|V_t\|(\mathbb R^{n+1})+\int_0^t\int_{\mathbb R^{n+1}} |h(\cdot,V_s)|^2\,d\|V_s\|ds\leq \mathcal H^n(\Gamma_0)$ for all $t>0$.
  \item[(5)] For any $0 < T < \infty$ and for any open and bounded set $U \subset \R^{n+1}$, the one-parameter family $\{V_t \,\mres_{U}\}_{t \in \left[0,T\right)}$ (where we regard $V_t\, \mres_U$ as a varifold in $U$) is an $n$-dimensional $L^2$ flow with generalized velocity vector $v(x,t)=h(x,V_t)$ on $U\times\left(0,T\right)$.
\end{itemize}
\end{theorem}
\begin{remark}
If $n=1$, the above Brakke flow satisfies the additional regularity 
properties obtained in \cite{KimTone2}: 
for $\mathcal L^1$-a.e.~$t>0$, ${\rm spt}\|V_t\|$ is locally
the union of a finite number of $W^{2,2}$ curves meeting at junctions  with angles of either $0$, $60$, or $120$ degrees
for $N\geq 3$, and only $0$ degree (no transverse crossing) for $N=2$. See \cite[Theorem 2.2, 2.3]{KimTone2} for the further details.
\end{remark}
\begin{definition}
With $\{V_t\}_{t\geq 0}$ as in Theorem \ref{main1}, let $\mu$ be the Radon measure on $\R^{n+1} \times \R^+$ given by $d\mu=d\|V_t\|dt$, and for $t\in\mathbb R^+$, define
the closed set
$$({\rm spt}\,\mu)_t:=\{x\in\R^{n+1}\,:\,(x,t)\in {\rm spt}\,\mu\}\,.$$
\end{definition}
The following Theorem \ref{main2} is satisfied in general for
Brakke flows, and thus it holds, in particular, for the flow produced in Theorem \ref{main1}. 
\begin{theorem}\label{main2}
The Radon measure $\mu$ and the Brakke flow $\{V_t\}_{t\geq 0}$ are related as follows. 
\begin{itemize}
    \item[(1)] ${\rm spt}\|V_t\|\subset({\rm spt}\,\mu)_t$ and 
    $\mathcal H^n(B_r\cap({\rm spt}\,\mu)_t)<\infty$ for all $t>0$ and $r>0$,
    \item[(2)] $\mathcal H^{n-1+\delta}(({\rm spt}\,\mu)_t\setminus {\rm spt}\|V_t\|)=0$
    for every $\delta>0$ and for $\mathcal L^1$-a.e.~$t\in\mathbb R^+$,
    \item[(3)] ${\rm spt}\|V_t\|$ is countably $n$-rectifiable
    for $\mathcal L^1$-a.e.~$t\in\mathbb R^+$,
    \item[(4)] $V_t={\bf var}({\rm spt}\|V_t\|,\theta_t)$ with $\theta_t(x)=
    \Theta^n(\|V_t\|,x)$ ($\|V_t\|$-a.e.~$x$)
    for $\mathcal L^1$-a.e.~$t\in\mathbb R^+$.
\end{itemize}
\end{theorem}

The following theorem shows that, in addition to the Brakke flow of Theorem \ref{main1}, 
there are evolving domains $\{E_i(t)\}_{t\geq 0}$ 
starting from $E_{0,i}$ defining a (generalized) ${\rm BV}$ solution for multi-phase MCF in $\R^{n+1}$. 
\begin{theorem}\label{main3}
For each $i=1,\ldots,N$, there exists a family of 
open sets $\{E_i(t)\}_{t\geq 0}$ such that, setting $\Gamma(t):=\mathbb R^{n+1}\setminus
\cup_{i=1}^N E_i(t)$:
\begin{itemize}
  \item[(1)] $E_i(0)=E_{0,i}$ for $i=1,\ldots,N$,
    \item[(2)] $E_1(t),\ldots,E_N(t)$ are pairwise disjoint for $t\in\mathbb R^+$,
    \item[(3)] $({\rm spt}\,\mu)_t=\Gamma(t)=\cup_{i=1}^N
    \partial E_i(t)$ for all $t>0$,
    \item[(4)] for all $t\in\mathbb R^+$, $\|V_t\|\geq |\nabla\chi_{E_i(t)}|$ for every
    $i=1,\ldots,N$, and $2\|V_t\| \geq \sum_{i=1}^N |\nabla\chi_{E_i (t)}|$,
    \item[(5)] $S(i):=\{(x,t)\,:\,x\in E_i(t),\,t\in\mathbb R^+\}$
    is open in $\mathbb R^{n+1}\times\mathbb R^+$ for $i=1,\ldots,N$,
    \item[(6)] for every $0 < T < \infty$, the families $\{E_i(t)\}_{t \in \left[0,T\right)}$ define a generalized ${\rm BV}$ solution for multi-phase MCF in the following sense: referring to Definition \ref{d:BV flow}, (a'') holds when $\Ha^n \mres_\Omega$ replaces $\Ha^n$ in \eqref{BV:finite energy}; (b'') holds when $\Ha^n \mres_\Omega$ replaces $\Ha^n$ in \eqref{scalar velocities are in L2}; (c'') holds. In fact, the scalar velocities $v_i$ are precisely $v_i(x,t) = h(x,V_t) \cdot \nu_i(x,t)$, so that \eqref{vi are velocities} reads as follows: for every $i \in \{1,\ldots,N\}$
    \begin{equation} \label{th:BV_identity}
        \left.\int_{E_i(t)}\phi(x,t)\,dx\,\right|_{t=t_1}^{t_2}=\int_{t_1}^{t_2}\Big(\int_{\partial^*E_i(t)}\phi\, h\cdot\nu_{i}\,
        d\mathcal H^n+\int_{E_i(t)}\frac{\partial\phi}{\partial t}\,dx\Big)\,dt
    \end{equation}
    for any $0\leq t_1<t_2<\infty$ and $\phi\in C^1_c(\mathbb R^{n+1}
    \times \mathbb R^+)$,
    \item[(7)]
    if $N\geq 3$, then, for $\mathcal L^1$-a.e.~$t\in\mathbb R^+$
    and $\mathcal H^n$-a.e.~$x\in\Gamma(t)$, 
    $\theta_t(x)=1$ implies $x\in \cup_{i=1}^N\partial^*
    E_i(t)$,
    \item[(8)]
    if $N=2$, then, for $\mathcal L^1$-a.e.~$t\in\mathbb R^+$
    and $\mathcal H^n$-a.e.~$x\in\Gamma(t)$,
    \begin{equation}
        \theta_t(x)=\left\{\begin{array}{ll}
        \mbox{odd integer} & \mbox{for }x\in\partial^* E_{1}(t)
        (=\partial^* E_2(t)), \\
        \mbox{even integer} & \mbox{for }x\in\Gamma(t)\setminus
        \partial^* E_1(t).\end{array}
        \right.
    \end{equation}
\end{itemize}
\end{theorem}

Notice that in (6) we speak about \emph{generalized} ${\rm BV}$ flow, because we cannot guarantee, in general, the validity of \eqref{BV formulation of mean curvature}. We will comment further on this point in Section \ref{on-g-BV}. We anticipate, nonetheless, that further information on the flow, including the validity of \eqref{BV formulation of mean curvature}, can be inferred when the Brakke flow is \emph{unit density}.
Even though the result follows immediately from Theorem
\ref{main3}(7), we state it separately. 
\begin{theorem}\label{main4}
Suppose that, for $0 \leq t_1 < t_2 < \infty$ and for an open set $U$, 
the Brakke flow in Theorem \ref{main1} is a 
unit density flow in $U\times (t_1,t_2)$. Then, for $\mathcal L^1$-a.e.~$t\in(t_1,t_2)$,
\begin{equation}\label{4-eq}
    \|V_t\|\mres_U=\mathcal H^n\mres_{U\cap \cup_{i=1}^N \partial^* E_i(t)}.
\end{equation}
Furthermore, \eqref{BV formulation of mean curvature} holds true with $v_i = h \cdot \nu_i$ for all vector fields $g \in C^1_c (U \times \left[t_1,t_2\right];\R^{n+1})$.
\end{theorem}

If we assume the following additional conditions on $\Gamma_0$, we can guarantee that the resulting Brakke flow $\{V_t\}_{t \ge 0}$ as above is unit density for short time; see also \cite[Theorem 2.2(4)]{Takasao-Tonegawa} for a similar statement valid in the context of mean curvature flows with a transport term.

\begin{theorem}\label{main5}
Suppose that $\mathcal H^n(\Gamma_0)<\infty$ and 
there exist $r_0>0$ and $\delta_0>0$ such that
\begin{equation}\label{decon}
    \sup_{x\in \mathbb R^{n+1},\, 0<r<r_0}\frac{\mathcal H^n(
    \Gamma_0\cap B_r(x))}{\omega_n r^n}<2-\delta_0. 
\end{equation}
Then there exists $T_0=T_0(n,r_0,\delta_0,\mathcal H^n(\Gamma_0))\in (0,\infty)$ such that $\{V_t\}_{t\in[0,T_0)}$
in Theorem \ref{main1} is a unit density Brakke flow
and
$\{(\chi_{E_1(t)},\ldots,\chi_{E_N(t)})\}_{t\in[0,T_0)}$ 
in Theorem \ref{main3} is a BV solution of MCF.
\end{theorem}
If we set ``the maximal existence time'' of unit density Brakke flow as
$\hat T:=\sup\{t\geq 0\,:\, \int_0^t \|V_s\|(\{x\,:\,\theta_s(x)\geq 2
\})\,ds=0\mbox{ and }\|V_t\|\neq 0\}$, under the assumption of Theorem \ref{main5}, either $\|V_t\|=0$ before $T_0$, or we have
$\hat T\geq T_0$ and $V_t$ is a non-trivial unit density flow on $[0,\hat T)$ and $\{(\chi_{E_1(t)},\ldots,\chi_{E_N(t)})\}_{t\in[0,\hat T)}$ 
is also a BV solution of MCF. 

\smallskip

Finally, the validity of the identity \eqref{th:BV_identity} allows to employ a simple argument to deduce a lower bound on the \emph{extinction time} for $V_t$, namely
on the quantity
\begin{equation} \label{def:extinct}
    T_* := \sup\lbrace t \in \R^+ \, \colon \, \|V_t\| \neq 0 \rbrace\,.
\end{equation}
Before stating the result, let us notice that if $\Ha^n(\Gamma_0) < \infty$ then, by an argument using the relative isoperimetric inequality, there is only one $i_0 \in \{1,\ldots,N\}$ such that $|E_{0,i_0}|=\infty$, and we can let $i_0=N$ without loss of generality. We then set $E(t):=\bigcup_{i=1}^{N-1}E_i(t)$.
\begin{theorem} \label{t:extinction}
If $\Ha^n(\Gamma_0) < \infty$, then the extinction time $T_*$ satisfies
\begin{equation}\label{e:extinction}
    T_* \geq 2 \left( \frac{|E(0)|}{\Ha^n(\Gamma_0)} \right)^2\,.
\end{equation}
\end{theorem}

Observe that if $\Gamma_0$ is bounded then $T_*$ must necessarily be finite, as one can see using the avoidance lemma of Brakke 
flows (see for example \cite[10.6 and 10.7]{Ilm1}) and a comparison with a 
shrinking sphere. The estimate \eqref{e:extinction} was first proved by Giga and Yama-uchi in \cite{Giga-Yamauchi} for two-phase MCF for the solution by level-set method, and it is sharp. In contrast, as far as the authors are aware of, such sharp lower bound had never been established before in the context of multi-phase MCF. Theorem \ref{t:extinction} will be proved in Subsection \ref{extinct}.

\subsection{Outline of the proofs and plan of the paper}
In what follows, we assume to have fixed
\begin{itemize}
    \item a function $\Omega$ as in Definition \ref{def:weight},
    \item a set $\Gamma_0$ and domains $\{E_{0,i}\}_{i=1}^N$ as in Assumption \ref{as1}.
\end{itemize}

The strategy towards the proof of the existence of the Brakke flow $\{V_t\}_{t \ge 0}$ and of the one-parameter families $\{E_i (t)\}_{t \ge 0}$ of open sets for $i \in \{1,\ldots,N\}$ from Theorems \ref{main1} and \ref{main3} respectively is analogous to that employed by Kim and the second-named author in \cite{KimTone}, with an important technical modification which is crucial to gain enough control on the change of volume of the grains and conclude the identity \eqref{th:BV_identity}. Let us explain this point in further detail. The scheme introduced in \cite{KimTone} can be roughly summarized as follows. One constructs, starting from $\{E_{0,i}\}_{i=1}^N$, a sequence (indexed by $j \in \mathbb N$) of piecewise constant-in-time flows of \emph{open partitions} of $\R^{n+1}$ of $N$ elements (see Section \ref{ssec:further notation}): more precisely, for each $j$ the corresponding flow $\{\E_j (t)\}_{t \in \left[0,j\right]}$ consists of constant open partitions $\E_{j}(t) = \E_{j,k} = \{E_{j,k,i}\}_{i=1}^N$ for $t$ in intervals (\emph{epochs}) $\left( (k-1)\, \Delta t_j, k\,\Delta t_j \right]$ of length $\Delta t_j \to 0^+$. The goal is then to show that if the partition $\E_{j,k+1}$ is constructed from the partition $\E_{j,k}$ appropriately then, along a suitable subsequence $j_\ell \to \infty$, the (varifolds associated to) the boundaries $\partial \E_{j_\ell}(t)$ converge to the desired Brakke flow $V_t$. 

The open partition $\E_{j,k+1}$ at a given epoch is constructed from the open partition $\E_{j,k}$ at the previous epoch by applying two operations, which we call \emph{steps}. The first step is a small \emph{Lipschitz deformation} of partitions with the effect of regularizing singularities by locally minimizing the area of the boundary of partitions at a suitably small scale; the second step consists of flowing the boundary of partitions by an appropriately defined \emph{approximate mean curvature vector}, obtained by smoothing the surfaces' first variation via convolution with a localized heat kernel. 

The only difference between the scheme employed in \cite{KimTone} and the one devised in the present paper is in the choice of the Lipschitz deformations in the first step. While in \cite{KimTone} one only requires that the change of volumes of the grains due to Lipschitz deformation is small (for a certain smallness scale), in the present paper we shall require that the change of volume is small \emph{compared to the reduction in surface measure}. We define such new class of ``\emph{volume-controlled}'' Lipschitz deformations 
in Section \ref{s:Lipschitz} and then we claim that the construction of \cite{KimTone} can be carried over using volume-controlled deformations, as outlined in
Section \ref{s:algorithm}: by the results of \cite{KimTone}, this completes the proof of Theorem \ref{main1}(1)-(4) (with some additional remarks), Theorem \ref{main2}(1), and Theorem \ref{main3}(1)-(5). 
There
are technical details for verifying that the 
modifications to the volume-controlled deformations
do not pose any difficulties, and we defer it to Appendix \ref{appendix:Lipschitz}.

In Section \ref{s: boundary evo}, the main 
result is Theorem \ref{main3}(6), namely the identity \eqref{th:BV_identity}. For that, we first prove that
any Brakke flow is $L^2$ flow in Section \ref{s:L2}, which 
is Theorem \ref{main1}(5). 
This fact seems to be first noted in \cite{Bertini}
but we include the proof for completeness. 
Next Section \ref{ssec:chi is BV} is the main
part of the proof, showing the existence of
measure-theoretic velocities. The proof consists of calculating explicitly the rate of change of integrals of test functions across different epochs in the approximating flows. 
We will prove the validity of an approximate identity for the sets $E_{j_\ell,i} (t)$, and the use of volume-controlled Lipschitz deformations will be crucial to conclude that the errors vanish in the limit as $\ell \to \infty$. The last Section \ref{ssec:derivative characterization} shows the existence of tangent
space of $\mu$ on the reduced boundary of grains (seen
as a set in space-time) and concludes the proof of 
Theorem \ref{main3}(6). 

In Section \ref{s:Ilmanen} we prove Theorem \ref{main2}(2)-(4), which, in particular, implies that the support $\spt\|V_t\|$ of the evolving varifolds is $\Ha^n$-equivalent to the boundary of partition $\Gamma (t) = \bigcup_{i=1}^N \pa E_i (t)$. This result is used
in Section \ref{s:two sided}.
The argument we follow was suggested by Ilmanen in \cite[Section 7.1]{Ilm_AC}, and we include it for the sake of completeness. It is based on the so-called \emph{Clearing Out Lemma} (see Lemma \ref{l:col}), a very robust result, which in turn follows from Huisken's monotonicity formula, stating roughly that if the localized density of $\|V_t\|$ at some point $y$ is too small at a scale $r > 0$ then necessarily the point $(y,t+r^2)$ does not belong to the support of the space-time measure $\mu$.

Section \ref{s:two sided} contains the proof of
Theorem \ref{main3}(7)(8).
They are improved version of the integrality theorem of
\cite[Theorem 8.6]{KimTone} in that the behaviors of 
approximating flows and their grains are tracked 
more in detail. 
A characterization in \cite[Section 4]{KimTone2}
of limiting behavior within a 
length scale of $o(1/j^2)$, where measure minimizing property dominates, is essentially used. At the end,
we describe the proofs of Theorem \ref{main4}
and \ref{main5}. 

Finally, section \ref{s:final remarks} contains some final remarks on the notion of generalized ${\rm BV}$ solutions introduced in Theorem \ref{main3}(6) and on the construction of canonical multi-phase Brakke flows with fixed boundary conditions in the spirit of \cite{ST19}, as well as the proof of Theorem \ref{t:extinction}.

\subsection{Further notation} \label{ssec:further notation}

We collect here some further notation and results which will be extensively used throughout the paper. We begin with the notion of open partitions of $\R^{n+1}$ of $N$ elements and corresponding admissible maps.

\begin{definition}
Let $N \ge 2$ be an integer, and let $\Omega$ be as in Definition \ref{def:weight}. A collection $\mathcal E=\{E_1,\ldots,E_N\}$ of subsets $E_i \subset \mathbb R^{n+1}$
is called an \emph{$\Omega$-finite open partition of $N$ elements} if
\begin{itemize}
    \item[(a)] $E_1,\ldots,E_N$ are open and mutually disjoint,
    \item[(b)] $\int_{\mathbb R^{n+1}\setminus\cup_{i=1}^N E_i}\Omega(x)\,
    d\mathcal H^n(x)<\infty$,
    \item[(c)] the set $\cup_{i=1}^N\partial E_i$ is countably $n$-rectifiable.
\end{itemize}
The set of all $\Omega$-finite open partitions of $N$ elements is denoted by
$\mathcal{OP}_\Omega^N$. 
\end{definition}
Note that it is allowed for some $E_i$ to be the empty set $\emptyset$. Since $\Omega>0$ everywhere, the property (b) implies that the closed set $\mathbb R^{n+1}\setminus \cup_{i=1}^N E_i$ has locally finite $\mathcal H^n$-measure, and thus no interior point. In particular, it holds 
$\R^{n+1} \setminus \bigcup_{i=1}^N E_i= \cup_{i=1}^N\partial E_i$. 
We define $$\partial\mathcal E:=\bigcup_{i=1}^N
\partial E_i\,,$$
and with a slight abuse of notation we use the same symbol to denote the varifold
\[{\bf var}\left(\bigcup_{i=1}^N
\partial E_i,1\right)\in {\bf IV}_n(\mathbb R^{n+1})\,,\] namely the unit density varifold induced from the countably $n$-rectifiable set $\pa\E$. 

\begin{definition}
Given $\E=\{E_i\}_{i=1}^N\in\mathcal{OP}_\Omega^N$, a function $f \colon \R^{n+1} \to \R^{n+1}$ is \emph{$\E$-admissible} if it is Lipschitz and if it satisfies the following. Set, for every $i$, $\tilde E_i := {\rm int}(f(E_i))$. Then:
\begin{itemize}
    \item[(a)] $\tilde E_1, \ldots , \tilde E_N$ are pairwise disjoint;
    \item[(b)] $\R^{n+1} \setminus \bigcup_{i=1}^N \tilde E_i \subset f\left( \bigcup_{i=1}^N \pa E_i \right) $;
    \item[(c)] $\sup_{x \in \R^{n+1}} \abs{f (x) - x} < \infty$.
\end{itemize}
\end{definition}

The following is \cite[Lemma 4.4]{KimTone}.

\begin{lemma}
Let $\E=\{E_i\}_{i=1}^N \in \mathcal{OP}_\Omega^N$, and let $f$ be $\E$-admissible. Set $\tilde\E := \{\tilde E_i\}_{i=1}^N$, where $\tilde E_i = {\rm int}(f(E_i))$. Then, $\tilde \E \in \mathcal{OP}_\Omega^N$. We will call $\tilde \E$ the \emph{push-forward} of $\E$ through $f$, denoted $\tilde \E =: f_\star\E$.
\end{lemma}

Next, we define a class of test functions with good properties. For $j \in \Na$, we set
\begin{eqnarray}
  \cA_j := \Big\{  \phi \in C^2 (\R^{n+1};\R^+) &\colon& \phi(x) \le \Omega (x)\,,\, \abs{\nabla\phi(x)} \leq j\,\phi(x)\,, \nonumber \\\label{the class Aj}
 && \|\nabla^2\phi(x)\| \le j\,\phi(x)\quad \mbox{for every $x \in \R^{n+1}$} \Big\}\,.
\end{eqnarray}
Note that $\Omega \in \cA_j$ for all $j \ge c_1$. 

\medskip

Finally, we give the notion of smoothed mean curvature vector of a varifold. Let $\psi \in C^\infty(\R^{n+1})$ be a radially symmetric cut-off function such that
\begin{eqnarray*}
&& \psi (x) = 1 \mbox{ for $|x|\le 1/2$}\,, \qquad \psi(x) = 0 \mbox{ for $|x|\ge 1$} \,, \\
&& 0 \le \psi(x) \le 1\,, \qquad |\nabla\psi(x)| \le 3 \,, \qquad \|\nabla^2\psi(x)\| \le 9 \mbox{ for all $x\in\R^{n+1}$}\,.
\end{eqnarray*}
Then, for every $\eps\in\left(0,1\right)$ we define
\begin{equation} \label{smoothing kernel}
    \hat\Phi_\eps(x) := \frac{1}{(2\pi\eps^2)^{\frac{n+1}{2}}} \exp\left( - \frac{|x|^2}{2\eps^2}\right)\,, \qquad \Phi_\eps(x) := c(\eps)\,\psi(x)\, \hat\Phi_\eps(x)\,,
\end{equation}
where $1 < c(\eps) \leq c(n)$ is a normalization constant chosen in such a way that
\[
\int_{\R^{n+1}} \Phi_\eps(x) \, dx = 1\,.
\]

We shall call $\Phi_\eps$ the smoothing kernel at scale $\eps$. For a varifold $V \in \V_n(\R^{n+1})$, we then define the $\eps$-smoothed mean curvature vector of $V$ to be the vector field $h_\eps(\cdot,V) \in C^\infty(\R^{n+1};\R^{n+1})$ defined by
\begin{equation} \label{smoothed mc vector}
    h_\eps(\cdot,V) := - \Phi_\eps \ast \left( \frac{\Phi_\eps \ast \delta V}{\Phi_\eps \ast \|V\| + \eps\,\Omega^{-1}} \right)\,.
\end{equation}

In the above formula, $\Phi_\eps \ast \|V\|$ is the measure on $\R^{n+1}$ defined by
\[
(\Phi_\eps \ast \|V\|) (\phi) := \|V\| (\Phi_\eps \ast \phi) = \int_{\R^{n+1}} \int_{\R^{n+1}} \Phi_\eps (x-y) \, \phi(y) \, dy \, d\|V\|(x)
\]
for all $\phi \in C_c(\R^{n+1})$, identified with the smooth function 
\[
(\Phi_\eps\ast\|V\|)(x) := \int_{\R^{n+1}} \Phi_\eps(y-x) \, d\|V\|(y)
\]
by means of the identity
\[
(\Phi_\eps \ast \|V\|) (\phi) = \langle \Phi_\eps \ast \|V\|\,,\,\phi\rangle_{L^2(\R^{n+1})}\,.
\]

Analogously, $\Phi_\eps \ast \delta V$ is the $\R^{n+1}$-valued measure on $\R^{n+1}$ defined by
\[
(\Phi_\eps \ast \delta V)(g) := \delta V (\Phi_\eps \ast g) = \int_{\R^{n+1}} g(y) \cdot \int_{\R^{n+1}\times {\bf G}(n+1,n)} S \left( \nabla\Phi_\eps (x-y) \right) \,dV(x,S)\, dy
\]
for all $g \in C_c(\R^{n+1};\R^{n+1})$, identified with the smooth vector field 
\[
(\Phi_\eps \ast \delta V)(x) := \int_{\R^{n+1} \times {\bf G}(n+1,n)} S (\nabla\Phi_\eps(y-x)) \, dV(y,S)\,.
\]

We state the following lemma concerning the smoothed mean curvature vector to be used in the 
subsequent sections. For the proof, the reader can consult \cite[Lemma 5.1]{KimTone}. 

\begin{lemma} \label{l:smoothing estimates}
For every $M > 0$, there exists a constant $\eps_1 \in \left( 0,1 \right)$, depending only on $n$, $c_1$ and $M$ such that the following holds. Let $V \in \V_{n}(\R^{n+1})$ be an $n$-dimensional varifold in $\R^{n+1}$ such that $\|V\|(\Omega) \leq M$, and for every $\eps \in \left( 0, \eps_1 \right)$, let $h_{\eps}(\cdot, V)$ be its smoothed mean curvature vector. Then:
\begin{equation} \label{e:h in L infty}
\abs{h_\eps(x, V)} \leq 2 \, \eps^{-2}\,,
\end{equation}
\begin{equation} \label{e:nabla h in L infty}
\|\nabla h_\eps(x,V)\| \leq 2\, \eps^{-4}\,.
\end{equation}
\end{lemma}

\section{Existence of multi-phase Brakke flow} \label{s:existence}

\subsection{Volume-controlled Lipschitz deformations} \label{s:Lipschitz}

In this subsection we introduce the modified class of Lipschitz deformations used in the present paper to gain improved control on the volume change of partitions. For $\E \in \mathcal{OP}_\Omega^N$ and $j \in \Na$, the class is denoted by ${\bf E}^{vc}(\mathcal{E},j)$ (upperscript $vc$
indicating ``volume controlled''), and it replaces the class ${\bf E}(\mathcal{E},j)$ defined in \cite[Definition 4.8]{KimTone}. 

\begin{definition}\label{ar8}
For $\mathcal{E}=\{E_i\}_{i=1}^N\in\mathcal{OP}_\Omega^N$ and $c_1\leq j\in\mathbb N$, define
${\bf E}^{vc}(\mathcal{E},j)$ to be the set of all $\mathcal E$-admissible functions
$f\,:\,\mathbb R^{n+1}\rightarrow\mathbb R^{n+1}$ such that, writing $\{\tilde E_i\}_{i=1}^N:=
f_{\star}\mathcal E
$, we have
\begin{itemize}
\item[(a)] $|f(x)-x|\leq 1/j^2$ for all $x\in\mathbb R^{n+1}$,
\item[(b)] ${\mathcal L}^{n+1}(\tilde E_i\triangle E_i)\leq \{\|\partial\mathcal E\|(\Omega)-\|\partial f_{\star}\mathcal E\|(\Omega)\}/j$ for all $i=1,\ldots,N$,
\item[(c)] $\|\partial f_{\star}\mathcal E\|(\phi)\leq \|\partial\mathcal E\|(\phi)$
for all $\phi\in \mathcal A_j$. 
\end{itemize}
\end{definition}
The difference between the above class ${\bf E}^{vc}(\mathcal E,j)$ and the class ${\bf E}(\mathcal E,j)$ in \cite[Definition 4.8]{KimTone} lies in the condition (b), which in \cite{KimTone} was simply ${\mathcal L}^{n+1}(\tilde E_i\triangle E_i)\leq 1/j$.
Since $\Omega\in\mathcal A_j$ for all $j\geq c_1$, (c) implies $\|\partial\mathcal E\|(\Omega)\geq \|\partial f_{\star}\mathcal E\|(\Omega)$, and thus the right-hand side of (b) is non-negative. 
In particular, the identity map $f(x)=x$ belongs to ${\bf E}^{vc}(\mathcal E,j)$ for $j\geq c_1$. We 
similarly modify the definitions of $\Delta_j\|\partial\mathcal E\|(\Omega)$, ${\bf E}(\mathcal E, C,j)$ and
$\Delta_j\|\partial\mathcal E\|(C)$ in \cite[(4.11)-(4.13)]{KimTone} by introducing their 
volume-controlled counterparts, as follows. 
\begin{definition}
For $\mathcal E\in \mathcal{OP}_\Omega^N$ and $c_1\leq j\in\mathbb N$, define
\begin{equation}\label{ar1}
\Delta_j^{vc}\|\partial\mathcal E\|(\Omega):=\inf_{f\in{\bf E}^{vc}(\mathcal E,j)}
\{\|\partial f_{\star}\mathcal E\|(\Omega)-\|\partial\mathcal E\|(\Omega)\} \le 0\,,
\end{equation}
and for a compact set $C\subset \mathbb R^{n+1}$, 
\begin{eqnarray}\label{ar2}
{\bf E}^{vc}(\mathcal E,C,j)&:=&\left\lbrace f\in{\bf E}^{vc}(\mathcal E,j)\,:\,\{x\,:\,f(x)\neq x\}\cup
\{f(x)\,:\,f(x)\neq x\}\subset C\right\rbrace\,, \\ \label{ar3}
\Delta_j^{vc}\|\partial\mathcal E\|(C)&:=&\inf_{f\in{\bf E}^{vc}(\mathcal E,C,j)}
\{\|\partial f_\star\mathcal E\|(C)-\|\partial\mathcal E\|(C)\}\,.
\end{eqnarray}
\end{definition}
Even with this modification, claims in \cite[Section 4]{KimTone} remain essentially the same:
\begin{lemma}\label{ar7}
For compact sets $C\subset \tilde C$, we have
\begin{equation}\label{ar4}
\Delta_j^{vc}\|\partial\mathcal E\|(\tilde C)\leq \Delta_j^{vc}\|\partial\mathcal E\|(C)
\end{equation}
and
\begin{equation}\label{ar5}
\Delta_j^{vc}\|\partial\mathcal E\|(\Omega)\leq (\max_{C}\Omega)\{\Delta_j^{vc}\|\partial\mathcal E\|
(C)+(1-\exp(-c_1{\rm diam}\,C))\|\partial\mathcal E\|(C)\}\,.
\end{equation}
\end{lemma}
\begin{lemma}
Suppose that $\{C_l\}_{l=1}^\infty$ is a sequence of compact sets which are mutually
disjoint and suppose that $C$ is a compact set with $\cup_{l=1}^\infty C_l\subset C$. Then
\begin{equation}\label{ar6}
\Delta_j^{vc}\|\partial\mathcal E\|(C)\leq \sum_{l=1}^\infty\Delta_j^{vc}\|\partial\mathcal E\|(C_l).
\end{equation}
\end{lemma}
We remark that the corresponding statement in \cite[Lemma 4.11]{KimTone} contains the additional assumption $\mathcal L^{n+1}(C)<1/j$, which we do not need to assume in the present setting. For completeness, we provide the proof.
\begin{proof}
By \eqref{ar3}, $\Delta_j^{vc}\|\partial\mathcal E\|(C)\geq -\|\partial\mathcal E\|(C)>-\infty$ for any compact set $C$.
Let $m\in\mathbb N$ and $\varepsilon\in(0,1)$ be
arbitrary. For all $l\leq m$, choose $f_l\in{\bf E}^{vc}(\mathcal E,C_l,j)$ such that 
$\Delta_j^{vc}\|\partial\mathcal E\|(C_l)+\varepsilon\geq \|\partial(f_l)_\star\mathcal E\|(C_l)
-\|\partial\mathcal E\|(C_l)$. We define a map $f\,:\,\mathbb R^{n+1}\rightarrow\mathbb R^{n+1}$
by setting $f\mres_{C_l}(x)=(f_l)\mres_{C_l}(x)$ for $l=1,\ldots,m$ and $f\mres_{\mathbb R^{n+1}\setminus\cup_{l=1}^m C_l}
(x)=x$. The $\mathcal E$-admissibility of $f$ follows from that of $f_l$ and from the fact that $\{C_l\}$ are mutually disjoint. To prove $f\in{\bf E}^{vc}(\mathcal E,j)$, we need to check Definition \ref{ar8}(a)-(c) and (a) follows immediately. 
Writing $\{\tilde E_i\}_{i=1}^N:=f_{\star}\mathcal E$, we have 
$\tilde E_i\triangle E_i=\cup_{l=1}^m C_l\cap(\tilde E_i\triangle E_i)$, and 
\begin{equation*}
\begin{split}
\mathcal L^{n+1}(\tilde E_i\triangle E_i)&=\sum_{l=1}^m \mathcal L^{n+1}(C_l\cap(\tilde E_i\triangle 
E_i))\leq \sum_{l=1}^m \{\|\partial\mathcal E\|(\Omega)-\|\partial(f_l)_\star\mathcal E\|(\Omega)\}/j \\
&=\sum_{l=1}^m\{\|\partial\mathcal E\|\mres_{C_l}(\Omega)-\|\partial(f_l)_\star\mathcal E\|\mres_{C_l}(\Omega)\}/j =\{\|\partial\mathcal E\|(\Omega)-\|\partial f_\star\mathcal E\|(\Omega)\}/j\,,
\end{split}
\end{equation*}
where we used (b) for $f_l\in {\bf E}^{vc}(\mathcal E,C_l,j)$ to conclude (b) for $f$. The condition
(c) can be checked similarly. These show that $f\in {\bf E}^{vc}(\mathcal E,j)$, and in fact $f\in {\bf E}^{vc}
(\mathcal E,C,j)$ by construction. By the definition of $\Delta_j^{vc}\|\pa\E\|(C)$, we have
\begin{equation*}
\begin{split}
\Delta_j^{vc}\|\partial\mathcal E\|(C) & \leq \|\partial f_\star\mathcal E\|(C)-\|\partial\mathcal E\|(C)
=\sum_{l=1}^m (\|\partial(f_l)_\star\mathcal E\|(C_l)-\|\partial\mathcal E\|(C_l)) \\
&\leq m\varepsilon+\sum_{l=1}^m \Delta_j^{vc}\|\partial\mathcal E\|(C_l).
\end{split}
\end{equation*}
By letting $\varepsilon\rightarrow 0^+$ first and then $m\rightarrow\infty$, we obtain \eqref{ar6}. 
\end{proof}
The following is similar to \cite[Lemma 4.12]{KimTone} with a minor change. The proof is identical. 
\begin{lemma}\label{adcheck}
Suppose that $\mathcal E=\{E_i\}_{i=1}^N\in\mathcal{OP}_\Omega^N$, $c_1\leq j\in\mathbb N$, $C$ is a compact subset of
$\mathbb R^{n+1}$, $f:\mathbb R^{n+1}\rightarrow\mathbb R^{n+1}$ is a $\mathcal E$-admissible function
such that, writing $\{\tilde E_i\}_{i=1}^N:=f_\star\mathcal E$, 
\begin{itemize}
\item[(a)] $\{x\,:\, f(x)\neq x\}\cup\{f(x)\,:\,f(x)\neq x\}\subset C$,
\item[(b)] $|f(x)-x|\leq 1/j^2$ for all $x\in\mathbb R^{n+1}$,
\item[(c)] $\mathcal L^{n+1}(\tilde E_i\triangle E_i)\leq \{\|\partial\mathcal E\|(\Omega)
-\|\partial f_\star\mathcal E\|(\Omega)\}/j$ for all $i=1,\ldots,N$,
\item[(d)] $\|\partial f_\star\mathcal E\|(C)\leq \exp(-j\,{\rm diam}\,C)\|\partial\mathcal E\|(C)$.
\end{itemize}
Then we have $f\in {\bf E}^{vc}(\mathcal E,C,j)$. 
\end{lemma}

\subsection{The constructive scheme} \label{s:algorithm} 

In this subsection we provide the detailed construction of the sequence of piecewise constant-in-time approximating flows of open partitions leading to the existence result of a multi-phase Brakke flow. We let the weight function $\Omega$ be as in Definition \ref{def:weight}, and we consider an initial rectifiable set $\Gamma_0$ with a corresponding $\Omega$-finite open partition of $N$ elements $\E_0$ as in Assumption \ref{as1}. For every natural number $j \ge c_1$, and for times $t \in \left[0,j\right]$, we define open partitions $\E_j (t) = \{E_{j,1}(t),\ldots,E_{j,N}(t)\}$ according to the following rule:
\begin{eqnarray} \label{app:0} 
    \E_j(0)&=&\E_0 \,, \\ \label{app:future}
    \E_j(t) &=& \E_{j,k} \qquad \mbox{for all $t \in \left( (k-1) \Delta t_j, k \Delta t_j \right]$}\,.
\end{eqnarray}
In \eqref{app:future}, the epoch length is $\Delta t_j = 2^{-p_j}$ for some $p_j \in \Na$, and $k\in\{1,\ldots,j\,2^{p_j}\}$. For each $k$, the open partition $\E_{j,k}$ is obtained from the open partition $\E_{j,k-1}$ (with the convention $\E_{j,0} = \E_0$) through successive modifications, encoded in the following two-step algorithm:
\begin{itemize}
    \item[(1)] First, one chooses $f_1 \in \mathbf{E}^{vc}(\E_{j,k-1}, j)$ with the property that
\begin{equation} \label{e:almost optimal}
\| \partial (f_1)_{\star} \E_{j,k-1} \|(\Omega) - \| \partial \E_{j,k-1} \|(\Omega) \leq  (1-j^{-5}) \, \Delta^{vc}_j \| \partial \E_{j,k-1} \|(\Omega)\,,
\end{equation}
and sets
\begin{equation}\label{first step}
\E^*_{j,k} := (f_1)_\star(\E_{j,k-1});
\end{equation}
thus, in particular,
\begin{equation}\label{first step set}
E^*_{j,k,i} := {\rm int}(f_1(E_{j,k-1,i})) \quad \mbox{for every $i \in \{1,\ldots,N\}$}\,.
\end{equation} 

\item[(2)] Next, one defines the map

\begin{equation} \label{motion by smoothed mean curvature}
f_2(x) := x + \Delta t_j \, h_{\eps_j}(x, \partial\E^*_{j,k})\,,
\end{equation} 
where $\eps_j \in \left(0,1\right)$, and $h_{\eps_j}(\cdot,\pa\E^*_{j,k})$ is the $\eps_j$-smoothed mean curvature vector of the multiplicity one varifold on $\pa\E^*_{j,k}$. Notice that $f_2$ is a diffeomorphism of $\R^{n+1}$ due to Lemma \ref{l:smoothing estimates} as soon as $\Delta t_j \ll \eps_j^4$. We set
 \begin{equation} \label{second step}
 \E_{j,k} := (f_2)_{\star}\E^*_{j.k}\,, 
 \end{equation}
 and thus
 \begin{equation}\label{second step set}
 E_{j,k,i} := f_2(E^*_{j,k,i})\quad \mbox{for every $i\in\{1,\ldots,N\}$}\,.
 \end{equation}
\end{itemize}

Notice that the scheme just defined differs from that adopted in \cite{KimTone} only in step (1) of the algorithm, where volume-controlled Lipschitz deformations are used. In spite of such modification, we claim that the proof from \cite{KimTone} can be essentially carried over to this framework, thus leading to the following theorem.

\begin{theorem}\label{p:properties and limit}
There is a constant $c_2=c_2(n) \gg 1$ with the following property. Let $\Omega$, $\Gamma_0$, and $\E_0 \in \mathcal{OP}_\Omega^N$ be as in Definition \ref{def:weight} and Assumption \ref{as1}. Then there exist 
\begin{itemize}
\item a subsequence $j_\ell$ of $\mathbb N$,
\item reals $\eps_{j_\ell} \in \left( 0, j_\ell^{-6}\right)$ with $\lim_{\ell\to\infty} \eps_{j_\ell}=0$,
\item integers $p_{j_\ell} \in \mathbb N$ with $\Delta t_{j_\ell}=2^{-p_j} \in \left( 2^{-1} \eps_{j_\ell}^{c_2}, \eps_{j_\ell}^{c_2} \right]$,
\item a family $\{\mu_t\}_{t \in \R^+}$ of Radon measures on $\R^{n+1}$,
\item and a family $\E(t)=\{E_1(t),\ldots,E_N(t)\}_{t \ge 0}$ of open sets
\end{itemize}
such that the approximating flow of open partitions $\E_{j_\ell}(t)$ defined by \eqref{app:0}-\eqref{app:future} satisfies for all $T < \infty$:
\begin{align} \label{e:mass bound}
    &\limsup_{\ell \to \infty} \sup_{t \in \left[0,T\right]} \| \partial \E_{j_\ell}(t)\|(\Omega) \leq \| \partial\E_0\|(\Omega)\,\exp(c_1^2 T / 2)\,,\\ \label{e:mean curvature bounds}
    &\limsup_{\ell\to\infty} \int_0^T \left( \int_{\R^{n+1}} \frac{\abs{\Phi_{\eps_{j_\ell}} \ast \delta (\partial \E_{j_\ell}(t))}^2\,\Omega}{\Phi_{\eps_{j_\ell}} \ast \|\partial \E_{j_\ell}(t)\| + \eps_{j_\ell}\,\Omega^{-1}} \, dx - \frac{1}{\Delta t_{j_\ell}}\, \Delta_{j_\ell}^{vc}\|\partial \E_{j_\ell}(t)\|(\Omega) \right) \, dt < \infty \,,    \\ \label{e:speed of lmm lip}
    &\lim_{\ell\to\infty} j_\ell^{2(n+1)} \,\Delta_{j_\ell}^{vc} \| \partial\E_{j_\ell}(t) \| (\Omega) = 0 \qquad \mbox{for a.e. $t \in \R^+$}\,,\\ \label{e:boundary convergence}
    &\lim_{\ell\to\infty} \| \partial \E_{j_\ell}(t) \| (\phi) = \mu_t (\phi) \qquad \mbox{for all $\phi \in C_c (\R^{n+1})$ and any $t \in \R^+$}\,, \\ \label{e:partitions convergence}
    &  \mbox{$\chi_{E_{j_\ell,i}(t)} \to \chi_{E_i(t)}$ in $L^1_{{\rm loc}}(\R^{n+1})$ as $\ell \to \infty$ for every $i \in \{1,\ldots,N\}$ and for every $t \in\mathbb R^+$}\,.
\end{align}
Furthermore, the following holds:
\begin{itemize}
    \item[(a)] There exists a subset $Z \subset \R^+$ with $\Leb^1(Z)=0$ such that, for every $t \in \R^+\setminus Z$, $\mu_t$ is integral: that is, there exists $V_t \in \IV_n(\R^{n+1})$ such that $\|V_t\|=\mu_t$;
    \item[(b)] If $V_t$ is defined to be an arbitrary varifold in $\V_n(\R^{n+1})$ with $\|V_t\|=\mu_t$ also for $t \in Z$, then the family $\{V_t\}_{t \in \R^+}$ is an $n$-dimensional Brakke flow in $\R^{n+1}$ satisfying the conclusions of Theorem \ref{main1}(1)-(4) and Theorem \ref{main2}(1);
    \item[(c)] The flow of grains $E_i(t)$ satisfies the conclusions of Theorem \ref{main3}(1)-(5).
\end{itemize}
\end{theorem}

The proofs of the claims contained in the statement of Theorem \ref{p:properties and limit} are analogous to the corresponding ones outlined in \cite{KimTone}, modulo few technical modifications in consequence of the use of volume-controlled Lipschitz deformations. More precisely:
\begin{itemize}
    \item the conclusions \eqref{e:mass bound} to \eqref{e:boundary convergence} are contained in \cite[Proposition 6.1, Proposition 6.4]{KimTone};
    \item the existence of the flow of grains $E_i (t)$, the convergence in \eqref{e:partitions convergence}, and the conclusion (c) is \cite[Theorem 3.5]{KimTone};
    \item the conclusion (a) is \cite[Lemma 9.1]{KimTone};
    \item the conclusion (b) is \cite[Theorem 3.2, Proposition 3.4]{KimTone}, except that 
    Theorem \ref{main1}(2) follows from the argument in 
    \cite[Proposition 6.10]{ST19} and
    Theorem \ref{main1}(4) follows from 
    Brakke's inequality and Theorem \ref{main1}(3) with $\Omega=1$: Fix $\phi\in C_c^\infty(\mathbb R^{n+1};[0,1])$ with $\phi(x)=1$ on $B_1$ and $\phi(x)=0$
    on $\mathbb R^{n+1}\setminus B_2$, and for $k\in\mathbb N$, set $\phi_k(x):=\phi(x/k)$. Use $\phi_k$ in 
    \eqref{brakineq} with $t_1=0$ and arbitrary $t=t_2>0$
    and let $k\rightarrow\infty$. In doing so, 
    \begin{equation*}
    \int_{0}^{t}\int|\nabla\phi_k||h|\,d\|V_s\|ds
    \leq \Big(\int_0^{t}\int |\nabla\phi_k|^2\,d\|V_s\|ds\Big)^{\frac12}\Big(\int_0^t\int |h|^2\,d\|V_s\|ds\Big)^{\frac12}\rightarrow 0
    \end{equation*}
    as $k\rightarrow\infty$ due to $|\nabla\phi_k|
    =|\nabla\phi|/k$ and uniform bounds from
    Theorem \ref{main1}(3). Then one can see that Theorem \ref{main1}(4) holds true.

\end{itemize}

We will not repeat the proofs of the above results, but we will detail the aforementioned changes due to volume-controlled Lipschitz deformations in Appendix \ref{appendix:Lipschitz}. For the time being, we will assume the validity of Theorem \ref{p:properties and limit}, and we will proceed with the derivation of all remaining results claimed in Section \ref{s:notation}.

\section{\textrm{BV} flow: proof of Theorem \ref{main3}(6)} \label{s: boundary evo}
 The main result of this section is the proof of Theorem \ref{main3}(6), which we isolate as the following

\begin{theorem} \label{t:velocity of boundary}
The one-parameter family $\{E_i(t)\}_{t\in\R^+}$ of open partitions defined in Theorem \ref{p:properties and limit} is a generalized ${\rm BV}$ solution to multi-phase MCF with scalar velocities
$v_i=h \cdot \nu_i$. More precisely, for every $i \in \{1,\ldots,N\}$ it holds
\begin{equation}\label{e:evo_eq}
\begin{split}
 \left.\int_{E_i(t)} \phi(x,t) \, dx\right|_{t=t_1}^{t_2} = & \int_{t_1}^{t_2} \int_{E_i(t)} \frac{\partial\phi}{\partial t}(x,t) \, dx\,dt  \\
& + \int_{t_1}^{t_2} \int_{\pa^*E_i(t)} \phi(x,t) \, h(x,V_t) \cdot \nu_{E_i(t)}(x)\,d\Ha^n(x)\,dt 
\end{split}
\end{equation}
for every $0 \le t_1 < t_2 < \infty$ and all test functions $\phi \in C^1_c(\R^{n+1}  \times \R^+)$.
\end{theorem}

\begin{remark} \label{rmk:what is BV}
Introducing the notation $\chi_i$ for the indicator function $\chi_i(x,t) := \chi_{E_i(t)} (x)$, and recalling that $ \nabla\chi_{E_i(t)} = - \nu_{E_i(t)}\, \Ha^n \mres_{\pa^*E_i(t)}$ as $\R^{n+1}$-valued Radon measures on $\R^{n+1}$, we see that the identity in \eqref{e:evo_eq} can be rephrased as 
\begin{equation}\label{e:evo_eq_BV}
\begin{split}
 \left.\int_{\R^{n+1}} \phi(x,t) \, \chi_i(x,t) \, dx\right|_{t=t_1}^{t_2} = & \int_{t_1}^{t_2} \int_{\R^{n+1}} \frac{\partial\phi}{\partial t}(x,t) \, \chi_i(x,t) \,dx\,dt  \\
& - \int_{t_1}^{t_2} \int_{\R^{n+1}} \phi(x,t) \, h(x,V_t) \cdot d\nabla\chi_{E_i(t)}(x)\,dt\,.
\end{split}
\end{equation}
This implies, in particular, that $\chi_i \in {\rm BV}_{{\rm loc}}(\R^{n+1} \times \R^+)$ with derivative
\begin{equation} \label{e:the BV derivative}
\begin{split}
    \nabla'\chi_i(x,t) &= \left( \nabla\chi_{E_i(t)}(x), -h(x,V_t) \cdot \nabla\chi_{E_i(t)}(x) \right) \, dt \\
            &= \left(-\nu_i(x,t), h(x,V_t) \cdot \nu_i(x,t)   \right) \, d\Ha^n \mres_{\pa^*E_i(t)} \, dt\,,
\end{split}
\end{equation}
where $\nabla'=(\nabla,\pa_t)$ and the identity holds in the sense of $(\R^{n+1}\times\R)$-valued Radon measures on $\R^{n+1} \times \R^+$. \\
 \end{remark}

The proof of Theorem \ref{t:velocity of boundary} will be obtained in three main steps, which are the content of Subsection \ref{s:L2} to Subsection \ref{ssec:derivative characterization}, respectively. First, we check 
that Brakke flow implies $L^2$ flow and recall the important observation 
Corollary \ref{cor:tangent vector} due to Mugnai-R\"{o}ger \cite{Mugnai}. Second, we will prove that, for each $i \in \{1,\ldots,N\}$, $\chi_i \in {\rm BV}_{{\rm loc}}(\R^{n+1} \times \R^+)$, see Proposition \ref{p:existence of a measure derivative}; then, we will characterize the measure-theoretic time-derivative of $\chi_i$ to show that \eqref{e:the BV derivative} holds.

\subsection{Characterization as $L^2$ flow} \label{s:L2}

In this subsection, we first prove Theorem \ref{main1}(4), which we isolate as the following
\begin{theorem} \label{t:Brakke implies L2}
Let $\{V_t\}_{t \in \R^+}$ be the Brakke flow defined in Theorem \ref{p:properties and limit}. Then, for every $0 < T < \infty$ and for every open and bounded $U \subset \R^{n+1}$, the varifolds $V_t\, \mres_{(U\times {\bf G}(n+1,n))}$ ($t \in \left[0,T\right)$) are an $n$-dimensional $L^2$ flow with generalized velocity vector $v(\cdot,t) = h (\cdot, V_t)$.
\end{theorem}

\begin{proof}
We verify the requirements of Definition \ref{d:L2-flow}. Conditions (a) and (b) are satisfied by $V_t$ in $\R^{n+1}$, and thus they are trivially satisfied when the varifolds are restricted to the open set $U$. Concerning (c'), we have immediately that
\[
\int_0^T \int_U \abs{h(\cdot,V_t)}^2\,d\|V_t\|\,dt \leq \left(\max_{{\rm clos}\,U} \Omega^{-1} \right) \, \int_0^T \int_{\R^{n+1}} \abs{h(\cdot, V_t)}^2\, \Omega \, d\|V_t\|\,dt < \infty \,, 
\]
so that $h(\cdot, V_t) \in L^2(\|V_t\| \mres_U; \R^{n+1})$ for a.e.\,$t \in \left[0,T\right)$. Analogously, $d\|V_t\|\,dt$ is a Radon measure on $U \times \left( 0, T \right)$, and in fact we have
\[
\left|\int_0^T\int_U \phi(x,t) \, d\|V_t\|(x)\, dt\right| \le \|\phi\|_{C^0} \, \left(\max_{{\rm clos}\,U} \Omega^{-1}\right) \, \Ha^n \mres_{\Omega}(\Gamma_0) \, \int_0^T \exp(c_1^2 t / 2) \, dt < \infty
\]
for every $\phi \in C_c(U \times \left( 0, T \right))$. To complete the proof, we are then only left with checking that \eqref{e:L^2-flow def} holds with $v(\cdot, t) = h(\cdot, V_t)$: with this choice, indeed, the condition in (d'1) is automatically satisfied due to the perpendicularity of mean curvature
(see \cite[Chapter 5]{Brakke}). Consider first a test function $\phi \in C^1_c(\R^{n+1} \times \left( 0, T \right];\R^+)$. Brakke's inequality \eqref{brakineq} then implies that
\begin{equation}\label{brL2_1}
\begin{split}
 & \int_0^T \int_{\R^{n+1}} \frac{\partial\phi}{\partial t}(x,t) + h(x,V_t) \cdot \nabla\phi(x,t) \, d\|V_t\|(x) \, dt \\
 & \qquad \ge \int_0^T \int_{\R^{n+1}} \phi(x,t) \, \abs{h(x,V_t)}^2 \, d\|V_t\|(x)\, dt \ge 0\,.
\end{split}
\end{equation}
In particular, the assignment 
\[
\phi \mapsto L\phi := \int_0^T \int_{\R^{n+1}} \frac{\partial\phi}{\partial t}(x,t) + h(x,V_t) \cdot \nabla\phi(x,t) \, d\|V_t\|(x) \, dt
\]
defines a positive linear functional on $C^1_c (\R^{n+1} \times \left(0,T\right])$: hence, $L$ is monotone, that is $L\phi_1 \le L \phi_2$ whenever $\phi_1 \leq \phi_2$ everywhere. For every $\eps \in \left(0,1\right)$, let $\psi_{U,\eps} \in C^1_c (\R^{n+1} \times \left(0,T\right])$ be the cut-off function $\psi_{U,\eps}(x,t) = \psi_U(x)\psi_\eps(t)$ defined according to the following prescriptions:
\begin{itemize}
    \item[(i)] $0 \le \psi_{U} \le 1$ and $0\le \psi_\eps\le 1$,
    \item[(ii)] $\psi_U \equiv 1$ on ${\rm clos}\, U$ and $\psi_U \equiv 0$ on $\dist(x,{\rm clos}\, U) \ge 1$,
    \item[(iii)] $\psi_\eps \equiv 1$ on $\left[\eps,T\right]$ and $\psi_\eps \equiv 0$ on $\left(0, \eps/2\right]$,
    \item[(iv)] $ \| \nabla \psi_U\|_{C^0} \le C$ and $ \|\psi_\eps' \|_{C^0} \leq C/\eps $  for a geometric constant $C$.
\end{itemize}

Let now $\phi \in C^1_c( U \times \left(0,T\right))$ be a non-zero function such that $\spt(\phi) \subset U \times \left[\eps,T\right]$. For such a function, by the definition of $\psi_{U,\eps}$, it holds 
\[
-\psi_{U,\eps} \le \frac{\phi}{\|\phi\|_{C^0}} \leq \psi_{U,\eps}\,,
\]
so that the linearity and monotonicity of $L$ yield
\[
\abs{L \phi} \le L \psi_{U,\eps} \, \|\phi\|_{C^0}\,.
\]
Notice that for such a $\phi$ the space integration can be restricted to $U$ (as the -- continuous -- derivatives of $\phi$ are necessarily zero on $U^c$ for every $t$), and thus the above argument shows that whenever $\phi \in C^1_c(U \times \left(0,T\right))$ is supported on $U \times \left[ \eps, T \right]$ it holds
\begin{equation} \label{brL2_2}
    \left|\int_0^T \int_{U} \frac{\partial\phi}{\partial t}(x,t) + h(x,V_t) \cdot \nabla\phi(x,t) \, d\|V_t\|(x) \, dt\right| \leq L \psi_{U,\eps}\, \|\phi\|_{C^0}\,.
\end{equation}
We next proceed to estimate $L\psi_{U,\eps}$. We have, setting $U_1 := \{\dist(x,{\rm clos}\,U) < 1\}$,
\[
\begin{split}
L\psi_{U,\eps} & \le \frac{C}{\eps} \int_{\eps/2}^\eps \|V_t\|(U_1)  \, dt + \int_0^T \int_{\R^{n+1}} \abs{h(\cdot,V_t)} \, \abs{\nabla\psi_U} \, d\|V_t\|\, dt \\ 
& \le C \, \left( \max_{{\rm clos}\, U_1} \Omega^{-1} \right)\, \left( \Ha^n \mres_{\Omega}(\Gamma_0) + (T\Ha^n \mres_{\Omega}(\Gamma_0))^{1/2} \, \left( \int_0^T \int_{\R^{n+1}} \abs{h(\cdot,V_t)}^2 \, \Omega\,d\|V_t\|\,dt \right)^{1/2} \right)\,.
\end{split}
\]
We see then that $\sup_{\eps > 0} L \psi_{U,\eps} < \infty$: in particular, thanks to \eqref{brL2_2}, we conclude that \eqref{e:L^2-flow def} holds with $v(\cdot,t)=h(\cdot,V_t)$ for every $\phi \in C^1_c(U \times \left( 0, T \right))$, thus completing the proof.
\end{proof}

The following is a simple corollary of Theorem \ref{t:Brakke implies L2} and \cite[Proposition 3.3]{Mugnai}.

\begin{corollary} \label{cor:tangent vector}
Let $\{V_t\}_{t \in \R^+}$ be the Brakke flow defined in Theorem \ref{p:properties and limit}, and let $\mu$ be the space-time measure $d\mu=d\|V_t\|\,dt$. Then,
\begin{equation} \label{tangent vector}
    \begin{pmatrix}
    h(x_0,V_{t_0}) \\
    1
    \end{pmatrix} \in T_{(x_0,t_0)}\mu
\end{equation}
at $\mu$-a.e.~$(x_0,t_0)$ such that the tangent space $T_{(x_0,t_0)}\mu$ exists.
\end{corollary}

\subsection{Existence of measure-theoretic velocities} \label{ssec:chi is BV}

We have the following

\begin{proposition} \label{p:existence of a measure derivative}
Let $\{\E(t)\}_{t \in \R^+}$ and $\{V_t\}_{t\in\R^+}$ be as in Theorem \ref{p:properties and limit}, and let $\mu\mres_{\Omega}$ denote the Radon measure $\mu\mres_\Omega := \Omega\,d\|V_t\|\,dt$. Also, for every $i\in\{1,\ldots,N\}$, set
\[
S(i) := \left\lbrace \left(x,t\right) \in \R^{n+1} \times \R^+ \, \colon \, x \in E_i(t) \right\rbrace\,.
\]
Then, for every $i$ there exists a function $u_i = u_i (x,t) \in L^2(\mu \mres_{\Omega})$ (on $\mathbb R^{n+1}\times [0,T]$
for arbitrary $T>0$) with $\spt(u_i) \subset \pa S(i)$ and such that
\begin{equation} \label{e:first version of BV}
\begin{split}
 \left.\int_{\R^{n+1}} \phi(x,t) \, \chi_i(x,t) \, dx\right|_{t=t_1}^{t_2} = & \int_{t_1}^{t_2} \int_{\R^{n+1}} \frac{\partial\phi}{\partial t}(x,t) \, \chi_i(x,t) \,dx\,dt  \\
& + \int_{t_1}^{t_2} \int_{\R^{n+1}} \phi(x,t) \, u_i(x,t) \, d\|V_t\|\,dt
\end{split}\end{equation}
for every $0 \leq t_1 < t_2 < \infty$ and all test functions $\phi \in C^1_c(\R^{n+1} \times \mathbb R^+)$.
\end{proposition}

Proposition \ref{p:existence of a measure derivative} readily implies the following

\begin{corollary} \label{cor:chi is BV}
For every $i\in \{1,\ldots,N\}$, $\chi_i \in {\rm BV}_{{\rm loc}}(\R^{n+1} \times \R^+)$, and
\begin{equation} \label{e:first version derivative}
    \nabla'\chi_i(x,t) = \left( \nabla\chi_{E_i(t)}\,dt, u_i(x,t)\,d\|V_t\|\,dt \right)\,.
\end{equation}
In particular, since $\chi_i$ is the indicator function of $S(i)$, the latter is a set of locally finite perimeter in $\R^{n+1} \times \R^+$.
\end{corollary}

\begin{proof}[Proof of Proposition \ref{p:existence of a measure derivative}]

We shall divide the proof into several steps. At first, we will drop the dependence of the test function on time, which will be introduced again at a later step. The idea of the proof is to obtain an evolution identity for the approximating flow $\pa\E_{j_\ell}(t)$ defined in Subsection \ref{s:algorithm} (where $j_\ell$ is the sequence of Theorem \ref{p:properties and limit}), and then to prove that, as $\ell \to \infty$, such identity converges to \eqref{e:first version of BV}. Notice that we can assume without loss of generality that $t_1=0$, since then the seemingly more general identity can be retrieved by simply taking differences. We will then also rename $t_2=T \in \left( 0, \infty \right)$.

\medskip

\textit{Step one: preliminary reductions.} By density, it is evidently sufficient to show that the identity \eqref{e:first version of BV} holds when $\phi=\phi(x)$ is a function in $C^2_c(\R^{n+1})$. We may assume without loss of generality that $|\phi|\leq \Omega$. Consider the approximating flow $\E_{j_\ell}(t)$, and fix $\ell$ so large that the flow is defined on the interval $\left[0,T\right]$. We will deduce the validity of an approximate identity for the approximating flow, with vanishing errors in the limit as $\ell \to \infty$. We fix the index $i \in \{1,\ldots, N \}$, and, for the sake of simplicity in the notation, we drop the corresponding subscripts, so to write $\E(t)$ in place of $\E_{j_\ell}(t)$ and $E(t)$ in place of $E_{j_\ell,i}(t)$. Recalling the construction of $\E(t)$, we let $k_T$ be the integer in $\{1,\ldots,j_\ell\,2^{p_{j_\ell}}\}$ such that $T \in \left( (k_T-1) \Delta t, k_T\Delta t\right]$ (with $\Delta t=\Delta t_{j_\ell}$), so that $E(T)=E(k_T\,\Delta t) = E_{k_T}:=E_{j_\ell,k_T,i}$. We then have the discretization 
\begin{equation}\label{discretization1}
\int_{E(T)} \phi(x) \, dx - \int_{E_0}  \phi(x) \, dx = \sum_{k=0}^{k_T-1} \left\lbrace\int_{E_{k+1}} \phi(x) \, dx - \int_{E_{k}}  \phi(x) \, dx\right\rbrace\,,
\end{equation}
and we can further decompose each summand on the right-hand side of \eqref{discretization1} as 
\begin{equation} \label{calculation1}
\int_{E_{k+1}} \phi(x) \, dx - \int_{E_{k}}  \phi(x) \, dx = D_1 + D_2\,,
\end{equation}
  where
 \begin{eqnarray} \label{first summand}
  D_1 &:=& \int_{E^*_{k+1}} \phi(x) \, dx - \int_{E_{k}}  \phi(x) \, dx\,, \\ \label{second summand}
  D_2 &:=& \int_{E_{k+1}} \phi(x) \, dx - \int_{E^*_{k+1}}  \phi(x) \, dx\,.
 \end{eqnarray}
  Here, $E^*_{k+1} = {\rm int}(f_1(E_k))$ for some Lipschitz function $f_1 \in \bE^{vc}(\E_k,j_\ell)$ satisfying the almost-optimality condition
 \begin{equation} \label{alm_opt in BV}
     \|\pa(f_1)_\star\E_k\|(\Omega) - \| \pa\E_k\|(\Omega) \leq (1-j_\ell^{-5})\, \Delta^{vc}\|\pa\E_k\|(\Omega)\,,
 \end{equation}
where $\Delta^{vc}\|\pa\E_k\|(\Omega):=\Delta^{vc}_{j_\ell}\|\pa\E_k\|(\Omega)$, and $ E_{k+1} = f_2(E^*_{k+1})$, with $f_2 (x) = x + \Delta t\, h_\eps(x,\pa\E^*_{k+1})$ ($\eps=\eps_{j_\ell}$). We will now proceed evaluating $D_1$ and $D_2$ separately. 

\medskip

\textit{Step two: evaluation of $D_1$ and $D_2$.} For $D_1$, we simply observe that
 \[
 D_1 = \int_{E^*_{k+1} \setminus E_{k}} \phi(x)\, dx - \int_{E_{k} \setminus E^*_{k+1}} \phi(x)\, dx\,,
 \]
 so that, using Definition \ref{ar8}(b) and \eqref{ar1}, it holds
 \begin{equation} \label{first error}
 |D_1| \leq 2\, \|\phi\|_{C^0}\,\mathcal{L}^{n+1}(E_{k} \triangle E^*_{k+1}) \leq - 2\,\|\phi\|_{C^0}\, \frac{\Delta^{vc} \| \partial \E_{k} \|(\Omega)}{\Delta t} \, \frac{\Delta t}{j_\ell}\,.
 \end{equation}
We then proceed with $D_2$, and in order to further ease the notation we set $f(x) := f_2(x)$, $F(x) = f(x) - x$, as well as $g := f^{-1}$ and $G(x) := g(x) - x$. We also simply write
\[
E := E_{k+1}\,, \qquad E^*:= E^*_{k+1} = g(E)\,.
\]
Using that $g$ is a diffeomorphism and changing variable in the second integral in $D_2$, we have
\begin{equation} \label{e:splitting D2}
\begin{split}
D_2 &= \int_{E} \phi(x)  \, dx - \int_{E} \phi(g(x))\, \mathbf{J}g(x) dx \\
 &= \int_{E} \{\phi(x) - \phi(g(x))\} \, dx + \int_{E} \phi(x) \, \{1-\mathbf{J}g(x)\}\, dx + \int_{E} \left( \phi(x) - \phi(g(x)) \right)\,\left( \mathbf{J}g(x)-1 \right) \, dx\,,
\end{split}
\end{equation}
where $\mathbf J g$ denotes the Jacobian determinant of $g$. Using that $G(x) = - F(g(x))$ and that $F(y) = \Delta t \, h_{\eps}(y)$ ($h_{\eps}(\cdot) = h_{\eps}(\cdot, \partial  \E^*_{k+1})$) together with \eqref{e:h in L infty}, we have that if $(\spt\,\phi)_1 := \left\lbrace \dist(x,\spt\,\phi) \leq 1 \right\rbrace$ then
\begin{align} \label{Aj zero order}
&\abs{\phi(g(x)) - \phi(x)}  \leq  \|\nabla\phi\|_{C^0}\, \abs{G(x)}\, \chi_{(\spt\,\phi)_1}(x) \leq \Delta t\,\eps^{-3}\, \chi_{(\spt\,\phi)_1}(x),  \\ \label{Aj first order}
&\abs{\phi(g(x)) - \phi(x) - \nabla\phi(x) \cdot G(x)} \leq \|\nabla^2\phi\|_{C^0}\, \abs{G(x)}^2 \, \chi_{(\spt\,\phi)_1}(x)  \leq \Delta t \, \eps^{c_2-6}\,\chi_{(\spt\,\phi)_1}(x)\,.
\end{align}
Notice that, in order to apply the estimate \eqref{e:h in L infty}, we are using that $\|\pa\E^*_{k+1}\|(\Omega) \leq \|\pa\E_k\|(\Omega)$ and we are assuming that $\eps=\eps_{j_\ell}$ is smaller than the $\eps_1$ given by Lemma \ref{l:smoothing estimates} corresponding to the choice $M=\sup_{t \in \left[0,T\right]}\|\pa\E_{j_\ell}(t)\|(\Omega)$, which is bounded by a quantity depending only on $T$, the initial datum and $c_1$ due to \eqref{e:mass bound}. In deducing \eqref{Aj zero order}-\eqref{Aj first order}, we have also used that, by the definition of $g$ and \eqref{e:h in L infty}, and for all $\ell$ large enough, if $x \notin (\spt\,\phi)_1$ then $g(x) \notin \spt\,\phi$, and that $\eps$ can be assumed sufficiently small, so that $\eps\,\|\nabla\phi\|_{C^0}$ and $\eps\,\|\nabla^2\phi\|_{C^0}$ are both bounded by $1$. Next, we also estimate
\begin{align}  \label{jacobian zero order}
\abs{\mathbf{J}g(x)-1} &\leq c(n) \, \|\nabla G(x) \| \leq c(n)\, \Delta t \, \eps^{-4}\,, \\ \label{jacobian first order}
\abs{\mathbf{J}g(x) - 1 - {\rm div}\,(G(x))} &\leq c(n)\, \|\nabla G(x)\|^2 \leq \Delta t \, \eps^{c_2-9}\,,
\end{align}
where we have used that $\nabla G = - [(\nabla F) \circ g]\, \nabla g = - [(\nabla F) \circ g] \, [(\nabla f)^{-1} \circ g]$, that $\nabla F = \Delta t\, \nabla h_{\eps}$, and the estimate \eqref{e:nabla h in L infty}.
We can then conclude from \eqref{Aj first order} and \eqref{jacobian first order} that the sum of the first two integrals in \eqref{e:splitting D2} is
\[
= - \int_E \{\nabla \phi(x) \cdot G(x) + \phi(x)\, {\rm div}\,(G(x))\}\, dx + {\rm Err_1}\,,
\]
where 
\begin{equation} \label{error1}
\abs{{\rm Err}_1} \leq \Delta t \, \eps^{c_2-10} 
\end{equation}
as soon as $\eps$ is small enough to absorb the constants depending on $\Leb^{n+1}((\spt\,\phi)_1)$. On the other hand, using that $\nabla \phi(x) \cdot G(x) + \phi(x)\, {\rm div}\,(G(x)) = {\rm div}\, (\phi(x) G(x))$ and the divergence theorem we have
\begin{equation} \label{the leading term}
\begin{split}
-\int_E \{ \nabla \phi(x) \cdot G(x) + \phi(x)\, {\rm div}\,(G(x))\}\, dx &= - \int_{\partial^* E} \phi(x) \, G(x) \cdot \nu_E(x) \, d\Ha^n(x) \\
&= \Delta t \int_{\partial^*E} \phi(x) \,   h_{\eps}(g(x))  \cdot \nu_E(x)   \, d\Ha^n(x)\,.
\end{split}
\end{equation}
Recall that, in the right-hand side of \eqref{the leading term}, $h_{\eps}(\cdot) = h_{\eps}(\cdot, \partial\E^*_{k+1})$. We then continue the chain of identities in \eqref{the leading term} as
\[
= \Delta t \int_{\partial^*E} \phi(x) \,   h_{\eps}(x,\partial \E_{k+1})  \cdot \nu_E(x)   \, d\Ha^n(x) + {\rm Err}_2\,,
\]
where
\[
{\rm Err}_2 := \Delta t \, \int_{\partial^*E} \phi(x)\, \nu_E(x) \cdot \left( h_{\eps}(g(x),\partial \E^*_{k+1}) - h_{\eps}(x,\partial \E_{k+1})  \right) \, d\Ha^n(x)\,.
\]
Since $\E_{k+1} = f_\star\E^*_{k+1}$, and $f$ is a diffeomorphism, we have that $\partial \E_{k+1} = f_\sharp \partial\E^*_{k+1}$, where $f_\sharp$ denotes the push-forward operator on integral varifolds through $f$. Calling, for simplicity, $V=\pa\E^*_{k+1}$ and $\hat V = \pa\E_{k+1} = f_\sharp V$, we can write 
\begin{equation} \label{err2 definition}
h_\eps(x, \hat V) - h_\eps(g(x),V) = [h_\eps(x,\hat V) - h_\eps(g(x),\hat V)] + [h_\eps(g(x),\hat V) - h_\eps(g(x),V)] \,.
\end{equation}
Recall, now, that the $\eps$-smoothed mean curvature vector $h_\eps(\cdot,V)$ is defined by
\[
h_\eps(\cdot,V) = - \Phi_\eps \ast \frac{\Phi_\eps \ast \delta V}{\Phi_\eps \ast \|V\| + \eps\,\Omega^{-1}}\,.
\]
In particular, using \cite[Equation (5.74) and Lemma 4.17]{KimTone}, it is easy to estimate the second summand on the right-hand side of \eqref{err2 definition} by
\begin{equation}\label{smc error 1}
\abs{h_\eps(g(x),\hat V) - h_\eps(g(x),V)} \leq \eps^{c_2-2n-14}(\|V\|(\Omega)+\|V\|(\Omega)^2)\,.
\end{equation}
For the first summand on the right-hand side of \eqref{err2 definition}, instead,
\[
|h_\eps(x,\hat V) - h_\eps(g(x),\hat V)| 
\leq |g(x)-x|\|\nabla h_\eps\|_{C^0} \leq \Delta t\|h_\eps\|_{C^0}
\|\nabla h_\eps\|_{C^0}\leq 4\eps^{c_2-6}
\]
by Lemma \ref{l:smoothing estimates}. 
Thus, we can finally estimate
\begin{equation} \label{error2}
\begin{split}
\abs{{\rm Err}_2} &\leq  \Delta t \, \eps^{c_2-2n-14}\, \|\hat V\|(\Omega) \left( \|V\|(\Omega) + \|V\|(\Omega)^2 + \eps^{2n+7} \right) \\ 
&\leq  \Delta t \, \eps^{c_2-2n-14}\, \|\hat V\|(\Omega) \left( \|\hat V'\|(\Omega) + \|\hat V'\|(\Omega)^2 +\eps^{2n+7} \right)\,,
\end{split}
\end{equation}
where $\hat V'=\partial\mathcal E_k$ and we used $\|V\|(\Omega)=\|\partial(f_1)_\star \mathcal E_{k}\|(\Omega)\leq \|\partial\mathcal E_k\|(\Omega)$. 
 Concerning the third integral in \eqref{e:splitting D2}, instead, we use \eqref{Aj zero order} and \eqref{jacobian zero order} to estimate 
\[
{\rm Err}_3 := \int_{E} \left( \phi(x) - \phi(g(x)) \right)\,\left( \mathbf{J}g(x)-1 \right) \, dx
\]
by
\begin{equation} \label{error3}
\abs{{\rm Err}_3} \leq \Delta t \, \eps^{c_2-8}  
\end{equation}
for all $\eps$ sufficiently small. We have thus proved that 
\begin{equation} \label{ready for sum}
\int_{E_{k+1}} \phi(x) \, dx - \int_{E_{k}} \phi(x) \, dx = \Delta t \, \left(  \int_{\pa^*E_{k+1}} \phi(x)\, \nu_{E_{k+1}} \cdot h_{\eps}(x,\pa\E_{k+1}) \, d\Ha^n(x) + {\rm Err}_{k+1}  \right)\,,
\end{equation}
where
\[
{\rm Err}_{k+1} := (\Delta t)^{-1}\, \left( D_1 + {\rm Err}_1 + {\rm Err}_2 + {\rm Err}_3 \right)
\]
satisfy the estimates \eqref{first error}, \eqref{error1}, \eqref{error2}, and \eqref{error3}.

\medskip

\textit{Step three: sum of contributions and limit $\ell \to \infty$.} Now we sum the identities \eqref{ready for sum} for $k = 0, \ldots, k_T-1$. Introducing back the subscripts $j_\ell$ and $i$, and using that (see \eqref{app:future})
\[
\partial \E_{j_\ell}(t) = \partial \E_{j_\ell,k+1} \qquad \mbox{for all $t \in \left( k \Delta t_{j_\ell}, (k+1) \Delta t_{j_\ell} \right]$}\,,
\]
we obtain from \eqref{discretization1} that
\begin{equation} \label{discretization final}
\int_{E_{j_\ell,i}(T)} \phi(x) \, dx - \int_{E_{0,i}} \phi(x) \, dx = \int_0^{k_T\,\Delta t_{j_\ell}}   \int_{\R^{n+1}} \phi(x)\,  h_{\eps_{j_\ell}}(x,\pa \E_{j_\ell}(t)) \cdot d\mu_{E_{j_\ell,i}(t)}(x) \, dt + {\rm Err}_{j_\ell}\,,
\end{equation}
with $\mu_{E_{j_\ell,i}(t)} = -\nabla \chi_{E_{j_\ell,i}(t)} = \nu_{E_{j_\ell,i}(t)} \, \Ha^n\mres_{\pa^* E_{j_\ell,i}(t)}$ and, since $k_T\, \Delta t_{j_\ell} \le T +1$,
\begin{align} 
\abs{{\rm Err}_{j_\ell}} &\leq \frac{1}{j_\ell} \int_0^{T+1} - \frac{\Delta^{vc}_{j_\ell} \| \partial \E_{j_\ell}(t) \| (\Omega)}{\Delta t_{j_\ell}} \, dt - \frac{\Delta^{vc}_{j_\ell}\|\pa\E_0\|(\Omega)}{j_\ell} \nonumber \\
&\quad+ \eps_{j_\ell}^{c_2 - 11} (T+1)  \,+\, \eps_{j_\ell}^{c_2-2n-14}\, (T+1)\, (M_{j_\ell}^2+M_{j_\ell}^3+
\eps^{2n+7}_{j_\ell}), \label{error spacetime} 
\end{align}
where $M_{j_\ell}:=\sup_{t\in[0,T+1]}\|\partial\mathcal E_{j_\ell}(t)\|(\Omega)$. Next, Lemma \ref{l:smoothing estimates} implies that
\[
\left|\int_T^{k_T\Delta t_{j_\ell}} \int_{\R^{n+1}} \phi\,  h_{\eps_{j_\ell}}(\cdot,\pa \E_{j_\ell}(t)) \cdot d\mu_{E_{j_\ell,i}(t)}\right| \leq M_{j_\ell}\,\eps_{j_\ell}^{c_2-3}\,,
\]
so that the end-point $k_T\,\Delta t_{j_\ell}$ in the time integral of \eqref{discretization final} can be replaced by $T$. Letting $\ell \to \infty$ in the identity \eqref{discretization final}, we have that the left-hand side converges to 
\begin{equation}\label{the lhs}
    \int_{E_i (T)} \phi(x)\, dx - \int_{E_{i,0}} \phi(x)\, dx
\end{equation}
due to \eqref{e:partitions convergence}, whereas ${\rm Err}_{j_\ell} \to 0$ thanks to \eqref{error spacetime}, \eqref{e:mass bound} and \eqref{e:mean curvature bounds}. We are then left with studying the limit
\begin{equation} \label{what we are missing}
    \lim_{\ell \to \infty} \int_0^T \int_{\R^{n+1}} \phi(x)\,h_{\eps_{j_\ell}}(x,\pa\E_{j_\ell}(t)) \cdot d\mu_{E_{j_\ell,i}(t)}(x)\,dt\,.
\end{equation}
To this aim, we first observe that, since
\[
h_{\eps}(x, V) = - \Phi_\eps \ast \left(\frac{\Phi_\eps \ast \delta V}{\Phi_\eps \ast \|V\| + \eps \, \Omega^{-1}}\right)(x)\,,
\]
and using that, by definition, $\Omega(x) \le \Omega (y) \,\exp(c_1\,\abs{x-y})$ for all $x,y \in \R^{n+1}$ together with the properties of the kernel $\Phi_\varepsilon$, we can then estimate
\begin{align}
&  \int_{\pa^*E_{j_\ell,i}(t)} \Omega \, \abs{h_{\eps_{j_\ell}}(\cdot,\pa\E_{j_\ell}(t))}^2 \,d\Ha^n \nonumber\\
&\qquad \leq C \int_{\R^{n+1}} \frac{\abs{\Phi_{\eps_{j_\ell}} \ast \delta (\pa\E_{j_\ell}(t))}^2\, \Omega(y)}{\Phi_{\eps_{j_\ell}} \ast \| \pa\E_{j_\ell}(t)\| + \eps_{j_\ell} \Omega^{-1}} \, \frac{\Phi_{\eps_{j_\ell}} \ast (\Ha^n \llcorner_{\pa^*E_{j_\ell,i}(t)})}{\Phi_{\eps_{j_\ell}} \ast \| \pa\E_{j_\ell}(t)\| + \eps_{j_\ell} \Omega^{-1}} \, dy \nonumber\\ \label{uniform L2 estimate}
&\qquad \leq C  \int_{\R^{n+1}} \frac{\abs{\Phi_{\eps_{j_\ell}} \ast \delta (\pa\E_{j_\ell}(t))}^2\, \Omega(y)}{\Phi_{\eps_{j_\ell}} \ast \| \pa\E_{j_\ell}(t)\| + \eps_{j_\ell} \Omega^{-1}} \, dy\,.
\end{align}

Therefore, for every $T>0$ and for every $\phi \in C_c(\R^{n+1}\times\left[0,T\right])$ it holds
\begin{align}
&\left|\int_0^T \int_{\R^{n+1}} \phi(x,t)\, \Omega(x)\, \abs{h_{\eps_{j_\ell}}(x,\pa\E_{j_\ell}(t)) \cdot d\mu_{E_{j_\ell,i}(t)}(x)}\, dt \right| \nonumber\\
&\qquad \leq  \left(\int_0^T \int_{\pa^*E_{j_\ell,i}(t)} \Omega\, \abs{h_{\eps_{j_\ell}}(x,\pa\E_{j_\ell}(t))}^2 \,d\Ha^n\,dt \right)^{\sfrac{1}{2}} \left( \int_0^T \int_{\pa^*E_{j_\ell,i}(t)} \Omega\, \abs{\phi(x,t)}^2\, d\Ha^n\,dt\right)^{\sfrac{1}{2}} \nonumber \\ \label{uniform measure bound}
& \qquad \overset{\eqref{uniform L2 estimate}}{\leq} C\,   \left( \int_0^T \int_{\R^{n+1}} \frac{\abs{\Phi_{\eps_{j_\ell}} \ast \delta (\pa\E_{j_\ell}(t))}^2\, \Omega(y)}{\Phi_{\eps_{j_\ell}} \ast \| \pa\E_{j_\ell}(t)\| + \eps_{j_\ell} \Omega^{-1}} \, dy\right)^{\sfrac{1}{2}}  \left( \int_0^T \int_{\pa^*E_{j_\ell,i}(t)}\Omega \, \abs{\phi(x,t)}^2\, d\Ha^n\,dt \right)^{\sfrac{1}{2}}\,.
\end{align}
In particular, using \eqref{e:mass bound} and \eqref{e:mean curvature bounds}, for any fixed $T > 0$,
\[
\Omega\,h_{\eps_{j_\ell}}(\cdot,\pa\E_{j_\ell}(t)) \cdot d\mu_{E_{j_\ell,i}(t)}\, dt
\]
are (signed) Radon measures in $\R^{n+1} \times \left[0,T\right]$ with uniformly bounded total variation, and we can let $\sigma_i$ denote a subsequential limit as $\ell  \to \infty$. Since $\Ha^n \mres_{\pa^*E_{j_\ell,i}(t)} \leq \| \pa\E_{j_\ell}(t) \|$, and the latter measures converge, as $\ell \to \infty$, to $\|V_t\|$ for every $t \in \R^+$, it is clear from \eqref{uniform measure bound} that $\sigma_i$ is absolutely continuous with respect to the measure $\mu\mres_{\Omega}=\Omega\,d\|V_t\|\,dt$. We will let $u_i$ denote the Radon-Nikod\'ym derivative of $\sigma_i$ with respect to $\mu\mres_{\Omega}$, so that
\begin{equation}\label{who is Vi}
\sigma_{i} = u_i\,\Omega\,d\|V_t\|\,dt
\end{equation}
in the sense of measures. The estimate \eqref{uniform measure bound} also implies that $u_i \in L^2(\mu\mres_{\Omega})$, with 
\begin{equation} \label{L2 bound}
    \|u_i\|_{L^2(\mu\mres_{\Omega})} \leq C\,   \limsup_{\ell \to \infty} \left( \int_0^T \int_{\R^{n+1}} \frac{\abs{\Phi_{\eps_{j_\ell}} \ast \delta (\pa\E_{j_\ell}(t))}^2\, \Omega(y)}{\Phi_{\eps_{j_\ell}} \ast \| \pa\E_{j_\ell}(t)\| + \eps_{j_\ell} \Omega^{-1}} \, dy\right)^{\sfrac{1}{2}}\,.
\end{equation}
Hence, we can now calculate the limit in \eqref{what we are missing} (along the aforementioned subsequence, not relabeled), which gives
\[
\lim_{\ell \to \infty} \int_0^T\int_{\R^{n+1}} (\phi\,\Omega^{-1}) \, \Omega\,h_{\eps_{j_\ell}}(\cdot,\pa\E_{j_\ell}(t)) \cdot d\mu_{E_{j_\ell,i}(t)}\, dt = \int_0^T\int_{\R^{n+1}} \phi\,u_i\,d\|V_t\|\,dt\,.
\]

We have then concluded the existence of $u_i \in L^2 (\mu \mres_{\Omega})$ such that
\begin{equation} \label{evolution rough}
 \int_{E_i (T)}  \phi(x)\, dx - \int_{E_{i,0}} \phi(x)\, dx = \int_0^T\int_{\R^{n+1}} \phi(x)\, u_i(x,t) \, d\|V_t\|(x)\,dt
 \end{equation}
 for all $\phi \in C_c(\R^{n+1})$, that is we have obtained \eqref{e:first version of BV} when $\phi$ does not depend on $t$.

\medskip

\textit{Step four: the case of time-dependent $\phi$.} We now consider the general case of a test function $\phi \in C^1_c(\R^{n+1} \times \left[ 0,T \right])$. The proof is analogous, with a few minor modifications which take the dependence on $t$ into account. First, formula \eqref{discretization1} becomes
\begin{equation}\label{discretization1_time}
    \int_{E(T)} \phi(x,k_T \Delta t) \, dx - \int_{E_0}  \phi(x,0) \, dx = \sum_{k=0}^{k_T-1} \left\lbrace\int_{E_{k+1}} \phi(x,(k+1)\Delta t) \, dx - \int_{E_{k}}  \phi(x,k\Delta t) \, dx\right\rbrace\,,
\end{equation}
whereas formula \eqref{calculation1} becomes
\begin{equation} \label{calculation1_time}
 \int_{E_{k+1}} \phi(x,(k+1)\Delta t) \, dx - \int_{E_{k}}  \phi(x,k\Delta t) \, dx = D_1 + D_2 + D_3\,,
\end{equation}
where
\begin{align} \label{first summand_time}
  D_1 &:= \int_{E^*_{k+1}} \phi(x,k\Delta t) \, dx - \int_{E_{k}}  \phi(x,k\Delta t) \, dx\,, \\ \label{second summand_time}
  D_2 &:= \int_{E_{k+1}} \phi(x,k\Delta t) \, dx - \int_{E^*_{k+1}}  \phi(x,k\Delta t) \, dx\,,\\ \label{third summand_time}
  D_3 &:= \int_{E_{k+1}} \phi (x,(k+1)\Delta t) \, dx - \int_{E_{k+1}} \phi (x, k\Delta t) \, dx\,.
 \end{align}
 By the Fundamental theorem of Calculus and Fubini's theorem we can immediately calculate
 \begin{equation} \label{third summand explicit}
     D_3 = \int_{k\Delta t}^{(k+1)\Delta t} \int_{E_{k+1}} \frac{\partial\phi}{\partial t} (x,t) \, dxdt  \,.
 \end{equation}
The summands $D_1$ and $D_2$ are, instead, treated as in the time-independent case, with $\phi$ replaced by $\phi(\cdot,k\Delta t)$. The identity \eqref{ready for sum} then becomes 
\begin{align}
& \int_{E_{k+1}} \phi(x,(k+1)\Delta t) \, dx - \int_{E_{k}} \phi(x,k\Delta t) \, dx = \int_{k\Delta t}^{(k+1)\Delta t} \int_{E_{k+1}} \frac{\partial\phi}{\partial t} (x,t) \, dxdt \nonumber \\ \label{really ready for sum_time}
& \qquad \qquad \qquad +\Delta t \, \left(  \int_{\pa^*E_{k+1}} \phi(x,k\Delta t)\, \nu_{E_{k+1}} \cdot h_{\eps}(x,\pa\E_{k+1}) \, d\Ha^n(x) + {\rm Err}_{k+1}  \right)\,.
\end{align}
Introducing back the subscripts $_{j_\ell}$ and $_i$, we can then write
\begin{align}
   & \Delta t \int_{\pa^*E_{k+1}} \phi(x,k\Delta t)\, \nu_{E_{k+1}} \cdot h_{\eps}(x,\pa\E_{k+1}) \, d\Ha^n(x) \nonumber \\ \label{time error}
&    \qquad \qquad =\int_{k\Delta t_{j_\ell}}^{(k+1)\Delta t_{j_\ell}} \int_{\pa^*E_{j_\ell,i}(t)} \phi(x,t) \, \nu_{E_{j_\ell,i}(t)} \cdot h_{\eps_{j_\ell}}(x,\pa\E_{j_\ell}(t)) \, d\Ha^n(x)\,dt + \widetilde{{\rm Err}}_{k+1}\,,
\end{align}
where
\begin{equation} \label{time error estimate}
    \abs{\widetilde{{\rm Err}}_{k+1}} \leq \Delta t_{j_\ell}\, \|\partial_t\phi\|_{C^0} \, \left(\max_{\spt \,\phi} \Omega^{-1}\right) \, \int_{k\Delta t_{j_\ell}}^{(k+1)\Delta t_{j_\ell}} \int_{\pa^*E_{j_\ell,i}(t)} \Omega\,\abs{h_{\eps_{j_\ell}}(x,\pa\E_{j_\ell}(t))} \,d\Ha^n(x)dt\,.
\end{equation}
Summing over $k=0,\ldots,k_T-1$, we can then replace \eqref{discretization final} with 
\begin{align} 
    &\int_{E_{j_\ell,i}(T)} \phi(x,k_T\Delta t_{j_\ell}) \, dx - \int_{E_{0,i}} \phi(x,0) \, dx = \int_0^{k_T\Delta t_{j_\ell}}\int_{E_{j_\ell,i}(t)} \frac{\partial\phi}{\partial t}(x,t) \, dx\,dt \nonumber \\ \label{discretization final_time}
   &\qquad \qquad + \int_0^{k_T\Delta t_{j_\ell}}   \int_{\R^{n+1}} \phi(x,t)\,  h_{\eps_{j_\ell}}(x,\pa \E_{j_\ell}(t)) \cdot d\mu_{E_{j_\ell,i}(t)}(x)  \, dt + \widetilde{{\rm Err}}_{j_\ell}\,,
\end{align}
where $\widetilde{{\rm Err}}_{j_\ell}$ contains, with respect to ${\rm Err}_{j_\ell}$, an additional error term which can be estimated by 
\begin{align} 
    \abs{\widetilde{{\rm Err}}_{j_\ell} - {\rm Err}_{j_\ell}} &\leq C\,\eps_{j_\ell}^{c_2-2} \, \|\partial_t\phi\|_{C^0} \, \left( \max_{\spt \,\phi} \Omega^{-1} \right)\, M_{j_\ell} \, (T+1), \label{final time error estimate}
\end{align}
where we have used \eqref{e:h in L infty}. In particular, when $\ell \to \infty$ also $\widetilde{{\rm Err}}_{j_\ell} \to 0$. Next, also in this case we can replace $k_T \Delta t_{j_\ell}$ with $T$ in the integrals on both sides of \eqref{discretization final_time}, paying an additional error term which we estimate by
\[
\left| \int_T^{k_T\Delta t_{j_\ell}} \int_{E_{j_\ell,i}(t)} \frac{\partial\phi}{\partial t} \,dx\,dt \right| \leq C\,\eps_{j_\ell}^{c_2} \, \|\partial_t\phi\|_{C^0} \, \Leb^{n+1}((\spt\,\phi)_1)\,,
\]
and similarly for others, and they all vanish as $\ell\to\infty$.
We can then finally conclude that 
\begin{equation} \label{evolution with time}
\begin{split}
    \int_{E_i (T)} \phi(x,T) \, dx - \int_{E_{i,0}} \phi(x,0) \, dx = & \int_0^T \int_{E_i(t)} \frac{\partial\phi}{\partial t}(x,t) \, dx\,dt \\ &+ \int_0^T \int_{\R^{n+1}} \phi(x,t)\,u_i(x,t) \, d\|V_t\|(x)\,dt\,,
\end{split}
\end{equation}
where $u_i \in L^2(\mu\mres_\Omega)$ is defined by \eqref{who is Vi}.

\medskip

\textit{Step five: support of $u_i$.} To conclude the proof of Proposition \ref{p:existence of a measure derivative}, it only remains to show that $\spt(u_i) \subset \partial S(i)$, where
\[
S(i) = \left\lbrace (x,t)\in \R^{n+1} \times \R^+ \, \colon \, x \in E_i(t)\right\rbrace\,.
\]
Notice that $S(i)$ is open (relatively to $\R^{n+1}\times \R^+$) by \cite[Theorem 3.5(5)]{KimTone}. Suppose now that $(\hat x, \hat t)\notin \pa S(i)$. Then, either $(\hat x, \hat t) \in S(i)$ or $(\hat x, \hat t) \notin {\rm clos}(S(i))$. In the first case, there is a neighborhood $(\hat x, \hat t) \in U \subset \R^{n+1}\times \R^+$ such that $x \in E_i(t)$ for all $(x,t) \in U$. In particular, $(\hat x, \hat t) \notin \spt\mu \cup (\spt\|V_0\| \times \{0\})$, and since $\sigma_i$ is absolutely continuous with respect to $\mu$, also $(\hat x, \hat t) \notin \spt (\sigma_i) = \spt(u_i)$.

\noindent
In the second case, we have that, necessarily, $\hat x \notin {\rm clos}(E_i(\hat t))$. Of course, if $\hat x \in E_{i'}(\hat t)$ for some $i' \neq i$ then $(\hat x,\hat t) \in S(i')$, and the above proof implies again that $(\hat x,\hat t) \notin \spt\mu \cup (\spt\|V_0\| \times \{0\})$, which then gives $(\hat x,\hat t) \notin \spt(u_i)$. Thus, we can assume that $\hat x \in \Gamma(\hat t) \setminus {\rm clos}(E_i(\hat t))$. Since $(\hat x, \hat t) \notin {\rm clos}(S(i))$, there is an open neighborhood $(\hat x,\hat t) \in U \subset \R^{n+1} \times \R^+$ such that $U \cap {\rm clos}(S (i)) = \emptyset$, and we can choose $U$ of the form $U=U_r(\hat x) \times \left[0,r^2\right)$ if $\hat t=0$ or $U = U_r(\hat x) \times \left( \hat t - r^2, \hat t + r^2 \right)$ if $\hat t > 0$. In both cases, if $(x,t) \in U$ then $x \notin E_i(t)$. For any function $\phi \in C^1_c (U)$, the identity \eqref{evolution with time} then gives
\begin{equation} \label{the key identity}
    0 = \iint_U \phi(x,t) \, u_i(x,t) \, d\|V_t\|(x)\,dt = \iint_U \phi(x,t) \, \Omega^{-1}(x)\, d\sigma_i(x,t)\,.
\end{equation}
Since $\phi$ is arbitrary, we deduce then that $\abs{\sigma_i}(U)=0$, and thus that $(\hat x, \hat t) \notin \spt(\sigma_i)=\spt(u_i)$. The proof of the theorem is now complete.
\end{proof}

\subsection{Boundaries move by their mean curvature} \label{ssec:derivative characterization}

In this subsection we complete the proof of Theorem \ref{t:velocity of boundary}, by achieving the representation for the measure-theoretic time derivative $\partial_t \chi_i$ as specified in Remark \ref{rmk:what is BV}. The crucial step is to prove the following lemma, for which the $L^2$ flow property of $\{V_t\}$ plays a pivotal role. We shall denote $\mathbf{p}$ and $\mathbf{q}$ the projections of $\R^{n+1}\times\R$ onto its factors, so that $\mathbf{p}(x,t)=x$ and $\mathbf{q}(x,t)=t$.
\begin{lemma} \label{l:key technical}
For every $i \in \{1,\ldots,N\}$, there is a set $G_i \subset \pa^*S(i)$ with $\Ha^{n+1}(\pa^*S(i)\setminus G_i)=0$ such that the following holds. For every $(x,t) \in G_i$:
\begin{itemize}
    \item[(1)] the tangent $T_{(x,t)}\mu$ exists, and $T_{(x,t)}\mu=T_{(x,t)}(\pa^*S(i))$;
    \item[(2)] \eqref{tangent vector} holds;
    \item[(3)] $x \in \pa^*E_i(t)$, and $T_x\|V_t\| = T_x(\pa^*E_i(t))$;
    \item[(4)] $\mathbf{p}(\nu_{S(i)}(x,t)) \ne 0$, and $\nu_{E_i(t)}(x)=\abs{\mathbf{p}(\nu_{S(i)}(x,t))}^{-1}\,\mathbf{p}(\nu_{S(i)}(x,t))$;
    \item[(5)] $T_x(\pa^*E_i(t)) \times \{0\}$ is a linear subspace of $T_{(x,t)}\mu$.
\end{itemize}
\end{lemma}
The proof of Lemma \ref{l:key technical} can be deduced with relatively little effort from arguments already contained in \cite{Mugnai}, but we include it for the reader's convenience. Before proceeding, we will need some consequences of the celebrated monotonicity formula by Huisken for MCF. Since said consequences will be needed also later on in the paper, we record them here. First, let us set some notation. For $(y,s) \in \R^{n+1} \times \R^+$, we define the $n$-dimensional backwards heat kernel $\rho_{(y,s)}$ by
\begin{equation} \label{Huisken heat}
\rho_{(y,s)}(x,t) := \frac{1}{(4\pi (s-t))^{n/2}}\, \exp\left( - \frac{\abs{x-y}^2}{4(s-t)} \right) \qquad \mbox{for all $t < s$ and $x \in \R^{n+1}$}\,,
\end{equation}
as well as the truncated kernel
\begin{equation}\label{Huisken truncated heat}
    \hat \rho^r_{(y,s)} (x,t) := \eta \left( \frac{x-y}{r} \right) \, \rho_{(y,s)} (x,t)\,,
\end{equation}
where $r > 0$ and $\eta \in C^\infty_c (U_2;\R^+)$ is a radially symmetric function such that $\eta \equiv 1$ on $B_1$, $0 \leq \eta \leq 1$, $\abs{\nabla\eta}\leq 2$ and $\|\nabla^2\eta\| \leq 4$. The following is \cite[Lemma 10.3]{KimTone}, and it is a variant of Huisken's monotonicity formula for MCF.
\begin{lemma} \label{l:monotonicity}
There exists $c(n) > 0$ with the following property. For every $0 \leq t_1 < t_2 < s < \infty$, $y \in \R^{n+1}$ and $r > 0$ it holds
\begin{equation}\label{e:monotonicity}
    \left. \|V_t\| (\hat\rho^r_{(y,s)} (\cdot, t)) \right|_{t=t_1}^{t_2} \leq c(n) \, r^{-2} \, (t_2 - t_1) \, \sup_{t \in \left[ t_1,t_2\right]} r^{-n} \, \|V_t\| (B_{2r} (y))\,. 
\end{equation}
\end{lemma}

As a consequence, one has the following \cite[Lemma 10.4]{KimTone}, which provides a uniform upper bound on mass density ratios.

\begin{lemma} \label{l:upper}
For every $L > 1$ there exists $\Lambda = \Lambda (n,L,\Omega,\|\pa\E_0\| (\Omega)) \in \left( 1, \infty \right)$ such that
\begin{equation} \label{e:upper}
    \sup\left\lbrace r^{-n} \, \|V_t\| (B_r (x)) \, \colon \, r \in \left( 0, 1 \right] \, , x \in B_L (0)\, , t \in [ L^{-1} , L ] \right\rbrace \leq \Lambda \,. 
\end{equation}
\end{lemma}

\begin{proof}[Proof of Lemma \ref{l:key technical}]
Fix $i \in \{1,\ldots,N\}$, and observe that, since $\chi_i \in {\rm BV}_{{\rm loc}}(\R^{n+1} \times \R^+)$ is the characteristic function of $S(i)$, it holds $|\nabla'\chi_i| = \Ha^{n+1} \mres_{\pa^*S(i)}$. We then let $g_i$ be the Radon-Nikod\'ym derivative of $\mu$ with respect to $|\nabla'\chi_i|$, namely
\begin{equation} \label{RN to boundary}
    g_i(x,t) = \frac{d\mu}{d|\nabla'\chi_i|}(x,t):= \lim_{r \to 0^+} \frac{\mu(B^{n+2}_r(x,t))}{|\nabla'\chi_i|(B^{n+2}_r(x,t))}\,,
\end{equation}
where $B^{n+2}_r(x,t)$ is the closed ball with radius $r$ and center $(x,t)$ in $\R^{n+1}\times \R$. By the Lebesgue-Radon-Nikod\'ym theorem, it holds $g_i(x,t) < \infty$ for $\Ha^{n+1}$-a.e. $(x,t) \in \pa^*S(i)$, and
\begin{equation} \label{e:Lebesgue decomposition}
    \mu = g_i \, \Ha^{n+1}\mres_{\pa^*S(i)} + \mu \mres_{\Sigma_i} \,\, \mbox{ for a set $\Sigma_i$ with $\Ha^{n+1}(\pa^*S(i) \cap \Sigma_i)=0$}\,.
\end{equation}
On the other hand, it is not difficult to see that 
\begin{equation} \label{abs cont}
\mu \ll \Ha^{n+1}\,.
\end{equation}
Indeed, first notice that, as an immediate corollary of Lemma \ref{l:upper}, one has
\begin{equation} \label{e:density upper bound for mu}
    \Theta^{* n+1}(\mu, (x,t)) := \limsup_{r \to 0^+} \frac{\mu(B^{n+2}_r(x,t))}{\omega_{n+1}r^{n+1}} < \infty \qquad \mbox{for every $(x,t) \in \R^{n+1}\times\left(0,\infty\right)$}\,.
\end{equation}
Let then $A$ be a set with $\Ha^{n+1}(A)=0$, and set, for $Q \in \mathbb N$:
\[
D_Q:= \left\lbrace (x,t) \in \R^{n+1}\times \left(0,\infty\right) \,\colon\, \Theta^{* n+1}(\mu,(x,t)) \leq Q \right\rbrace\,.
\]
Then, by \cite[Theorem 3.2]{Simon}, one has
\[
\mu (D_Q \cap A)\leq 2^{n+1}\,Q \, \Ha^{n+1}(D_Q \cap A) =0\,,
\]
so that the conclusion $\mu (A) = 0$ follows because 
$\R^{n+1}\times(0,\infty)=\bigcup_{Q \in \mathbb N} D_Q$
due to \eqref{e:density upper bound for mu}.\\
Combining \eqref{e:Lebesgue decomposition} with \eqref{abs cont}, we immediately have that
\begin{equation} \label{e:key fact for char}
    \mu \mres_{\pa^*S(i)} = g_i \, \Ha^{n+1} \mres_{\pa^*S(i)}\,.
\end{equation}
Similarly, by taking the Radon-Nikod\'ym derivative of $|\nabla'\chi_i|$ with respect to $\mu$ we see that
\begin{equation} \label{RN to boundary 2}
\frac{d|\nabla'\chi_i|}{d\mu}(x,t) \quad \mbox{is finite for $\mu$-a.e. $(x,t)$}\,.
\end{equation}
Since $|\nabla'\chi_i| \ll \mu$ by Theorem \ref{main3}(4) and Corollary \ref{cor:chi is BV}, this is also finite for $|\nabla'\chi_i|$-a.e.~and 
\eqref{RN to boundary} shows that this is equal to $1/g(x,t)$, so that, in particular,
$g(x,t)>0$ for $\Ha^{n+1}$-a.e. $(x,t) \in \pa^*S(i)$. Together with \eqref{e:key fact for char}, the latter information implies that $\mu \mres_{\pa^*S(i)}$ is an $(n+1)$-rectifiable measure, so that $T_{(x,t)}(\mu\mres_{\pa^*S(i)})$ exists and is equal to $T_{(x,t)}(\pa^*S(i))$ for $\Ha^{n+1}$-a.e. $(x,t) \in \pa^*S(i)$. On the other hand, by \cite[Theorem 3.5]{Simon}, for $\Ha^{n+1}$-a.e. $(x,t) \in \pa^*S(i)$ it holds
\begin{align} \label{e:classical fact}
    \lim_{r \to 0^+} \frac{\mu(B^{n+2}_r(x,t) \setminus \pa^*S(i) )}{r^{n+1}} =0\,. 
\end{align}
If $(x_0,t_0) \in \pa^*S(i)$ is such that \eqref{e:classical fact} holds and $\phi \in C_c(B^{n+2}_1)$ is arbitrary then
\begin{equation} \label{with or without you}
\begin{split}
   & \limsup_{r \to 0^+} \left| \int_{(\R^{n+1}\times\R^+)\setminus\pa^*S(i)} \phi(r^{-1}(x-x_0,t-t_0)) \, r^{-(n+1)} \, d\mu(x,t) \right| \\
   &\qquad \leq \|\phi\|_{C^0}  \, \lim_{r \to 0^+} \frac{\mu(B^{n+2}_r(x_0,t_0) \setminus \pa^*S(i) )}{r^{n+1}}  = 0\,.
\end{split}
\end{equation}
This shows that
\[
\begin{split}
&\lim_{r\to 0^+} \int \phi(r^{-1}(x-x_0,t-t_0)) \, r^{-(n+1)} \, d\mu(x,t) \\ & \qquad\qquad\qquad = \lim_{r\to 0^+} \int_{\pa^*S(i)} \phi(r^{-1}(x-x_0,t-t_0)) \, r^{-(n+1)} \, d\mu(x,t) \\
& \qquad\qquad\qquad = g_i(x_0,t_0) \, \int_{T_{(x_0,t_0)}(\pa^*S(i))} \phi(x,t) \, d\Ha^{n+1}(x,t)
\end{split}
\]
for every $\phi \in C_c(\R^{n+1}\times\R)$ and for $\Ha^{n+1}$-a.e. $(x_0,t_0) \in \pa^*S(i)$. This completes the proof of (1), whereas (2) is an immediate consequence of (1) and Corollary \ref{cor:tangent vector}, since, on $\pa^* S(i)$, $\mu$ and $\Ha^{n+1}$ have the same null sets by \eqref{e:key fact for char}. 

\smallskip

Next, to prove (3) and (4), we are going to use the coarea formula for sets of finite perimeter, slicing $\pa^*S(i)$ according to hyperplanes $\{\mathbf{q}=t\}$. Set, for $t \in \R^+$, $(\pa^*S(i))_t := \pa^*S(i)\cap\{\mathbf{q}=t\} = \{x \, \colon \, (x,t)\in \pa^*S(i)\}$. Then, using e.g. \cite[Theorem 18.11]{Maggi_book} together with the fact that $S(i) \cap \{\mathbf{q}=t\}=E_i(t)$ we have that 
\begin{equation}\label{e:sofp_slicing}
    \Ha^n ((\pa^*S(i))_t\,\Delta\,\pa^*E_i(t))=0
\end{equation}
for a.e. $t\in\R^+$; moreover, for such $t$
\begin{align} 
      \label{e:slicing_normal}
    &\mathbf{p}(\nu_{S(i)}(x,t)) \neq 0 \,, \\ \label{e:slicing measure}
   & \nu_{E_i(t)}(x) = \frac{\mathbf{p}(\nu_{S(i)}(x,t))}{\abs{\mathbf{p}(\nu_{S(i)}(x,t))}}\,,
\end{align}
for $\Ha^n$-a.e. $x\in(\pa^*S(i))_t$. Let 
\[
Z:=\left\lbrace t\in\R^+\,\colon\, \eqref{e:sofp_slicing}\,\mbox{fails}\right\rbrace\,,
\]
and for every $t\in\R^+$ set
\[
Z_t := \left\lbrace x\in(\pa^*S(i))_t \,\colon\, x \notin \pa^*E_i(t) \, \mbox{or \eqref{e:slicing_normal}-\eqref{e:slicing measure} fail} \right\rbrace\,,
\]
so that $\Leb^1(Z)=0$ and, for every $t\notin Z$, $\Ha^n(Z_t)=0$. Consider then the Borel function $\kappa(x,t):=\chi_{Z_t}(x)$ on $\R^{n+1}\times\R^+$. We have:
\[
\begin{split}
    0 &= \int_0^\infty \chi_{\R^+\setminus Z}(t)\, \Ha^n(Z_t)\,dt = \int_0^\infty  \Ha^n(Z_t) \, dt\\
    &= \int_0^\infty \int_{\pa^*S(i) \cap \{\mathbf{q}=t\}} \kappa(x,t)\, d\Ha^n(x)\,dt \\
    &=\int_{\pa^*S(i)} \kappa(x,t) \, \left|\nabla^{\pa^*S(i)}\mathbf{q}(x,t)\right| d\Ha^{n+1}(x,t)\,,
\end{split}
\]
where in the last identity we have used the coarea formula \cite[Eq. (2.72)]{AFP}, and where $\nabla^{\pa^*S(i)}\mathbf{q}(x,t)$ is the tangential gradient, along $\pa^*S(i)$, of $\mathbf{q}$ at $(x,t)$. Combining now (1) and Corollary \ref{cor:tangent vector}, we see that
\begin{equation} \label{non-trivial time projection}
    \begin{pmatrix}
    h(x,V_t) \\
    1
    \end{pmatrix}
    \in T_{(x,t)}(\pa^*S(i)) \qquad \mbox{at $\Ha^{n+1}$-a.e. $(x,t) \in \pa^*S(i)$}\,,
\end{equation}
which readily implies that $\left|\nabla^{\pa^*S(i)}\mathbf{q}(x,t)\right|>0$ for $\Ha^{n+1}$-a.e. $(x,t) \in \pa^*S(i)$. Hence, it must be $\kappa(x,t)=0$ for $\Ha^{n+1}$-a.e. $(x,t)\in\pa^*S(i)$, thus proving the first part of (3) and (4). At such points, the identity $T_x\|V_t\|=T_x(\pa^*E_i(t))$ is then obtained by repeating the argument in (1) at fixed $t$. 

\smallskip

Finally, (5) is always true at points satisfying (1)-(4): indeed, if $\xi \in T_x(\pa^*E_i(t))$ then
\[
\begin{pmatrix}\xi\\0\end{pmatrix} \cdot \nu_{S(i)}(x,t) = \xi \cdot \mathbf{p}(\nu_{S(i)}(x,t)) = \abs{\mathbf{p}(\nu_{S(i)}(x,t))} \, \xi \cdot \nu_{E_i(t)}(x)=0\,.
\]
This completes the proof.
\end{proof}

\begin{proof}[Proof of Theorem \ref{t:velocity of boundary}]
By virtue of Proposition \ref{p:existence of a measure derivative}, we only need to show the validity of \eqref{e:the BV derivative}. Fix $i\in \{1,\ldots, N\}$.
Using that $S(i)$ is a set of locally finite perimeter in $\R^{n+1}\times\R^+$, for any $0 \le t_1 < t_2 < \infty$ and any $\phi \in C^1_c(\R^{n+1} \times \left( t_1, t_2 \right))$ we have that
\begin{equation} \label{e:characterization1}
\begin{split}
\int_{\R^{n+1}\times\R^+} \frac{\partial\phi}{\partial t}(x,t) \, \chi_i(x,t) \, dxdt &= \int_{S(i)} \frac{\partial\phi}{\partial t}(x,t)  \, dxdt \\ &= \int_{\pa^*S(i)} \phi(x,t) \, \mathbf{q}(\nu_{S(i)}(x,t)) \, d\Ha^{n+1}(x,t)\,.
\end{split}
\end{equation}

Let $G_i$ be the set of Lemma \ref{l:key technical}. For every $(x,t) \in G_i$, we have
\begin{equation} \label{e:the tangent}
    T_{(x,t)}\mu = \left( T_x(\pa^*E_i(t)) \times \{0\} \right) \oplus {\rm span} \begin{pmatrix}
    h(x,V_t) \\
    1
    \end{pmatrix}\,,
\end{equation}
and thus
\begin{equation} \label{e:the normal}
    \nu_{S(i)}(x,t) = \frac{1}{\sqrt{1 + \abs{h(x,V_t)}^2}}
    \begin{pmatrix}
     \nu_{E_i(t)}(x) \\
     - h(x,V_t) \cdot \nu_{E_i(t)}(x)
    \end{pmatrix}\,,
\end{equation}
where we have used that $h(x,V_t) \perp T_x\|V_t\|=T_x(\pa^*E_i(t))$. In particular, since $\Ha^{n+1}(\pa^*S(i)\setminus G_i)=0$, \eqref{e:characterization1} yields
\begin{equation} \label{e:characterization4}
\begin{split}
   & \int_{\R^{n+1}\times\R^+} \frac{\partial\phi}{\partial t}(x,t) \, \chi_i(x,t) \, dxdt \\
   &\qquad\qquad= - \int_{G_i} \phi(x,t) \, h(x,V_t) \cdot \nu_{E_i(t)}(x) \, \frac{1}{\sqrt{1+\abs{h(x,V_t)}^2}} \,d\Ha^{n+1}(x,t)\,.
   \end{split}
\end{equation}
Due to \eqref{e:the normal}, at points $(x,t)\in G_i$ the coarea factor with respect to the slicing via $\mathbf{q}$ is precisely
\[
\abs{\nabla^{\pa^*S(i)}\mathbf{q}(x,t)} = \frac{1}{\sqrt{1+\abs{h(x,V_t)}^2}}\,,
\]
so that the chain of identities in \eqref{e:characterization4} can be continued as
\begin{equation} \label{e:characterization5}
\begin{split}
    &= - \int_{0}^{\infty}\int_{G_i \cap \{\mathbf{q}=t\}} \phi(x,t) \, h(x,V_t) \cdot \nu_{E_i(t)}(x) \, d\Ha^n(x)\,dt \\
    &=- \int_{0}^{\infty}\int_{\pa^*E_i(t)} \phi(x,t) \, h(x,V_t) \cdot \nu_{E_i(t)}(x) \, d\Ha^n(x)\,dt\,,
 \end{split}
\end{equation}
again by the coarea formula. This completes the proof.
\end{proof}

\section{Ilmanen's density lower bound} \label{s:Ilmanen}

This section contains a proof of the following fact (Theorem \ref{t:ilm_lb} below),
from which we may conclude the proofs of Theorem \ref{main2}(2)-(4): if $\{V_t\}_{t \in \R^+}$ is as in Proposition \ref{p:properties and limit}, then there exists a threshold $\theta_0 > 0$ such that, at ``most'' points (in a sense to be suitably specified) on the support of the evolving varifolds the mass density is not smaller than $\theta_0$. 
Such density lower bound is already stated in the work of Ilmanen \cite[Section 7.1]{Ilm_AC}, and we include here its proof only for the sake of completeness. 

\medskip

Using Lemma \ref{l:monotonicity} and Lemma \ref{l:upper}, we prove next the following lemma, which is a variant of Brakke's clearing out lemma (cf. \cite[Section 6.3]{Brakke}, \cite[Section 6.1]{Ilm_AC} and \cite[Lemma 10.6]{KimTone}). 
Recalling the notation set in \eqref{Huisken heat}-\eqref{Huisken truncated heat}, we define, for $\delta > 0$,
\begin{equation} \label{r delta kernel}
    \hat \rho^{r,\delta}_{y} (x) := \hat\rho^r_{(y,t+\delta\,r^2)}(x,t) = \eta\left(\frac{x-y}{r} \right) \, \frac{1}{(4\pi\delta r^2)^{n/2}} \, \exp\left( - \frac{\abs{x-y}^2}{4\delta r^2}\right)\,.
\end{equation}

\begin{lemma} \label{l:col}
For any $L > 1$ there exists $\delta_0 = \delta_0 (n,L,\Omega,\|\pa\E_0\| (\Omega)) \in \left( 0, 1 \right)$ with the following property. If $(y,t) \in B_L (0) \times [L^{-1},L]$ and $r \in \left( 0, \frac12 \right)$ are such that
\begin{equation} \label{col:hp}
    \|V_t\| (\hat\rho^{r,\delta_0}_y) < \frac12\,,
\end{equation}
then $(y,t+\delta_0 r^2) \notin \spt \mu$.
\end{lemma}

\begin{proof}
For $\delta_0 \in \left(0,1\right)$ to be specified later, suppose towards a contradiction that $(y,t + \delta_0 r^2) \in \spt \mu$. Then, by \cite[Lemma 10.1(2)]{KimTone} there is a sequence $(y_i, t_i)$ converging to $(y,t+\delta_0 r^2)$ such that $V_{t_i} \in {\bf IV}_n (\R^{n+1})$ and $y_i \in \spt \|V_{t_i}\|$. Since $V_{t_i}$ is an integral varifold $\var (M_{t_i}, \theta_{t_i})$, in any arbitrarily small neighborhood of a point $x \in \spt\|V_{t_i}\|$ there is a point $\tilde x \in M_{t_i}$ where $M_{t_i}$ has an approximate tangent plane ${\rm Tan}(M_{t_i}, \tilde x)$ and the density $\Theta_{V_{t_i}}(\tilde x) = \theta_{t_i}(\tilde x)$ is an integer. Thus, we can assume without loss of generality that ${\rm Tan}(M_{t_i}, y_i)$ exists and $\Theta_{V_{t_i}}(y_i) \geq 1$. Assume also without loss of generality that $t_i > t$ for all $i$, and fix $\varepsilon > 0$. Then, apply Lemma \ref{l:monotonicity} with $t_1 = t$, $t_2 = t_i$, $s = t_i + \varepsilon$, and $y = y_i$ to get
\begin{equation} \label{col1}
    \left. \|V_s\| (\hat \rho^r_{(y_i,t_i + \varepsilon)} (\cdot, s)) \right|_{s=t}^{t_i} \leq c(n) \, r^{-2} \, (t_i - t) \, \sup_{s \in \left[ t, t_i \right]} r^{-n} \|V_s\| (B_{2r}(y_i))\,.
\end{equation}
Consider first the term
\begin{eqnarray*}
\|V_{t_i}\| (\hat \rho^r_{(y_i, t_i + \varepsilon)}(\cdot, t_i)) &=& \int_{M_{t_i}} \eta \left(\frac{x-y_i}{r} \right) \, \rho_{(y_i,\varepsilon)} (x,0) \, \theta_{t_i}(x) \, d\Ha^n(x) \\ 
&\geq& \int_{M_{t_i} \cap B_{r} (y_i)} \rho_{(y_i,\varepsilon)} (x,0) \, \theta_{t_i}(x) \, d\Ha^n(x)\,.
\end{eqnarray*}
By changing variable $z = \frac{x-y_i}{\sqrt{\varepsilon}}$ in the integral on the right-hand side, it is not difficult to see that, in the limit as $\varepsilon \to 0^+$, the latter converges to 
\[
\theta_{t_i} (y_i) \, \int_{{\rm Tan}(M_{t_i},y_i)} \rho_{(0,1)} (z,0) \, d\Ha^n (z) = \theta_{t_i} (y_i) \geq 1\,.
\]
We can then conclude
\begin{equation} \label{col2}
    1 \leq \|V_t\| (\hat \rho^r_{(y_i,t_i)}(\cdot, t)) + c(n) \, r^{-2} \, (t_i - t) \, \sup_{s \in \left[ t, t_i \right]} r^{-n} \|V_s\| (B_{2r}(y_i))\,.
\end{equation}
We let then $i \to \infty$: using that $(y_i,t_i) \to (y, t + \delta_0 r^2)$, and recalling \eqref{r delta kernel} we get from \eqref{e:upper}
\begin{eqnarray} 
    1 &\leq& \|V_t\| (\hat\rho^{r,\delta_0}_y) + c(n) \,\delta_0 \, \sup_{s \in \left[ t, t+\delta_0 r^2 \right]} r^{-n} \|V_s\| (B_{2r}(y)) \nonumber \\ \label{col3}
    &\leq& \|V_t\| (\hat\rho^{r,\delta_0}_y) + c(n)\, \delta_0 \, \Lambda (n,L+1,\Omega,\|\pa\E_0\| (\Omega))\,.
\end{eqnarray}
Choosing $\delta_0$ such that the second summand in the right-hand side of \eqref{col3} is $\leq \frac12$ leads to a contradiction with \eqref{col:hp}.
\end{proof}

\begin{remark} \label{rmk:kernel density}
Observe that, since
\begin{equation*} \label{est kernel density}
\| V_t \| (\hat \rho^{r,\delta_0}_y) \leq (4\pi\delta_0)^{-n/2} \, r^{-n} \, \|V_t\| (U_{2r} (y))\,,
\end{equation*}
Lemma \ref{l:col} immediately implies the following: if $(y,t) \in B_L (0) \times [L^{-1},L]$ is such that, for some $r \in \left( 0, 1 \right)$,
\begin{equation} \label{col density:hp}
    r^{-n} \, \|V_t\| (U_{r} (y)) < \theta_0 := \frac{(4\pi\delta_0)^{n/2}}{2^{n+1}}\,,
\end{equation}
then $(y,t + 4^{-1} \delta_0 r^2) \notin \spt\mu$.
\end{remark}

\begin{theorem} \label{t:ilm_lb}

For $L > 1$, let $\theta_0 = \theta_0 (n,L+1,\Omega,\|\pa\E_0\|(\Omega))$ be the number defined in \eqref{col density:hp}, and define, for $t \geq 0$, the sets
\begin{eqnarray*}
\mathcal{Z}^0 &:=& \left\lbrace (y,t) \in \spt\mu \cap (B_L (0) \times [L^{-1},L]) \, \colon \, \limsup_{r \to 0^+} r^{-n} \, \|V_t\| (U_r (y)) < \theta_0 \right\rbrace\,,\\
\mathcal{Z}^0_t &:=& \mathcal{Z}^0 \cap (\R^{n+1} \times \{t\})\,.
\end{eqnarray*}
Then, there exists $G \subset \R^+$ with $\Leb^1(G)=0$ such that $\Ha^{n-1+\alpha}(\mathcal Z^0_t) = 0$ for every $\alpha > 0$ for every $t \in \R^+ \setminus G$.
\end{theorem}

\begin{proof}

For every $\theta < \theta_0$, and for every $\sigma \in \left( 0,1 \right)$, define
\[
\mathcal Z^0_{\theta,\sigma} := \left\lbrace (y,t) \in \spt\mu \cap (B_L (0) \times [L^{-1},L]) \, \colon \, r^{-n} \, \|V_t\| (U_r (y)) < \theta \quad \mbox{for every $r \in \left( 0 , \sigma \right)$} \right\rbrace\,,
\]
so that
\begin{equation} \label{ilm spacchetto}
\mathcal{Z}^0 = \bigcup_{\theta < \theta_0\,, \sigma \in \left( 0, 1 \right)} \mathcal Z^0_{\theta,\sigma}\,.
\end{equation}

Let $(y,t) \in \mathcal Z^0_{\theta,\sigma}$. For any $t' \in \left( t , t + 4^{-2}\delta_0 \sigma^2 \right]$, let $r>0$ be such that $r^2 = 4\, \delta_0^{-1} \, (t'-t)$ (so that necessarily $r \in \left( 0, \sigma/2 \right]$), and let $y' \in B_{\gamma r} (y)$ (with $\gamma \in \left(0,1\right)$ to be specified) be arbitrary. It holds then $(y',t) \in B_{L+1} (0) \times [(L+1)^{-1}, L+1]$. Furthermore, since $U_{r}(y') \subset U_{(1+\gamma)r}(y)$, and since $(1+\gamma)r < 2r < \sigma$, it holds
\begin{equation} \label{shifted point}
   r^{-n}\, \|V_t\|(U_r(y')) < (1+\gamma)^n\,\theta\,.
\end{equation}
Hence, for $\gamma$ small enough (depending on the ratio $\theta_0/\theta$) it holds $r^{-n}\, \|V_t\|(U_r(y'))< \theta_0$, so that Remark \ref{rmk:kernel density} implies that $(y',t+4^{-1}\delta_0 r^2) = (y',t') \notin \spt\mu$. In particular, $(y',t') \notin \mathcal Z^0_{\theta,\sigma}$. Analogously, if $t' \in \left[t-4^{-2}\delta_0\sigma^2 ,t\right)$ and $y' \in B_{\gamma r}(y)$ for $r^2 = 4\delta_0^{-1}(t-t')$ and $\gamma$ small enough, we have that necessarily $(y',t') \notin \mathcal Z^0_{\theta,\sigma}$. For otherwise, it would be $\left( (1+\gamma) r \right)^{-n}\,\|V_{t'}\| (U_{(1+\gamma)r} (y')) < \theta$, and thus
\begin{equation} \label{shifted point backwards}
    r^{-n} \, \|V_{t'}\| (U_r (y)) < (1+\gamma)^n\,\theta < \theta_0\,,
\end{equation}
which would then imply $(y,t'+4^{-1}\delta_0 r^2) = (y,t) \notin \spt\mu$, a contradiction.

We have then concluded the following dichotomy:
\begin{equation} \label{dichotomy}
    \begin{split}
        & \mbox{either $(y,t) \notin \mathcal Z^0_{\theta,\sigma}$}\,, \\
        & \mbox{or $(y',t') \notin \mathcal Z^0_{\theta,\sigma}$ whenever $0<\abs{t-t'}\leq 4^{-2}\delta_0\sigma^2$ and $\abs{y'-y}^2 \leq 4\,\gamma^2\,\delta_0^{-1}\,\abs{t-t'}$}\,.
    \end{split}
\end{equation}
Therefore, if $(y,t) \in \mathcal Z^0_{\theta,\sigma}$ then the truncated double paraboloid 
\begin{equation} \label{paraboloid}
\mathcal P(y,t) := \{\abs{y'-y}^2 \leq 4\,\gamma^2\,\delta_0^{-1} \, \abs{t-t'} \leq 4^{-1}\,\gamma^2\,\sigma^2\}
\end{equation}
intersects $\mathcal Z^0_{\theta,\sigma}$ only in $(y,t)$. Next, setting $2\,\tau := 4^{-1}\,\gamma^2\,\sigma^2$, we consider sets
\[
\mathcal Z^0_{\theta,\sigma,y_0,t_0} := \mathcal Z^0_{\theta,\sigma} \cap \left( B_1 (y_0) \times \left[ t_0 - \tau, t_0 + \tau \right] \right)\,, \qquad (y_0 \in B_L (0)\,, t_0 \in [L^{-1},L ])\,,  
\]
so that a countable union of such sets covers $\mathcal Z^0_{\theta,\sigma}$. Fix any such set, and call it $\mathcal Z'$ for the sake of simplicity: it will suffice to show that, setting $\mathcal Z_t' := \mathcal Z' \cap (\R^{n+1} \times \{t\})$, it holds $\Ha^{n-1+\alpha}(\mathcal Z_t') = 0$ for every $\alpha > 0$ and a.e.~$t \geq 0$. Notice that if $(y,t) \in \mathcal Z'$ then, by the definition of $\tau$, the set $\mathcal Z' \cap \left( \{y\} \times \R \right)$ is contained in $\mathcal P (y,t)$: in particular, for $y \in B_1(y_0)$ the fiber $\{y\}\times\R$ intersects $\mathcal Z'$ in at most one point. Let $\mathbf p$ be the coordinate projection $\mathbf p (x,t) = x$, fix $\delta > 0$, and cover the set $\mathbf p(\mathcal Z') \subset B_1(y_0)$ by countably many open balls $U_{r_i}(y_i)$ so that
\begin{equation} \label{covering}
  r_i \leq \delta \,, \qquad  \sum_i \omega_{n+1}\,r^{n+1}_i \leq 2\,\Leb^{n+1} (B_1 (y_0))\,.
\end{equation}

For every center $y_i$ of the balls in the covering, let $t_i$ be the only point such that $(y_i,t_i) \in \mathcal Z'$, and notice that, as a consequence of the first part of the proof, if $(y,t) \in \mathcal Z'$ with $y \in U_{r_i}(y_i)$ then necessarily $\abs{t-t_i} < 4^{-1}\gamma^{-2}\delta_0 r_i^2$. In other words, for $\delta$ suitably small the cylinders $U_{r_i}(y_i) \times \left( t_i - 4^{-1}\gamma^{-2}\delta_0 r_i^2, t_i + 4^{-1}\gamma^{-2}\delta_0 r_i^2 \right)$ are a covering of $\mathcal Z'$. We can then estimate
\begin{eqnarray*}
    \int_{t_0-\tau}^{t_0+\tau} \Ha^{n-1+\alpha}_{\delta}(\mathcal Z_t') \, dt &\leq& \int_{t_0-\tau}^{t_0+\tau} \sum_{i\,\colon\, \abs{t-t_i} < 4^{-1}\gamma^{-2}\delta_0 r_i^2} \omega_{n-1+\alpha} \, r_{i}^{n-1+\alpha} \, dt \\ 
    &\leq& \sum_i \int_{t_i - 4^{-1}\gamma^{-2}\delta_0 r_i^2}^{t_i + 4^{-1}\gamma^{-2}\delta_0 r_i^2} \omega_{n-1+\alpha} \, r_i^{n-1+\alpha} \, dt \\
    &\leq& C(n,\alpha)\,\gamma^{-2}\, \delta_0\,\delta^\alpha\,,
\end{eqnarray*}
where we have used \eqref{covering}. Letting $\delta \to 0^+$, we find then 
\begin{equation} \label{conclusion ilm}
    \int_{t_0-\tau}^{t_0+\tau} \Ha^{n-1+\alpha} (\mathcal Z_t') \, dt = 0
\end{equation}
by monotone convergence, and hence, by taking countable unions,
\begin{equation} \label{final ilm}
    \int_{0}^\infty \Ha^{n-1+\alpha} ((\mathcal Z^0_{\theta,\sigma})_t) \, dt = 0\,, 
\end{equation}
with the obvious meaning of the symbols, and taking into account that $\mathcal Z_t'$ is empty when $t < L^{-1}$ or $t > L$. The conclusion follows from \eqref{ilm spacchetto}.
\end{proof}

The following is the immediate corollary of Theorem \ref{t:ilm_lb}, which proves
Theorem \ref{main2}(2). 

\begin{corollary} \label{cor:global_lb}
There exists $G \subset \R^+$ with $\Leb^1 (G) = 0$ such that for every $t \in \R^+ \setminus G$:
\begin{equation} \label{few zero density}
    \Ha^{n-1+\alpha}\left(\left\lbrace y \in \R^{n+1} \, \colon \, (y,t) \in \spt\mu \quad \mbox{and} \quad \lim_{r\to 0^+} r^{-n}\, \|V_t\| (U_r (y)) = 0\right\rbrace\right) = 0 \qquad \forall\,\alpha > 0\,.
\end{equation}
In particular, recalling from Theorem \ref{main2}(1) and Theorem \ref{main3}(3) that $\spt\|V_t\| \subset \{x \, \colon \, (x,t) \in \spt\,\mu\} = \Gamma (t)$, the set $\Gamma(t) = \R^{n+1} \setminus \bigcup_{i=1}^N E_i (t) = \bigcup_{i=1}^N \partial E_i (t)$ is $\Ha^n$-equivalent to $\spt \|V_t\|$ for a.e. $t \geq 0$, and in fact 
\begin{equation} \label{final conclusion ilmanen}
\dim_{\Ha}(\Gamma(t) \setminus \spt\|V_t\|) \leq n-1 \qquad \mbox{for a.e. $t \geq 0$}\,.
\end{equation}
\end{corollary}
Since $V_t\in {\bf IV}_n(\mathbb R^{n+1})$ for 
a.e.~$t\geq 0$, there exists a countably $n$-rectifiable
set $\tilde \Gamma(t)$ such that $V_t={\bf var}
(\tilde\Gamma(t),\theta_t)$ and $\theta_t(x)=\Theta^n(\|V_t\|,x)$. By definition,
we have $\mathcal H^n(\tilde\Gamma(t)\setminus{\rm spt}\|V_t\| )=0$ and by \cite[Theorem 3.5]{Simon}, $\Theta^{n}(\|V_t\|,x)=0$
for $\mathcal H^n$-a.e.~$x\in {\rm spt}\|V_t\|\setminus\tilde \Gamma(t)$.
The latter claim shows ${\rm spt}\|V_t\|\setminus\tilde \Gamma(t)$ is included in the set appearing in \eqref{few zero density}, and we can conclude that $\tilde\Gamma(t)$ is $\mathcal H^n$-equivalent to ${\rm spt}\|V_t\|$. This proves Theorem \ref{main2}(3)(4). 

\section{Two-sidedness at unit density point} \label{s:two sided}
In this section, we prove Theorem \ref{main3}(7)(8), which are
restated in Prposition \ref{lay2} and \ref{evenodd}, respectively. To do so, we 
need to analyze the behavior of approximating flows. The first Lemma
\ref{two} shows that, for a.e.~$t$, only 
the reduced boundaries of approximating grains contribute to the limit
measure and that the measures coming from ``interior boundaries'' $\partial E_{j_\ell,k}(t)
\setminus\partial^* E_{j_\ell,k}(t)$ vanish
in the limit. Roughly speaking, this is due to the measure minimizing property 
in the length scale of $o(1/j^2)$ 
which has the effect of eliminating the interior boundaries. 
\begin{lemma}\label{two}
For a.e.~$t\in\R^+$ and ${\mathcal E}_{j_\ell}(t)=\{E_{j_\ell,k}(t)\}_{k=1}^N$ in Theorem \ref{p:properties and limit}, we have 
\begin{equation}\label{bcv0}
\lim_{\ell \rightarrow\infty}\sum_{k=1}^N \|\partial^* E_{j_\ell,k}(t)\|=2\mu_t.
\end{equation}
\end{lemma}
\begin{proof}
We fix $t$ such that 
\begin{equation}\label{bcv4}
\lim_{\ell\rightarrow\infty}\frac{\Delta_{j_\ell}^{vc}\|\partial\mathcal E_{j_\ell}(t)\|(\Omega)}{j_\ell \Delta t_{j_\ell}}
=0,
\end{equation}
which holds for a.e.~$t\in\R^+$ due to \eqref{e:mean curvature bounds}
and drop $t$ for simplicity in the following. 
It is sufficient to prove that $$\lim_{\ell\rightarrow\infty} \Ha^n(U_R\cap \spt\|\partial\mathcal E_{j_\ell}\|
\setminus \cup_{k=1}^N
\partial^* E_{j_\ell,k}))=0$$ for arbitrary $R\geq 1$ since $\spt\|\partial\mathcal E_{j_\ell}\|\subset
\cup_{k=1}^N\partial E_{j_\ell,k}$, $\Ha^n(\cup_{k=1}^N\partial E_{j_\ell,k}\setminus \spt\|\partial\mathcal E_{j_\ell}\|)=0$ and 
$$2\|\partial\mathcal E_{j_\ell}\|=\sum_{k=1}^N\|\partial^* E_{j_\ell,k}\|
+2\Ha^n\mres_{\spt\|\partial\mathcal E_{j_\ell}\|\setminus \cup_{k=1}^N
\partial^* E_{j_\ell,k}}.$$
As in \cite[Section 4]{KimTone2}, let $r_\ell:=1/(j_\ell)^{2.5}$. Define 
$$
Z_\ell:=\{x\in U_R\cap \spt\|\partial\mathcal E_{j_\ell}\|\,:\, \|\partial\mathcal E_{j_\ell}\|(B_{r_\ell}(x)) 
\leq c_3 r_\ell^n\}\mbox{ and }Z_\ell^c:=U_R\cap \spt\|\partial\mathcal E_{j_\ell}\|\setminus Z_\ell,
$$
where $c_3$ is the same constant appearing in \cite[Proposition 7.2]{KimTone} 
and apply it to each ball $B_{r_\ell}(\hat x)$ with $\hat x\in Z_\ell$ 
to estimate $\|\partial\mathcal E_{j_\ell}\|(B_{r_\ell/2}(\hat x))$ : there exists $c_4=c_4(n)$ as
in the claim and a $\mathcal E_{j_\ell}$-admissible function $f$ and $r\in [r_\ell/2,r_\ell]$ such that
\begin{enumerate}
\item[(1)]
$f(x)=x$ for $x\in \mathbb R^{n+1}\setminus U_{r}(\hat x)$,
\item[(2)]
$f(x)\in B_{r}(\hat x)$ for $x\in B_{r}(\hat x)$,
\item[(3)]
$\|\partial f_{\star}\mathcal E_{j_\ell}\|(B_{r}(\hat x))\leq\frac12\|\partial\mathcal E_{j_\ell}\|
(B_{r}(\hat x))$,
\item[(4)]
$\mathcal L^{n+1}(E_{j_\ell,k}\triangle \tilde{E}_{j_\ell,k})\leq c_4(\|\partial\mathcal E_{j_\ell}\|(B_{r}(\hat x)))^{\frac{n+1}{n}}$ for all $k$, where $\{{\tilde E}_{j_\ell,k}\}_{k=1}^N=f_{\star}\mathcal E_{j_\ell}$. 
\end{enumerate} 
We may use Lemma \ref{adcheck} with $C=B_{r}(\hat x)$ and above (1)-(4) to check that 
$f\in {\bf E}^{vc}(\mathcal E_{j_\ell},j_\ell)$ as long as $j_\ell$ is large enough (actually if $\exp(-j_\ell
/2r_\ell)=\exp(-j_\ell^{\sfrac32}/2)<1/2$). Then, by the definition of $\Delta_{j_\ell}^{vc}\|\partial\mathcal E_{j_\ell}\|(\Omega)$, (1) and (3) as well as $\max_{B_r(\hat x)}\Omega\leq \exp(2c_1 r_\ell)
\min_{B_r(\hat x)}\Omega$, we have
\begin{equation}
\begin{split}
\Delta_{j_\ell}^{vc}\|\partial\mathcal E_{j_\ell}\|(\Omega)&\leq \|\partial f_{\star}\mathcal E_{j_\ell}\|(\Omega)
-\|\partial\mathcal E_{j_\ell}\|(\Omega)\\
&\leq (\min_{B_r(\hat x)}\Omega)\big\{\exp(2c_1 r_\ell)\|\partial f_{\star}\mathcal E_{j_\ell}\|(B_r(\hat x))
-\|\partial\mathcal E_{j_\ell}\|(B_r(\hat x))\big\} \\ 
&\leq (\min_{B_r(\hat x)}\Omega)(\frac12 \exp(2c_1 r_\ell)-1) \|\partial\mathcal E_{j_\ell}\|(B_r(\hat x))
\\ &\leq -\frac{\min_{B_{2R}}\Omega}{4}\|\partial\mathcal E_{j_\ell}\|(B_r(\hat x))\leq -\frac{\min_{B_{2R}}\Omega}{4}\|\partial\mathcal E_{j_\ell}\|(B_{r_\ell/2}(\hat x))
\end{split}
\label{bcv1}
\end{equation}
for all sufficiently large $j_\ell$. By the Besicovitch covering theorem, we have a mutually disjoint
set of closed balls $\{B_{r_\ell/2}(\hat x_i)\}_{\hat x_i\in Z_\ell}$ (whose number is at most 
$c(n)r_\ell^{-n-1}$) such that 
\begin{equation}\label{bcv2}
\|\partial\mathcal E_{j_\ell}\|(Z_\ell)\leq {\bf B}(n)\sum \|\partial\mathcal E_{j_\ell}\|(B_{r_\ell/2}(\hat x_i))
\end{equation}
and \eqref{bcv1} and \eqref{bcv2} show that
\begin{equation}\label{bcv3}
\|\partial\mathcal E_{j_\ell}\|(Z_\ell)\leq - \frac{c(n)}{r_\ell^{n+1}\min_{B_{2R}}\Omega} \Delta_{j_\ell}^{vc}\|\partial
\mathcal E_{j_\ell}\|(\Omega).
\end{equation}
Since $-\Delta_{j_\ell}^{vc}\|\partial\mathcal E_{j_\ell}\|(\Omega)\ll r_\ell^{n+1}$ due to \eqref{bcv4}, the right-hand side of 
\eqref{bcv3} converges to $0$ as $\ell\rightarrow\infty$. 
By \eqref{bcv3}, we need to prove
\begin{equation}\label{bcv5}
\lim_{\ell\rightarrow\infty} \Ha^n(Z_\ell^c
\setminus \cup_{k=1}^N
\partial^* E_{j_\ell,k})=0.
\end{equation}
By the Besicovitch covering theorem, we have a set of mutually disjoint balls $\{B_{r_\ell}(x_i)\}$ with
$x_i\in Z_\ell^c\setminus \cup_{k=1}^N
\partial^* E_{j_\ell,k}$ such that 
\begin{equation}\label{bcv7}
\Ha^n(Z_\ell^c
\setminus \cup_{k=1}^N
\partial^* E_{j_\ell,k})\leq {\bf B}(n)\sum_{i}\Ha^n(B_{r_\ell}(x_i)\cap Z_\ell^c
\setminus \cup_{k=1}^N
\partial^* E_{j_\ell,k}).
\end{equation}
Because of the lower bound of measure $\|\partial\mathcal E_{j_\ell}\|(B_{r_\ell}(x_i))\geq c_3 r_\ell^n$
for $x_i\in Z_\ell^c$, the number of disjoint balls is bounded by $c_3^{-1}r_\ell^{-n}\|
\partial\mathcal E_{j_\ell}\|(B_{2R})$. If we prove that 
\begin{equation}
\label{bcv6}
\lim_{\ell\rightarrow\infty}\sup_i r_\ell^{-n}\Ha^n(B_{r_\ell}(x_i)\cap Z_\ell^c\setminus\cup_{k=1}^N
\partial^* E_{j_\ell,k})=0,
\end{equation}
combined with \eqref{bcv7}, we would prove \eqref{bcv5}, ending the proof. Assume for a contradiction that \eqref{bcv6} were not true and we had a 
subsequence (denoted by the same index) 
$x_\ell\in Z_\ell^c$ such that 
\begin{equation}\label{bcv8}
0<\alpha\leq r_\ell^{-n}\Ha^n(B_{r_\ell}(x_\ell)\cap Z_\ell^c\setminus\cup_{k=1}^N\partial^* E_{j_\ell,k}).
\end{equation}
Consider a rescaling $F_\ell:\mathbb R^{n+1}\rightarrow\mathbb R^{n+1}$ defined by $F_\ell(x)=(x-x_\ell)/r_\ell$
and consider a sequence $V_\ell:=(F_\ell)_{\sharp}(\partial\mathcal E_{j_\ell})$ and $F_\ell(E_{j_\ell,k})$ ($k=1,\ldots,N$). This rescaling was
discussed in \cite[Section 4]{KimTone2}, and there exists a subsequence (denoted by the same index) 
and a limit $V\neq 0$ with the properties stated in \cite[Theorem 4.1]{KimTone2}, which shows that $V$ is 
a unit density varifold and $\spt\|V\|$ is a real-analytic minimal hypersurface away from a closed 
lower-dimensional singularity. Moreover, $\spt\|V\|$ is two-sided in the following sense. Let 
$E_1,\ldots,E_N\subset\mathbb R^{n+1}$ be limits of $F_\ell(E_{j_\ell,1}),\ldots,F_\ell(E_{j_\ell,N})$
(see \cite[p.517]{KimTone2}). By the lower-semicontinuity of BV function, we have $\|\partial^*
E_k\|\leq \|V\|$ and $\{E_1,\ldots,E_N\}$ can be defined so that they are mutually disjoint open sets such 
that $\mathbb R^{n+1}\setminus\spt\|V\|=\cup_{k=1}^N E_k$. The two-sidedness means that, at each
regular point $x$ of $\spt\|V\|$, there are two distinct indices $k_1,k_2\in\{1,\ldots,N\}$ such that 
$x\in {\rm clos}\,E_{k_1}\cap{\rm clos}\,E_{k_2}$ (\cite[Lemma 4.8]{KimTone2}). In particular, this implies that
\begin{equation}\label{bcv9}
\|V\|=\frac12 \sum_{k=1}^N\|\partial^* E_k\|.
\end{equation}
On the other hand, \eqref{bcv8} implies that
\begin{equation}\label{bcv10}
\begin{split} \|V_\ell\|(U_2)=
&\|V_\ell\|(U_2\setminus \cup_{k=1}^N \partial^*(F_\ell(E_{j_\ell,k})))+\|V_\ell\|
(U_2\cap \cup_{k=1}^N \partial^* (F_\ell(E_{j_\ell,k}))) \\ 
&\geq r_\ell^{-n} \Ha^n(U_{2r_\ell}(x_\ell)\cap Z_\ell^c\setminus \cup_{k=1}^N\partial^*E_{j_\ell,k}) 
+\|V_\ell\|
(U_2\cap \cup_{k=1}^N \partial^* (F_\ell(E_{j_\ell,k})))\\
&\geq \alpha+\frac12\sum_{k=1}^N
\|\partial^*(F_\ell(E_{j_\ell,k}))\|(U_2).
\end{split}
\end{equation}
Since 
\begin{equation}\label{bcv11}
\|\partial^*E_k\|(U_2)\leq \liminf_{\ell\rightarrow\infty}\|\partial^*(F_\ell(E_{j_\ell,k}))\|(U_2)
\end{equation}
for each $k$, \eqref{bcv9}-\eqref{bcv11} show
\begin{equation*}\label{bcv12}
\|V\|(U_2)\leq \liminf_{\ell\rightarrow\infty}\|V_\ell\|(U_2)-\alpha\leq \|V\|(B_2)-\alpha.
\end{equation*}
Since $\|V\|(\partial U_2)=0$, this is a contradiction. The argument up to this point shows
\eqref{bcv6}, which in turn shows the claim \eqref{bcv0}.
\end{proof}


The next Lemma \ref{lay1} shows that 
$\partial\mathcal E_{j_\ell}(t)$ is locally and subsequencially close to a ``$\theta$-layered 
sheets'' after appropriate blow-ups, for almost all times and places. 
\begin{lemma} \label{lay1}
For $\mathcal L^1$-a.e.~$t\in\mathbb R^+$ and 
$\mathcal H^n$-a.e.~$x\in{\rm spt}\,\|V_t\|$, there exist
$\theta:=\Theta^n(\|V_t\|,x)\in \mathbb N$, $T:={\rm Tan} (\|V\|,x)\in{\bf G}(n+1,n)$, $r_\ell\rightarrow 0+$,
a subsequence $\{j'_\ell\}_{\ell=1}^\infty\subset\{j_\ell\}_{\ell=1}^\infty$ and $\mathcal H^n$-measurable sets $W_{\ell}\subset T\cap B_{r_\ell}$ with the following property (after a change of variable, we may assume
that $x=0$ in the following). 
\newline
Define $f_{(r_\ell)}(y):=y/r_\ell$ for $y\in \mathbb R^{n+1}$
and $\{E_{j'_\ell,1},\ldots,E_{j'_\ell,N}\}:=\mathcal E_{j'_\ell}(t)$. Then we have
\begin{equation}\label{ts1.6}
    \lim_{\ell\rightarrow\infty} (f_{(r_\ell)})_{\sharp} \partial
    \mathcal E_{j'_\ell}(t) ={\bf var}(T,\theta)
\end{equation}
as varifolds,
\begin{equation}\label{ts1.4}
    \mathcal H^0(B_{r_\ell}\cap T^{-1}(a)\cap \cup_{i=1}^N
    \partial E_{j'_\ell,i})=\theta
\end{equation}
for all $a\in W_\ell$, 
\begin{equation}\label{ts1.5}
    \lim_{\ell\rightarrow\infty} \sup_{a\in W_\ell} \{|T^{\perp}(x/r_\ell)|\,:\, x\in 
    B_{r_\ell}\cap T^{-1}(a)\cap \cup_{i=1}^N
    \partial E_{j'_\ell,i}\}=0,
\end{equation}
\begin{equation} \label{ts1.3}
    \lim_{\ell\rightarrow\infty}\frac{\mathcal H^n(W_\ell)}{\omega_n r_\ell^n}=1.
\end{equation}
The claims \eqref{ts1.4} and \eqref{ts1.5} also hold with $\partial^* E_{j'_\ell,i}$ in place of 
$\partial E_{j'_\ell,i}$. 
\end{lemma}

\begin{proof}
Without loss of generality, assume that $t\in\mathbb R^+$ satisfies 
\begin{equation}\label{ts1}
   \liminf_{\ell\rightarrow\infty} \Big(\int_{\mathbb{R}^{n+1}} \frac{|\Phi_{\varepsilon_{j_\ell}} \ast\delta(\partial
    \mathcal E_{j_\ell}(t))|^2\Omega}{\Phi_{\varepsilon_{j_\ell}}\ast\|\partial\mathcal E_{j_\ell}(t)\|+\varepsilon_{j_\ell}\Omega^{-1}}\,dx-\frac{1}{\Delta t_{j_\ell}}
    \Delta_{j_\ell}^{vc}\|\partial\mathcal E_{j_\ell}(t)\|(\Omega)\Big)<\infty,
\end{equation}
which holds for a.e.\,$t\in\mathbb R^+$ by \eqref{e:mean curvature bounds} and
Fatou's lemma. By the compactness theorem of 
\cite[Theorem 8.6]{KimTone}, we may conclude that there exists a converging 
subsequence in the sense of varifold  
$\{\partial\mathcal E_{j_\ell}(t)\}_{\ell=1}^\infty$ (denoted by the same index)
and the limit $V_t\in {\bf IV}_n(\mathbb R^{n+1})$ with $\mu_t=\|V_t\|$. 
By \eqref{e:partitions convergence}, each $\{E_{j_\ell,i}(t)\}_{\ell=1}^\infty$ also converges
to $E_i(t)$ in $L^1_{loc}(\mathbb R^{n+1})$ for $i=1,\ldots,N$.

By Corollary \ref{cor:global_lb},  
for a.e.\,$t\in\mathbb R^+$, we have
\begin{equation}
    \theta_t\,\mathcal H^n\mres_{\cup_{i=1}^N\partial E_i(t)}=\theta_t\,\mathcal H^n\mres_{{\rm spt}\,\|V_t\|}=\|V_t\|,
\end{equation}
where $\theta_t(x):=\Theta^n(\|V_t\|,x)$ is $\mathcal H^n$-a.e.~integer-valued. 
Note in particular that ${\rm spt}\,\|V_t\|$ as well as
$\cup_{i=1}^N\partial E_i(t)$ are countably
$n$-rectifiable. 
We fix such a generic $t$ and subsequently
drop the dependence of $t$. 

In the following, we use the same argument in
the proof of \cite[Theorem 8.6]{KimTone}. 
Let $\{\mathcal E_{j_\ell}\}_{\ell=1}^\infty$ be a subsequence (denoted by 
the same index) such that
the quantity in \eqref{ts1} is uniformly bounded. Let $U\subset\mathbb R^{n+1}$ be a bounded open set. It is sufficient
to prove the claim on $U$. 
As in \cite[p.112-113]{KimTone}, for each $j,q\in\mathbb N$, let $A_{j,q}$ be a set 
consisting of all $x\in {\rm clos}\,U$ such that
\begin{equation}
    \|\delta(\Phi_{\varepsilon_j}\ast\partial\mathcal E_j)\|(B_r(x))
    \leq q \|\Phi_{\varepsilon_j}\ast\partial\mathcal E_j\|(B_r(x))
\end{equation}
for all $r\in (j^{-2},1)$ and additionally define 
\begin{equation}
    A_q:=\{x\in {\rm clos}\,U\,:\, \mbox{there exist }x_\ell\in A_{j_\ell,q}\mbox{ for
    infinitely many $\ell$ with $x_\ell\rightarrow x$}\}.
\end{equation}
Following the argument in \cite[p.113]{KimTone}, one can prove that 
\begin{equation}\label{contr3}
    \|V\|(U\setminus \cup_{q=1}^\infty A_q)=0.
\end{equation}
Thus, for $\mathcal H^n$-a.e.~$x\in{\rm spt}\,\|V\|$, we have 
some $q\in\mathbb N$ such that $x\in A_q$, and additionally,
the approximate tangent space exists with multiplicity 
$\theta\in\mathbb N$. 
Without loss of generality, we may 
assume that $x=0$ and write $T:={\rm Tan}(\|V\|,x)$.
Since $0\in A_q$, there exists a further subsequence of 
$\{j_\ell\}_{\ell=1}^\infty$ (denoted by the same index) such that 
$x_{j_\ell}\in A_{j_\ell,q}$ with $\lim_{\ell\rightarrow\infty}x_{j_\ell}=0$.
Set $r_\ell=1/\ell$ and define $f_{(r_\ell)}(x):=x/r_\ell$ and
\begin{equation}\label{cont3.1}
    V_{j_\ell}:=(f_{(r_\ell)})_\sharp\partial\mathcal E_{j_\ell},\,\,
    \tilde V_{j_\ell}:=(f_{(r_\ell)})_\sharp(\Phi_{\varepsilon_{j_\ell}} 
    \ast\partial\mathcal E_{j_\ell}), \,\,V_{j_\ell}^*:=
    (f_{(r_\ell)})_\sharp\partial^*\mathcal E_{j_\ell},
\end{equation}
where $\partial^*\mathcal E_{j_\ell}$ denotes the unit density varifold defined from
$\cup_{i=1}^N\partial^* E_{j_\ell,i}$. 
We may choose a further subsequence with the
following properties:
\begin{equation}\label{contr7}
    \lim_{\ell\rightarrow\infty}
    V_{j_\ell} =\lim_{\ell\rightarrow\infty}\tilde V_{j_\ell}={\bf var}(T,\theta),
\end{equation}
\begin{equation}
    \lim_{\ell\rightarrow\infty}\frac{x_{j_\ell}}{r_\ell}=0,\hspace{.5cm}
    \lim_{\ell\rightarrow\infty}\frac{j_\ell^{-1}}{r_\ell}=0,
\end{equation}
\begin{equation}\label{contr5}
    \lim_{\ell\rightarrow\infty} r_\ell^{-n} \mathcal H^n(B_{r_\ell}\cap \partial\mathcal E_{j_\ell}
\setminus\partial^*\mathcal E_{j_\ell})=0.
\end{equation}
Note that the choice with \eqref{contr5} is possible due to 
\eqref{bcv0}. We then proceed verbatim as in \cite[p.114-120]{KimTone} with $\nu=\theta+1$ and $d=\theta$. In particular, we 
fix $\lambda\in(1,2)$ such that $\lambda^{n+1}\theta<\theta+1$ (as in \cite[(8.111)]{KimTone}) and use \cite[Lemma 8.5]{KimTone}. In summary, with the same notation 
as in \cite{KimTone}, we obtain a sequence of sets
$G_\ell^{**}\subset\partial\mathcal E_{j_\ell}\cap B_{(\lambda-1)r_\ell}$  with the following properties for all large enough $\ell$;
\begin{equation}\label{contr9}
    \lim_{\ell\rightarrow\infty}r_\ell^{-n}\|\partial\mathcal E_{j_\ell}\|
    (B_{(\lambda-1)r_\ell}\setminus G_\ell^{**})=0,
\end{equation}
\begin{equation}\label{contr9.1}
    \lim_{\ell\rightarrow\infty} \sup\,\{|T^{\perp}(x/r_\ell)|\,:\, x\in G_\ell^{**}\}=0,
\end{equation}
\begin{equation}\label{contr6}
    \sup_{a\in T\cap B_{(\lambda-1)r_\ell}}\mathcal H^0(\{x\in G_\ell^{**}\, :\, T(x)=a\})\leq\theta.
\end{equation}
These claims are, respectively, (8.154), (8.156) (stated differently) and
(8.159) of \cite{KimTone}. The measurability of $G_\ell^{**}$ 
is not stated in \cite{KimTone}. However, since each step to
estimate $\|\partial\mathcal E_{j_\ell}\|(B_{(\lambda-1)r_\ell}\setminus G_\ell^{**})$ uses covering arguments, if necessary, we may simply 
take the complement of these coverings with no change of estimates and 
obtain a possibly smaller $G_\ell^{**}$ which is a Borel set. 
We next prove that, writing $\tilde g_\ell(a):=\mathcal H^0(\{x\in G_\ell^{**}\,:\, T(x)=a\})$ for $a\in T$, we have
\begin{equation}\label{contr8}
    \lim_{\ell\rightarrow\infty} \frac{\mathcal H^n 
    (\{a\in T\cap B_{(\lambda-1)r_\ell}\,:\,\tilde g_\ell(a)=\theta\}) }{\omega_n (\lambda-1)^n r_\ell^n}=1.
\end{equation}
Note that $\tilde g_\ell(a)$ as above on $T$ is $\mathcal H^n$-measurable (see
\cite[Lemma 5.8]{EG}). 
To see \eqref{contr8}, first, by \eqref{contr7} and \eqref{contr9}, we have
\begin{equation}\label{unid1}
    \lim_{\ell\rightarrow\infty} (f_{(r_\ell)})_\sharp\big(\partial
    \mathcal E_{j_\ell}\mres_{G_\ell^{**}\times{\bf G}(n+1,n)}\big)
    ={\bf var}(T\cap B_{\lambda-1},\theta)
\end{equation}
as varifolds. 
We also have
\begin{equation}\label{unid2}
    \|T_{\sharp}\circ(f_{(r_\ell)})_\sharp (\partial
    \mathcal E_{j_\ell}\mres_{G_\ell^{**}\times{\bf G}(n+1,n)})\|(B_{\lambda-1})
    =(r_\ell)^{-n}\int_{T} \tilde g_\ell(a)\,
    d\mathcal H^n(a).
\end{equation}
Since $T_\sharp$ commutes with $\lim_{\ell\rightarrow\infty}$ and  
$T_\sharp{\bf var}(T\cap B_{\lambda-1},1)={\bf var}(T\cap B_{\lambda-1},1)$, \eqref{unid1}
and \eqref{unid2} show that
\begin{equation}\label{unid2.5}
    \theta\omega_n(\lambda-1)^n=\lim_{\ell\rightarrow\infty}
    (r_\ell)^{-n}\int_T \tilde g_\ell(a)\,d\mathcal H^n(a).
\end{equation}
Since $\tilde g_\ell\leq \theta$ due to
\eqref{contr6} (note also $G_\ell^{**}\subset B_{(\lambda-1)r_\ell}$),
\eqref{unid2.5} shows
\begin{equation}\label{contr6sub}
    0=\lim_{\ell\rightarrow\infty}(r_\ell)^{-n} \int_{T\cap B_{(\lambda-1)r_\ell}} (\theta- \tilde g_\ell)\,d\mathcal H^n\geq
    \lim_{\ell\rightarrow\infty} (r_\ell)^{-n}\mathcal H^n
    (T\cap B_{(\lambda-1)r_\ell}\cap\{\tilde g_\ell\leq 
    \theta-1\}).
\end{equation}
Then \eqref{contr6} and \eqref{contr6sub} show \eqref{contr8}. 
Finally, we define
\begin{equation}\label{wdef1}
    W_\ell:=\{a\in T\,:\,\tilde g(a)=\theta\}\setminus \big(T(\partial\mathcal E_{j_\ell}\cap B_{(\lambda-1)r_\ell}\setminus
    G_\ell^{**})\cup T(\partial\mathcal E_{j_\ell}\cap B_{(\lambda-1)r_\ell}\setminus
    \partial^*\mathcal E_{j_\ell})\big).
\end{equation}
Since $\tilde g$ is $\mathcal H^n$-measurable, so is $W_\ell$. Due to \eqref{contr9},
\eqref{contr5} and \eqref{contr8}, we can deduce \eqref{ts1.3}.
For any $a\in W_\ell$, $T^{-1}(a)\cap \partial\mathcal E_{j_\ell}\cap B_{(\lambda-1)r_\ell}$ consists
of $\theta$ points belonging to $G_\ell^{**}\cap \partial^*\mathcal E_{j_\ell}$
by \eqref{wdef1}. This proves \eqref{ts1.4} (both with 
$\partial\mathcal E_{j_\ell}$ and $\partial^*\mathcal E_{j_\ell}$). 
This combined with \eqref{contr9.1} also proves \eqref{ts1.5},
and we may conclude the proof after renaming $(\lambda-1)r_\ell$ as $r_\ell$. 
\end{proof}

\begin{proposition}
\label{lay2}
For $V_t$ and $\{E_i(t)\}_{i=1}^N$ in Proposition \ref{p:properties and limit},
for a.e.~$t\in\mathbb R^+$, we have
\begin{equation}\label{aun1}
    \mathcal H^n(\{x\,:\,\theta_t(x)=1\}\setminus\cup_{i=1}^N\partial^*
    E_i(t))=0.
\end{equation}
\end{proposition}
\begin{proof}    
We prove \eqref{aun1} for a.e.~$t$ such that the conclusion of 
Lemma \ref{lay1} 
holds. With such $t$ fixed, we drop $t$ and 
suppose that \eqref{aun1} were not true for a contradiction. 
We may apply the result of Lemma \ref{lay1} and find a
point $x\in{\rm spt}\,\|V\|\setminus\cup_{i=1}^N\partial^* E_i$ with $\Theta^n(\|V\|,x)=1$
and $T:={\rm Tan}(\|V\|,x)$. Moreover, by the 
well-known property of
the set of finite perimeter, there exists some $k\in\{1,\ldots,N\}$
such that $\lim_{r\rightarrow 0}\frac{\mathcal L^{n+1}(E_k\cap 
B_r(x))}{\omega_{n+1}r^{n+1}}=1$. Without loss of generality,
we may assume that $x=0$, $T=\{x_{n+1}=0\}$ and $k=1$.
Since $\chi_{E_{j_\ell,1}}\rightarrow \chi_{E_1}$ in 
$L^1_{loc}$ as $\ell\rightarrow\infty$, in choosing the subsequence
(denoted by the same index) in Lemma \ref{lay1}, we may additionally arrange the choice
so that we have 
\begin{equation}\label{contr4}
    \lim_{\ell\rightarrow\infty}\frac{\mathcal L^{n+1}(E_{j_\ell,1}\cap B_{r_\ell})}{\omega_{n+1}r_\ell^{n+1}}=1. 
\end{equation}
For simplicity, write the union of the measure-theoretic boundaries as $\partial_* \mathcal E_{j_\ell}
:=\cup_{i=1}^N\partial_* E_{j_\ell,i}$. Since $\partial\mathcal E_{j_\ell}$ is closed, we remind the reader that \begin{equation}\label{unid3}
    \partial^*\mathcal E_{j_\ell} \subset\partial_*\mathcal E_{j_\ell} \subset \partial\mathcal E_{j_\ell}.
\end{equation} 
We also define the measure-theoretic interior and exterior of $E_{j_\ell,1}$ (\cite[Definition 5.13]{EG}) as:
\begin{equation}
    I_{\ell}:=\Big\{x\in\mathbb R^{n+1}\,:\, \lim_{r\rightarrow 0}
    \frac{\mathcal L^{n+1}(B_r(x)\setminus E_{j_\ell,1})}{r^{n+1}}=0\Big\}
\end{equation}
and 
\begin{equation}
  O_{\ell}:=\Big\{x\in\mathbb R^{n+1}\,:\, \lim_{r\rightarrow 0}
    \frac{\mathcal L^{n+1}(B_r(x)\cap E_{j_\ell,1})}{r^{n+1}}=0\Big\}. 
\end{equation}
We note from the definition that 
\begin{equation} \label{unid8}
    I_\ell\cap \partial_*\mathcal E_{j_\ell}=\emptyset. 
\end{equation}
We next use some results in the proof of \cite[Theorem 5.23]{EG} 
on the property of measure-theoretic interior and exterior. For each
$m,k\in\mathbb N$, define
\begin{equation}\label{aunid1}
\begin{split}
    & G_\ell(k):=\Big\{x\in\mathbb R^{n+1}\,:\, \mathcal L^{n+1}(B_r(x)\cap 
    O_\ell)\leq \frac{\omega_{n}r^{n+1}}{3^{n+2}}\mbox{ for }0<r<\frac{3}{k}\Big\}, \\
    & H_\ell(k):=\Big\{x\in\mathbb R^{n+1}\,:\, \mathcal L^{n+1}(B_r(x)\cap 
    I_\ell)\leq \frac{\omega_{n}r^{n+1}}{3^{n+2}}\mbox{ for }0<r<\frac{3}{k}\Big\},
    \end{split}
\end{equation}
and
\begin{equation}\label{aunid2}
\begin{split}
    & G_\ell^{\pm}(k,m):=G_\ell(k)\cap \Big\{x\,:\,x\pm se_{n+1}\in O_\ell\mbox{ for }0<s<\frac{3}{m}\Big\}, \\
    & H_\ell^{\pm}(k,m):=H_\ell(k)\cap \Big\{x\,:\,x\pm se_{n+1}\in I_\ell\mbox{ for }0<s<\frac{3}{m}\Big\}. 
    \end{split}
\end{equation}
Here $e_{n+1}$ is the unit vector pointing towards the positive direction of $x_{n+1}$-axis. These sets have the property (see Step 3 of \cite[Theorem
5.23]{EG}) that
\begin{equation}
    \mathcal H^{n}(T(G_\ell^{\pm}(k,m)))=\mathcal H^n(T(H_\ell^\pm (k,m)))=0
\end{equation}
for all $k,m\in\mathbb N$. Moreover, for all
\begin{equation}
    a\in T\setminus \cup_{k,m=1}^\infty T\big(
    G^+_\ell(k,m)\cup G^-_\ell(k,m)\cup H^+_\ell(k,m)\cup H^-_\ell(k,m)\big)
\end{equation}
and with $\mathcal H^0(T^{-1}(a)\cap \partial_*E_{j_\ell,1})<\infty$, if $x_1,x_2\in T^{-1}(a)$ with $T^{\perp}(x_1)<T^{\perp}(x_2)$ and
$x_1\in I_\ell$ and $x_2\in O_\ell$,
then there exists $x_3\in T^{-1}(a)\cap \partial_* E_{j_\ell,1}$ such that $T^{\perp}(x_1)<
T^{\perp}(x_3)<T^{\perp}(x_2)$ (see Step 5 of \cite[Theorem 5.23]{EG}). Here, $x_1$ is in interior and $x_2$ is in exterior of 
$E_{j_\ell,1}$ over $a$, and the claim is that there must be a
``boundary point'' $x_3$ between these two points. The same claim 
holds if $x_1\in O_\ell$ and $x_2\in I_\ell$ instead. 

Let $W_\ell$ be the set obtained in Lemma \ref{lay1}. 
By \cite[Lemma 5.9]{EG}, we have 
$\mathcal L^{n+1}(I_\ell\triangle E_{j_\ell,1})=0$, so with \eqref{contr4},
\begin{equation}\label{unid5}
    \lim_{\ell\rightarrow\infty}\frac{\mathcal L^{n+1}(I_\ell\cap B_{r_\ell} )}{\omega_{n+1}r_\ell^{n+1}}=1.
\end{equation}
Then, by the Fubini Theorem and \eqref{unid5}, we may choose a sequence $\{b_{\ell}\}_{\ell=1}^\infty 
\subset \mathbb R^+$ such that $b_{\ell}\in [r_\ell/3,r_\ell/2]$ and so that, writing 
\begin{equation*}\begin{split}
    &A_{\ell}^+:=
B_{r_\ell}\cap \{x_{n+1}= b_\ell\},\\
&A_{\ell}^-:=
B_{r_\ell}\cap \{x_{n+1}= -b_\ell\},
\end{split}
\end{equation*}
we have
\begin{equation}\label{unid7}
    \lim_{\ell\rightarrow\infty} \frac{\mathcal H^n(I_\ell
    \cap A_{\ell}^+)}{\mathcal H^n(A_\ell^+)}=\lim_{\ell\rightarrow\infty} \frac{\mathcal H^n(I_\ell
    \cap A_{\ell}^-)}{\mathcal H^n(A_\ell^-)}=1.
\end{equation}
On cylinders $T(A_\ell^+)\times[-b_\ell,b_\ell]\subset
\mathbb R^{n+1}$, by \eqref{aunid1}, \eqref{aunid2} and the stated property 
thereafter, for $\mathcal H^n$-a.e. $a\in W_\ell$, we have the following property:
\begin{equation}\label{unid6}
\begin{split}
    &\mbox{if }(a,s_1)\in I_\ell\mbox{ and }(a,s_2)\in O_\ell
    \mbox{ with } -b_\ell\leq s_1<s_2\leq b_\ell,
    \\ &\mbox{ then there
    exists }\hat s\in (s_1,s_2)\mbox{ such that }(a,\hat s)\in 
    \partial_* E_{j_\ell,1},
\end{split}
\end{equation}
and similarly if $(a,s_1)\in O_\ell$ and $(a,s_2)\in I_\ell$. 
On the other hand, we know that $T^{-1}(a)\cap B_{r_\ell}\cap 
\partial_*\mathcal E_{j_\ell}$ is a singleton located close to $T$
due to \eqref{ts1.4} and \eqref{ts1.5}. Here, we used the fact
that \eqref{ts1.4} is satisfied for both $\partial\mathcal E_{j_\ell}$
and $\partial^*\mathcal E_{j_\ell}$ as well as \eqref{unid3}.
We use this fact to 
\begin{equation*}
    a\in  W_\ell\cap T(I_\ell\cap A_\ell^+)\cap 
    T(I_\ell\cap A_\ell^-)=:W_\ell^*.
\end{equation*}
Note that $(a,r_\ell)$ and $(a,-r_\ell)$ are both in $I_\ell$ 
and $(\{a\}\times[-r_\ell,r_\ell])\cap \partial_* \mathcal E_{j_\ell}$
is a singleton due to the way $ W_\ell$ is defined. 
If $(\{a\}\times [-r_\ell,r_\ell])\cap O_\ell\neq \emptyset$,
\eqref{unid6} implies that there must be at least two points of
$\partial_* E_{j_\ell,1}$ in $\{a\}\times[-r_\ell,r_\ell]$, since
both crossing from $I_\ell$ to $O_\ell$ and the other way
around have to happen. 
Since $\partial_*\mathcal E_{j_\ell}=\cup_{i=1}^N\partial_*
E_{j_\ell,i}$, this is a contradiction. Combined with
\eqref{unid8}, we conclude that for $\mathcal H^n$-a.e.\,$a\in W_{\ell}^*$, $\{a\}\times[-r_\ell,r_\ell]$ is a disjoint union of
one point of $\partial_* \mathcal E_{j_\ell}$ and two line segments included
in $I_\ell$, with no point of $O_\ell$. 
Because of \eqref{ts1.3} and \eqref{unid7}, one also 
sees 
\begin{equation}
    \lim_{\ell\rightarrow\infty}\frac{\mathcal H^n(W_\ell^*)}{\mathcal H^n(A_\ell^+)}=1.
\end{equation}
In particular, $W_\ell^*$ has a positive $\mathcal H^n$ 
measure in $T\subset \mathbb R^n\times\{0\}$ 
and there must
be a Lebesgue point $a$ of $W_\ell^*$ such that 
$(\{a\}\times[-r_\ell,r_\ell])\cap \partial_* \mathcal E_{j_\ell}$ is a 
singleton, say, $\{(a,s)\}$. Then, by the Fubini theorem and 
the property of $W_\ell^*$,
\begin{equation*}
\begin{split}
    r^{-(n+1)}\mathcal L^{n+1}(O_\ell\cap B_r((a,s)))&\leq 
    r^{-(n+1)}\mathcal L^{n+1}((B^n_r(a)\setminus W_\ell^*)\times[s-r,s+r]) \\
    &\leq 2r^{-n} \mathcal H^n(B_r^n(a)\setminus W_\ell^*)
\end{split}
\end{equation*}
which converges to $0$ as $r\rightarrow 0$ since $a$ is a 
Lebesgue point of $W_\ell^*$ in $T$. Since $\mathcal L^{n+1}(O_\ell
\cap B_r(x))=\mathcal L^{n+1}(B_r(x)\setminus E_{j_\ell,1})$, 
this implies that $(a,s)\in I_\ell$.
On the other hand, $(a,s)\in 
\partial_* \mathcal E_{j_\ell}$, a contradiction to \eqref{unid8}. This concludes the proof.
\end{proof}
\begin{proposition}\label{evenodd}
Assume $N=2$. For $V_t$ and $\{E_i(t)\}_{i=1}^2$ in Proposition \ref{p:properties and limit}, for a.e.~$t\in \mathbb R^+$, we have 
\begin{equation}
    \theta_t(x)=\left\{\begin{array}{ll} \mbox{odd} & \mathcal H^n\,a.e.\,x
    \in \partial^* E_1(t)(=\partial^*E_2(t)),\\
    \mbox{even} & \mathcal H^n\,a.e.\,x\in {\rm spt}\,\|V_t\|\setminus\partial^*
    E_1(t).
    \end{array}\right.
\end{equation}
\end{proposition}
\begin{proof}
The proof proceeds similarly as the proof of 
Proposition \ref{lay2}, except that we need to localize the argument to
each layers. Fix a bounded 
open set $U\subset\mathbb R^{n+1}$. We may choose a 
generic $t\in\mathbb R^+$ as before and drop $t$.
Since $\partial^* E_1=\partial^* E_2$ by the definition of the
reduced boundary, for $\mathcal H^n$-a.e.~$x\in {\rm spt}\,\|V\|\setminus \partial^* E_1$, we have
\begin{equation}\label{uns1}
    \lim_{r\rightarrow 0}\frac{\mathcal L^{n+1}
(E_i\cap B_r(x))}{\omega_{n+1}r^{n+1}}=1\mbox{ for either }i=1
\mbox{ or }2. 
\end{equation}
On the other hand, for $\mathcal H^n$-a.e.~$x\in \partial^*
E_1$, there exists a unit outer normal $\nu$ to $\partial^*
E_1$ such that, letting $B_r^{+(-)}(x):=\{y\in B_r(x)\,:\, (y-x)\cdot\nu\geq(\leq) 
0\}$, we have 
\begin{equation}\label{uns2}
   \lim_{r\rightarrow 0}\frac{\mathcal L^{n+1}(E_1\cap B_r^+(x))}{\omega_{n+1}r^{n+1}}=0\,\,\mbox{ and }\,\,
     \lim_{r\rightarrow 0}\frac{\mathcal L^{n+1}(E_1\cap B_r^-(x)
   )}{\omega_{n+1}r^{n+1}}=\frac12.
\end{equation}
We use Lemma \ref{lay1}, and 
in the proof of Lemma \ref{lay1} we may additionally assume that
the chosen subsequence satisfies
\begin{equation}\label{uns3}
    \lim_{\ell\rightarrow\infty}\frac{\mathcal L^{n+1}(E_{j_\ell,i}
    \cap B_{r_\ell})}{\omega_{n+1}r_\ell^{n+1}}=1
    \,\,\mbox{for either $i=1$ or $2$}
\end{equation}
if $0\in {\rm spt}\,\|V\|\setminus \partial^*E_1$ and
\begin{equation}\label{uns4}
  \lim_{\ell\rightarrow \infty}\frac{\mathcal L^{n+1}(E_{j_\ell,1}\cap B_{r_\ell}^+)}{\omega_{n+1}r_\ell^{n+1}}=0\,\,\mbox{ and }\,\,
     \lim_{\ell\rightarrow \infty}\frac{\mathcal L^{n+1}(E_{j_\ell,1}\cap B_{r_\ell}^-
   )}{\omega_{n+1}r_\ell^{n+1}}=\frac12
\end{equation}
if $0\in\partial^*E_1$. Without
loss of generality, we may assume $i=1$ in \eqref{uns3}, $T=\{x_{n+1}=0\}$
and $B_{r_\ell}^+=B_{r_\ell}\cap \{x_{n+1}\geq 0\}$ in \eqref{uns4}.
By \eqref{ts1.6}, we have
\begin{equation}\label{anti3}
    \lim_{\ell\rightarrow\infty}(r_\ell)^{-n} \int_{B_{2r_{\ell}}} \|S-T\|\,d(\partial
    \mathcal E_{j_\ell})(x,S)=0.
\end{equation}
Set 
\begin{equation}\label{anti2}
    C_\ell:=\Big\{x\in\partial\mathcal E_{j_\ell}\cap B_{r_\ell}\,:\, \|{\rm Tan}(\|\partial\mathcal E_{j_\ell}\|,x)-T\|\leq 1/10\Big\}.
\end{equation}
Given any $0<\delta<r_\ell$ and $\mathcal H^n$-a.e.~$x\in \partial\mathcal E_{j_\ell}$, by the rectifiability of $\partial\mathcal E_{j_\ell}$, there exists $0<r<\delta$ such that
\begin{equation}\label{anti1}
\frac12\leq \frac{1}{\omega_n r^n}\|\partial\mathcal E_{j_\ell}\|(B_r(x))\mbox{ and }
   \frac{1}{\omega_n r^n} \int_{B_r(x)}\|{\rm Tan}(\|\partial
   \mathcal E_{j_\ell}\|,x)-S\|\,d(\partial\mathcal E_{j_\ell})\leq \frac{1}{40}.
\end{equation}
Then, for $\mathcal H^n$-a.e.~$x\in \partial\mathcal E_{j_\ell}\cap B_{r_\ell}\setminus C_\ell$, 
\eqref{anti1} and \eqref{anti2} show
\begin{equation}
\begin{split}
    \int_{B_r(x)} \|T-S\|\,d(\partial\mathcal E_{j_\ell}) &
    \geq \|{\rm Tan}(\|\partial\mathcal E_{j_\ell}\|,x)-T\|\,
    \|\partial\mathcal E_{j_\ell}\|(B_r(x))-\omega_n r^n/40 \\
    & \geq\omega_n r^n(1/20-1/40)=\omega_n r^n/40.
    \end{split}
\end{equation}
We cover $\partial\mathcal E_{j_\ell}\cap B_{r_\ell}\setminus C_\ell$ by 
such balls and use the Besicovitch covering theorem to show that
\begin{equation}\label{anti4}
    \mathcal H^n(\partial\mathcal E_{j_\ell}\cap B_{r_\ell}\setminus C_\ell)
    \leq 40{\bf B}_{n+1}\int_{B_{2r_\ell}} \|T-S\|\,d(\partial\mathcal E_{j_\ell}),
\end{equation}
where ${\bf B}_{n+1}$ is the Besicovitch constant. Then \eqref{anti3} and
\eqref{anti4} show
\begin{equation}
    \lim_{\ell\rightarrow\infty}(r_\ell)^{-n}
    \mathcal H^n(\partial\mathcal E_{j_\ell}\cap B_{r_\ell}\setminus C_\ell)
    =\lim_{\ell\rightarrow\infty}(r_\ell)^{-n}
    \mathcal H^n(T(\partial\mathcal E_{j_\ell}\cap B_{r_\ell}\setminus C_\ell))
    =0,
\end{equation}
so that $T^{-1}(a)\cap \partial\mathcal E_{j_\ell}\cap B_{r_\ell} \subset 
C_\ell$ for $a$ in a large portion of $T\cap B_{r_\ell}$ for large enough $\ell$. 
Now recall the property of $W_\ell$ in Lemma \ref{lay1}. We may redefine $W_\ell$
by $W_\ell\setminus T(\partial\mathcal E_{j_\ell}\cap B_{r_\ell}\setminus C_\ell)$
and keep the properties \eqref{ts1.4}-\eqref{ts1.3}. By this, we additionally 
have that 
\begin{equation}
\|{\rm Tan}( \|\partial\mathcal E_{j_\ell}\|,x)-T\|\leq 1/10\mbox{ for all }x\in T^{-1}(W_\ell)\cap B_{r_\ell}\cap \partial^*\mathcal E_{j_\ell}.
\end{equation}
We now proceed similarly to the previous Proposition \ref{lay2}, and choose $b_\ell$
satisfying \eqref{unid7} in the case of \eqref{uns3}. The case of \eqref{uns4}
can be handled similarly so we discuss the former case. For all large $\ell$, 
we may choose a Lebesgue point $a$ of $W_\ell$ in $T$ such that 
$(a,-b_\ell)$ and $(a,b_\ell)$ are also Lebesgue points of $I_\ell\cap A_\ell^-$
and $I_\ell\cap A_\ell^+$, respectively. By \eqref{ts1.4} and \eqref{ts1.5}, there are 
$-b_\ell:=u_0<u_1<\ldots<u_\theta<u_{\theta+1}:=b_\ell$ such that $\cup_{k=1}^\theta\{(a,u_k)\}
=T^{-1}(a)\cap B_{r_\ell}\cap \partial^*\mathcal E_{j_\ell}$. At each point
$(a,u_k)$, since it is in $\partial^* E_{j_\ell,1}$, 
the blow-up of $E_{j_\ell,1}$ converges to a half-space, with the approximate
tangent space having a small slope relative to $T$ due to \eqref{anti2}. 
Then, for sufficiently small $0<\delta<\min_{0\leq k\leq \theta}
\{|u_{k+1}-u_k|\}$ and $k=1,\ldots, \theta$, we may choose $b_{\ell,k}
\in [\delta/3,\delta/2]$ so that
\begin{equation}
   \frac{ \mathcal H^n (T^{-1}(B^n_\delta(a))\cap \{x_{n+1}=u_k+b_{\ell,k}\}\cap O_\ell)}{
    \omega_n \delta^n}
    \geq 1-\frac{1}{6\theta}
\end{equation}
and 
\begin{equation}
    \frac{\mathcal H^n(T^{-1}(B^n_\delta(a))\cap \{x_{n+1}=u_k-b_{\ell,k}\}\cap I_\ell)}{
    \omega_n \delta^n}
    \geq 1-\frac{1}{6\theta},
\end{equation}
or the inequalities replacing the role of $O_\ell$ and $I_\ell$. We may also 
assume that 
\begin{equation}
    \frac{\mathcal H^n(B_\delta^n(a)\cap W_\ell)}{\omega_n\delta^n}\geq \frac{8}{9}
\end{equation}
since $a$ is a Lebesgue point of $W_\ell$, and similarly
\begin{equation}
    \frac{\mathcal H^n(T^{-1}(B_\delta^n(a))\cap I_\ell\cap A_\ell^{\pm})}{\omega_n
    \delta^n}\geq \frac{8}{9}.
\end{equation}
With these properties, we can make sure that, with respect to $\mathcal H^n$, 
$1/3$ of $W_\ell\cap B_\delta^n(a)$ has the property that, if $\tilde a$ is
in this set, 
\begin{equation}
    (\tilde a,\pm b_\ell)\in I_\ell, \,\, (\tilde a, u_k+b_{\ell,k})\in O_\ell,\,\,
    (\tilde a,u_k-b_{\ell,k})\in I_\ell \mbox{ or vice-versa for } k=1,\ldots,
    \theta.
\end{equation}
Using \eqref{unid6}, for $\mathcal H^n$-a.e.~$\tilde a$ as above, there 
exists some $s_k\in (u_k-b_{\ell,k},u_k+b_{\ell,k})$ with $(\tilde a,s_k)\in 
\partial_* E_{j_\ell,1}$ for each $k=1,\ldots,\theta$, and there are no other point of $\partial_* E_{j_\ell,1}$ along the line segment connecting $(\tilde a,-b_\ell)$ and $(\tilde a,b_\ell)$. Looking at this line segment and the intersection of $O_\ell$
and $I_\ell$, 
since these two endpoints are in $I_\ell$ and each $(\tilde a,s_k)$
is sided by $O_\ell$ and $I_\ell$, $\theta$ has to be necessarily even.
This finishes the proof in the case of \eqref{uns3}. For \eqref{uns4}, the 
similar argument results in the situation that $(\tilde a,-b_\ell)\in I_\ell$ and $(\tilde a, b_\ell)\in O_\ell$,
which necessitates that $\theta$ is odd. This concludes the proof. 
\end{proof}
Finally, we comment on the proofs of Theorem \ref{main4} and \ref{main5}. If $V_t$ is a unit density flow in $U\times(t_1,t_2)$,
then $\theta_t(x)=1$ for $\|V_t\|$-a.e.~$x\in U$ and
a.e.~$t\in(t_1,t_2)$. By Theorem \ref{main2}(2)(4), we may 
assume that $V_t={\bf var}
(\Gamma(t),1)$, and by Theorem \ref{main3}(7), $\Gamma(t)$
may be replaced by $\cup_{i=1}^N\partial^* E_i(t)$. Thus 
\eqref{4-eq} follows immediately. To check that \eqref{BV formulation of mean curvature} holds, since $\sum_{i\neq j}\mathcal H^n\mres_{I_{i,j}(t)}
=2\mathcal H^n\mres_{\Gamma(t)}=2\|V_t\|$
(see \eqref{chap1} and \eqref{chap2}), the left-hand 
side of \eqref{BV formulation of mean curvature} 
is equal to $2\int_0^T \delta V_t(g)\,dt$. Since $v_i\nu_i=(h\cdot\nu_i)\nu_i
=h$ for $\mathcal H^n$-a.e.~$x\in I_{i,j}(t)$ due to the 
perpendicularity of the mean curvature vector, the right-hand
side of \eqref{BV formulation of mean curvature} is 
$-2\int_0^T \int h\cdot g\,d\|V_t\|dt$. Since the 
generalized mean curvature vector exists for a.e.~$t>0$,
they are equal indeed. This proves the claim of Theorem 
\ref{main4}. For Theorem \ref{main5}, under the assumption,
one can show that there exists $T_0=T_0(n,\mathcal H^n(\Gamma_0),r_0,\delta_0)>0$ such that $\int_{\Gamma_0} \rho_{(y,s)}(x,0)\,d\mathcal H^n(x)<2-\delta_0/2$
(recall \eqref{Huisken heat}) for all
$y\in\mathbb R^{n+1}$ and $0<s\leq T_0$. Then, Huisken's
monotonicity formula shows that $\int_{\mathbb R^{n+1}}
\rho_{(y,s)}(x,t)\,d\|V_t\|(x)$ is non-increasing on 
$t\in[0,s)$ and thus $< 2-\delta_0/2$. For a contradiction, if $V_t$ is not unit density on $[0,T_0]$,
there would exist some $t\in (0,T_0)$ with $V_t\in {\bf IV}_n
(\mathbb R^{n+1})$, and $y\in {\rm spt}\|V_t\|$ such that 
$\Theta^n(\|V_t\|,y)\geq 2$ and where $T_y\|V_t\|$ exists. 
Then, one can prove that $\lim_{\epsilon\rightarrow 0+}
\int_{\mathbb R^{n+1}}\rho_{(y,t+\epsilon)}(x,t)\,d\|V_t\|(x)
=\Theta^n(\|V_t\|,y)\geq 2$. Since $t+\epsilon<T_0$ for all
small $\epsilon>0$, this would be a contradiction. Thus, 
$V_t$ is a unit density Brakke flow on $[0,T_0]$. Once this
is proved, by Theorem \ref{main4}, the claim of 
BV solution also follows. This is the outline of proof of
Theorem \ref{main5}. 

\section{Final remarks} \label{s:final remarks}
\subsection{On generalized BV solutions}\label{on-g-BV}
As explained in Definition \ref{d:BV flow}, a BV flow is classically defined as consisting of two objects: $N$ families of sets of finite perimeter 
$E_i(t)$, and velocities $v_i$. From them, naturally,
one can define a unit density varifold $V_t={\bf var}(\cup_{i=1}^N\partial^*
E_i(t),1)$ and check that \eqref{BV formulation of mean curvature}
implies that the generalized mean curvature $h(\cdot,V_t)$ is equal to $v_i\nu_i$
on $\partial^* E_i(t)$ for $i=1,\ldots,N$. 
The generalized BV flow of Theorem \ref{main3}(6)
involves the accompanying Brakke flow $V_t$ in addition to the families
of sets of finite perimeter, and one may wonder if the definition 
makes sense even without the reference to 
the Brakke flow. In fact, it is interesting to 
observe that each $\partial^* E_i(t)$ for a.e.~$t>0$ is $C^2$-rectifiable
due to Menne's $C^2$-rectifiability theorem \cite[Theorem 4.8]{Menne} and
one can define a unique second fundamental form for $\partial^* E_i(t)$ 
as well as mean curvature vector by the $C^2$-approximability property, independent of $V_t$. The mean curvature vector defined in this sense
coincides with $h(\cdot,V_t)$ $\|V_t\|$-a.e. on $\partial^* E_i(t)$. Thus, $h\cdot\nu_i$ on $\partial^* E_i(t)$ is uniquely defined from the 
$C^2$-rectifiability without reference to $V_t$. 
On the other hand, the summation  $\tilde h:=\frac12\sum_{i=1}^N(h\cdot\nu_i)\nu_i$ may not correspond, in general, to the generalized
mean curvature vector of ${\bf var}(\cup_{i=1}^N\partial^* E_i(t),1)$ if
there is some non-trivial higher multiplicity portion of $V_t$. For example on $\mathbb R^2$, define
$E_+:=\{(x,y)\,:\,y>0\mbox{ if }x\leq 0,\,y>x^2\mbox{ if }x>0\}$ and $E_-:=\{(x,y)\,:\, y<0\mbox{ if }x\leq 0,\,
y<-x^2\mbox{ if }x>0\}$. Then $V:={\bf var}(\partial E_+,1)
+{\bf var}(\partial E_-,1)$ has a bounded generalized
mean curvature while ${\bf var}(\pa^*E_+ \cup \pa^*E_-,1)$ has a singular first variation at the origin. 
Note that $V$ has multiplicity $=2$ on the negative
$x$-axis. 
In this sense, the formula \eqref{BV formulation of mean curvature} does not hold in general. Note that
\eqref{d:e_diss} is relevant only when $\mathcal H^n(\Gamma_0)$ is finite, and it follows from Theorem
\ref{main1}(4) with $v_i=h\cdot\nu_i$.
Over all, for 
generalized BV solution, it makes sense to consider the pair of 
sets of finite perimeter and Brakke flow together, unlike the original
BV solutions discussed in Definition \ref{d:BV flow}.

\subsection{MCF with fixed boundary conditions} \label{fixed}
In \cite{ST19}, given a strictly convex bounded
domain $U\subset\mathbb R^{n+1}$ with $C^2$
boundary $\partial U$, a countably $n$-rectifiable set $\Gamma_0
\subset U$ with $\mathcal H^n(\Gamma_0)<\infty$
and an open partition $E_{0,1},\ldots,E_{0,N}$
of $U$ such that $\Gamma_0=U\setminus \cup_{i=1}^N
E_{0,i}$, existence of a Brakke flow and a family of open 
partitions with fixed
boundary condition is established for the given 
initial datum. The construction method is along the lines of \cite{KimTone},
and we may also carry it out using the volume-controlled Lipschitz maps. If one compares the construction in \cite{ST19} with that in \cite{KimTone}, one sees that differences occur only 
near the boundary $\partial U$: more precisely, the approximate smoothed mean curvature vector is damped 
near the portion of $\Gamma_0$ close to $\partial U$, and
there is 
another step in each epoch -- a Lipschitz retraction
step (see \cite[Section 2.6]{ST19}). Hence, the proof
of the present paper works with no essential change
away from $\partial U$, and \eqref{th:BV_identity}
holds for $\phi\in C_c^1(U\times\mathbb R^+)$; since the formula does not involve
$\nabla\phi$, by approximation, the same formula holds even
for $\phi\in C^1({\rm clos}\,U\times[0,T])$
for arbitrary $T>0$.  
Since the existence results in \cite{stu-tone2} are 
based on \cite{ST19}, the same applies to the solutions discussed in \cite{stu-tone2}.

\subsection{A lower bound estimate for extinction time}\label{extinct}

Here we prove Theorem \ref{t:extinction}. Suppose that $\Ha^n(\Gamma_0) < \infty$, and assume without loss of generality that $N$ is the index of the only grain $E_{0,i}$ with infinite volume. Define then the open set $E(t):=\cup_{i=1}^{N-1}
E_i(t)$. Since $\nu_i\cdot h=-\nu_j\cdot h$ on 
$\partial^* E_i(t)\cap\partial^* E_j(t)$ for $i\neq j$ $\mathcal H^n$-a.e., after summing over $i=1,\ldots,N-1$ formula \eqref{th:BV_identity} gives
\begin{equation*}
 |E(t)|-|E(0)|
 =\int_{0}^{t}\int_{\partial^* E(s)}
 \nu\cdot h\,d\mathcal H^n
\end{equation*}
for all $0<t<\infty$,
where $\nu$ is the outer unit normal of $\partial^* E(s)$. In particular, if we set $v(t) := |E(t)|$ then we have
\begin{equation} \label{derivative of volume}
    v'(t) = \int_{\pa^*E(t)} \nu \cdot h \,d\Ha^n \qquad \mbox{for a.e. $t \in \R^+$}\,.
\end{equation}
Next, set $a(t) := \|V_t\|(\R^{n+1})$. By Brakke's inequality \eqref{brakineq}, the upper derivative $a'_+(t)$ of $a(t)$ satisfies
\begin{equation} \label{derivative of area}
   - a'_+(t) \geq \int |h|^2 \, d\|V_t\| \geq \int_{\Gamma(t)} |h|^2\, d\Ha^n \geq \int_{\pa^*E(t)} |h|^2 \,d\Ha^n \qquad \mbox{for a.e. $t \in \R^+$}\,.
\end{equation}
Combining \eqref{derivative of volume} and \eqref{derivative of area} then gives the inequality
\[
-v'(t) \leq \left( \Ha^n (\pa^*E(t)) \right)^{\frac12}\,\left( \int_{\pa^*E(t)} |h|^2\,d\Ha^n \right)^{\frac12} \leq \left(- a(t)\,a'_+(t) \right)^{\frac12} = \left( - \frac{(a^2)'_+(t)}{2} \right)^{\frac12}
\]
for a.e. $t \in \R^+$. Integrating over any interval $\left[0,T\right]$ and using the Cauchy-Schwarz inequality then yields
\begin{equation} \label{ext estimate 1}
    |E(0)| - |E(T)| \leq \sqrt{T/2} \, \left( \|V_0\|(\R^{n+1})^2 - \|V_T\|(\R^{n+1})^2  \right)^{\frac12} \leq \sqrt{T/2} \, \Ha^n(\Gamma_0)\,.
\end{equation}


Since $V_t\neq 0$ as long as $E(t)\neq \emptyset$, the extinction time $T_*$ is at least equal to the first time when $E(t)$ becomes the empty set, so that \eqref{ext estimate 1} implies 
\[    T_* \geq 2\, \left( \frac{|E(0)|}{\Ha^n (\Gamma_0)} \right)^2\,,
\]
that is \eqref{e:extinction}. \qed


\appendix

\section{The existence theorem of \cite{KimTone} revisited} \label{appendix:Lipschitz}

Here we point out places which require a change to
${\bf E}^{vc}(\mathcal E,j)$ from 
${\bf E}(\mathcal E,j)$ in \cite{KimTone,KimTone2}. 
It turned out that the proofs
require no essential change and the only point to 
be checked is that the same Lipschitz maps used in the
proofs satisfy the condition of Definition \ref{ar8}(b).
\subsection{Construction of approximate flows}
For the construction of discrete approximate sequence in \cite[Section 6]{KimTone},
we simply replace ${\bf E}(\mathcal E_{j,l},j)$ by ${\bf E}^{vc}(\mathcal E_{j,l},j)$
and $\Delta_j\|\partial\mathcal E_{j,l}\|(\Omega)$ by $\Delta_j^{vc}\|\partial\mathcal E_{j,l}\|(\Omega)$
when $f_1$ is chosen in \cite[(6.9)]{KimTone}. 
As in \cite[(6.10)]{KimTone}, if we define $\{E_{j,l,i}\}_{i=1}^N:=\mathcal E_{j,l}$,
$\mathcal E_{j,l+1}^*:=(f_1)_\star\mathcal E_{j,l}$ and $\{E_{j,l+1,i}^*\}_{i=1}^N=\mathcal E^*_{j,l+1}$, by 
Definition \ref{ar8}(b) and
\eqref{ar1}, we have for each $i=1,\ldots,N$
\begin{equation*}
\mathcal L^{n+1}(E_{j,l+1,i}^*\triangle E_{j,l,i})\leq \{\|\partial\mathcal E_{j,l}\|(\Omega)
-\|\partial \mathcal E_{j,l+1}^*\|(\Omega)\}/j\leq -(\Delta_j^{vc}\|\partial\mathcal E_{j,l}\|(\Omega))/j.
\end{equation*}
The change in \cite[(6.9)]{KimTone} is also reflected in the estimate \cite[(6.4)]{KimTone} and we have
\begin{equation*}
\begin{split}
&\frac{\|\partial\mathcal E_{j,l}\|(\Omega)-\|\partial\mathcal E_{j,l-1}\|(\Omega)}{\Delta t_j}+\frac14 \int_{\mathbb R^{n+1}}
\frac{|\Phi_{\varepsilon_j}\ast\delta(\partial\mathcal E_{j,l})
|^2\Omega}{\Phi_{\varepsilon_j}\ast\|\partial\mathcal E_{j,l}\|
+\varepsilon_j\Omega^{-1}}\,dx \\
&-\frac{(1-j^{-5})}{\Delta t_j}\Delta_j^{vc}\|\partial\mathcal E_{j,l-1}\|(\Omega)\leq
\varepsilon_j^{\frac18}+\frac{c_1^2}{2}\|\partial\mathcal E_{j,l-1}\|(\Omega).
\end{split}
\end{equation*}
This leads to the estimates \eqref{e:mean curvature bounds}
and \eqref{e:speed of lmm lip}. 
\subsection{Proofs of rectifiability and integrality}
\label{PRI}
For the proofs of rectifiability and integrality of $\mu_t$ for 
$\mathcal L^1$-a.e.~$t\in\mathbb R^+$, the smallness of 
$\Delta_j\|\partial\mathcal E(t)\|$ is essential in 
\cite[Section 7\&8]{KimTone}. The general idea is that, whenever
$\Delta_j\|\partial\mathcal E(t)\|$ is used in \cite{KimTone}, the proofs
use contradiction arguments and some appropriate Lipschitz deformations with
drastic measure reduction are constructed. Here, 
``drastic'' means that the measure is typically reduced 
by some factor of the measure 
itself, so that the reduction is typically 
much larger than the 
volume change caused by the deformation. Thus, even if 
we impose the additional condition, it is satisfied by the same Lipschitz
deformations in \cite{KimTone} and we only need to check 
that it is indeed the case. We point out the following
three separate places. 

(1) In \cite[Proposition 7.2]{KimTone}, it is proved 
    that there exists a $\mathcal E$-admissible function $f$
    which reduces the measure $\|\partial\mathcal E\|(B_r)$ 
    by the factor of $1/2$ (see (3))
    for some $r\in[R/2,R]$ when the measure in $B_R$ is 
    sufficiently small. Note that (4) 
    gives the desired estimate on the 
    change of volume of each grain in terms of
    $\|\partial\mathcal E\|(B_r)$. Since 
    the radii $r$ of balls used later are typically
    $O(1/j^2)$, thus the volume change is $O(r^{n+1})=O(r^n/j^2)$, and $\|\partial\mathcal E\|(B_r)=O(r^{n})$,
    we can deduce that 
    $f$ belongs to ${\bf E}^{vc}(\mathcal E,j)$. Since 
    the claim of \cite[Proposition 7.2]{KimTone} is about
    the existence of $\mathcal E$-admissible function, no 
    change is required in the proof. 
    
(2) In \cite[Theorem 7.3]{KimTone}, the assumption (4) 
    should be replaced by the volume-controlled counterpart.
    In the proof, on p.94-95, a Lipschitz map $f$ is defined
    with the properties stated in the bottom of p.94 using
    \cite[Proposition 7.2]{KimTone}. Using \cite[(7.40)]{KimTone} which gives
    \begin{equation*}
        (1-2^{-\sfrac12})\|\partial\mathcal E_{j_l}\|\mres_{\Omega}
        (B_{r_x}(x))\leq \|\partial\mathcal E_{j_l}\|\mres_{\Omega}(B_{r_x}(x))-\|\partial
        (f_x)_\star\mathcal E_{j_l}\|\mres_{\Omega}(B_{r_x}(x)),
    \end{equation*}
    one can proceed in \cite[(7.43)]{KimTone} as
    \begin{equation*}
        \begin{split}
            \mathcal L^{n+1}(E_i\triangle\tilde E_i) &
            \leq c_4\sum_{k=1}^\Lambda (\|\partial\mathcal E_{j_l}\|
            (B(k))^{\frac{n+1}{n}} \leq  \frac{c_4 c_3^{1/n}}{2j_l^2} \sum_{k=1}^\Lambda\|\partial \mathcal E_{j_l}\|
            (B(k)) \\
            & \leq \frac{c_4 c_3^{1/n}(\min_{B_3(x_0)}\Omega)^{-1}}{2(1-2^{-\sfrac12})j_l^2} \sum_{k=1}^\Lambda (\|\partial\mathcal E_{j_l}\|\mres_{\Omega}
            (B(k))-\|\partial f_\star \mathcal E_{j_l}\|\mres_{\Omega}(B(k))) \\
            & =\frac{c_4 c_3^{1/n}(\min_{B_3(x_0)}\Omega)^{-1}}{2(1-2^{-\sfrac12})j_l^2}  (\|\partial\mathcal E_{j_l}\|(\Omega)-\|\partial f_\star \mathcal E_{j_l}\|(\Omega)).
            \end{split}
    \end{equation*}
    Thus, for all sufficiently large $l$, Definition \ref{ar8}(b) is satisfied. 
    The use of \cite[Lemma 4.12]{KimTone} is also justified,
    and we have $f\in {\bf E}^{vc}(\mathcal E_{j_l},j_l)$.

(3) Throughout Section 8 of \cite{KimTone}, the only 
    crucial point that needs to be checked is in \cite[Lemma 8.1]{KimTone} which involves the actual construction of 
    a measure-reducing Lipschitz deformation. It proves roughly that when $\partial\mathcal E$ is flat and close to being measure-minimizing within a cylinder of size $O(1/j^2)$, then the measure 
    has to be an integer multiple of discs. The argument proceeds by assuming the contrary. The intuitive picture is that, if
    $\partial\mathcal E$ does not have a measure close to a 
    multiple of discs, then one can locate a hole which can be 
    expanded horizontally. This would cause a drastic reduction of measure and lead to a contradiction to the almost measure-minimizing property. More precisely, besides the 
    change of $\Delta_j\|\partial\mathcal E\|$ to $\Delta_j^{vc}\|\partial\mathcal E\|$ throughout, near 
    the end of the proof of \cite[Lemma 8.1]{KimTone}, p.106, one
    needs to check the ``expansion map'' $f_a$ is in ${\bf E}^{vc}(\mathcal E,E(r_1,\rho_1),j)$. We recall that
    $Y\subset T^{\perp}$ is a set of $\nu$ points with 
    \begin{eqnarray*}
        {\rm diam}\,Y<j^{-2}\,\,\mbox{(\cite[Lemma 8.1(3)]{KimTone})} \\
        r_1<R<j^{-2}/2\,\,\mbox{(4 lines above \cite[(8.7)]{KimTone} and \cite[Lemma 8.1(1)]{KimTone})} \\
        \rho_1=(1+R^{-1}r_1)\rho<2\rho<j^{-2}
        \,\,\mbox{(\cite[(8.7)]{KimTone} and \cite[Lemma 8.1(1)]{KimTone})}
    \end{eqnarray*} 
    and, with $T\in {\bf G}(n+1,n)$ fixed, 
    \begin{equation*}
        E(r_1,\rho_1)=\{x\in\mathbb R^{n+1}\,:\, |T(x)|\leq r_1,\,
        {\rm dist}\,(T^{\perp}(x),Y)\leq \rho_1\}\,\,
        \mbox{ (\cite[(8.1)]{KimTone}).}
    \end{equation*}
 The map $f_a$ is defined to be the identity map
 outside of $E(r_1,\rho_1)$, so the change of volume of grains 
 caused by
 $f_a$ is at most 
 \begin{equation}\label{A1}
 \mathcal L^{n+1}(E(r_1,\rho_1))\leq
 2\omega_{n}\nu r_1^n\rho_1<2\omega_n\nu r_1^n/j^2.
 \end{equation}
 As one can see in \cite[(8.67)]{KimTone}, the reduction of measure is
 \begin{equation}\label{A2}
     \|\partial(f_a)_{\star}\mathcal E\|(E(r_1,\rho_1))
     -\|\partial\mathcal E\|(E(r_1,\rho_1))<-\frac12 (1-\zeta)
     \omega_n r_1^n
 \end{equation}
 and we also have from \cite[(8.8)]{KimTone} that
 \begin{equation} \label{A3}
     \|\partial\mathcal E\|(E(r_1,\rho_1))=(\nu-\zeta)\omega_n r_1^n.
 \end{equation}
 We need to see the difference with the weight $\Omega$, 
 and since ${\rm diam}\,E(r_1,\rho_1)< 4/j^2$,
 \begin{equation}\label{A4}
 \begin{split}
     &\|\partial(f_a)_{\star}\mathcal E\|(\Omega)
     -\|\partial\mathcal E\|(\Omega) \\ &=
     \|\partial(f_a)_{\star}\mathcal E\|\mres_\Omega (E(r_1,\rho_1))
     -\|\partial\mathcal E\|\mres_\Omega (E(r_1,\rho_1)) \\
     &\leq (\max_{E(r_1,\rho_1)}\Omega) \|\partial(f_a)_{\star}\mathcal E\| (E(r_1,\rho_1))-(\min_{E(r_1,\rho_1)}\Omega)\|\partial\mathcal E\| (E(r_1,\rho_1)) \\
     &\leq (\min_{E(r_1,\rho_1)}\Omega)(e^{4c_1/j^2}\|\partial(f_a)_{\star}\mathcal E\| (E(r_1,\rho_1))-\|\partial\mathcal E\| (E(r_1,\rho_1))) \\
     & \leq (\min_{E(r_1,\rho_1)}\Omega)\big\{e^{4c_1/j^2}(\|\partial(f_a)_{\star}\mathcal E\| (E(r_1,\rho_1))-\|\partial\mathcal E\| (E(r_1,\rho_1)))
     \\ &\,\,\,\,+(e^{4c_1/j^2}-1)\|\partial\mathcal E\|(E(r_1,\rho_1)\big\} \\
     &\leq -\frac12 (1-\zeta)\omega_n(\min_{E(r_1,\rho_1)}\Omega)e^{4c_1/j^2}r_1^n
     +\frac{4c_1}{j^2}e^{4c_1/j^2} (\nu-\zeta)\omega_n r_1^n.
     \end{split}
 \end{equation}
 In the last line, we used \eqref{A2} and \eqref{A3}. Note that the first term of the last line is a negative term
 of order $O(r_1^n)$, while the change of volume expressed
 in \eqref{A1} is $O(r_1^n/j^2)$. Thus \eqref{A1} and \eqref{A4} give the desired 
 inequality
 \begin{equation*}
     \mathcal L^{n+1}(E_i\triangle\tilde E_i)
     \leq \mathcal L^{n+1}(E(r_1,\rho_1))
     \leq (\|\partial\mathcal E\|(\Omega)-\|\partial(f_a)_\star\mathcal E\|(\Omega))/j
 \end{equation*}
 for all sufficiently large $j$, and we have $f_a\in {\bf E}^{vc}(\mathcal E,E(r_1,\rho_1),j)$. 
 The rest of the proof is not affected by the change of
 volume-controlled deformation. 
 \subsection{Volume change of grains}
 The motivation of having $\mathcal L^{n+1}
 (E_i\triangle\tilde E_i)<1/j$ in \cite{KimTone} is 
 the use in the proof of \cite[Lemma 10.10]{KimTone}, and
 in fact, it is the only place that this inequality is 
 essentially used to derive any conclusion. 
 In the proof, see the second line from the 
 bottom of \cite[p.134]{KimTone}, it is used to make sure that the 
 volume change of grains is small for each discrete time step
 and the continuity of the labelling of each grain is derived in the end. 
The similar smallness of volume change is available with 
the volume-controlled counterparts since $\|\partial
\mathcal E_{j_\ell}(t)\|(\Omega)$ is uniformly bounded 
for a fixed time interval $[0,T]$ by \eqref{e:mass bound}.
Thus the proof can be carried out similarly. 
\subsection{Changes in \cite{KimTone2}}
 The results from \cite{KimTone2} are used in the present paper 
 and the modifications are needed there as well. On the other hand,
 similarly, one can check that the
 change to the volume-controlled counterpart does
 not cause any difficulties. The part which 
 is relevant to the change is Section 4 of \cite{KimTone2}. In the proof of Lemma 4.2, a Lipschitz retraction map $\hat{F}$ is used, with the reduction of measure
 inside of $B_{r_\ell R}(z^{(\ell)})$ being
 $\beta(r_\ell R)^n$, while the volume of the ball
 is $O((r_\ell R)^{n+1})$ (see (4.7) and (4.9)).
 Since $r_\ell=1/j_\ell^{2.5}$ (see just after (4.3)), one can check that $\hat F\in {\bf E}^{vc}(\mathcal 
 E_{j_\ell},j_\ell)$ for all large $\ell$. 
 The similar argument can be
 applied to the proof of Lemma 4.3 and Lemma 4.5. 
 The proof of Lemma 4.6 uses \cite[Lemma 8.1]{KimTone}, with the modifications discussed above in \ref{PRI}(3). In the proofs of Lemma 4.7--4.9, 
 the Lipschitz maps reduce the measure
 in similar manners, and they belong to the 
 volume-controlled counterparts. 
In particular, all of the results in \cite{KimTone2}
hold true even with the modifications. 

\bibliographystyle{plain}
\bibliography{MCF_Plateau_biblio_dot}

\begin{thebibliography}{10}

\bibitem{Allard}
W.~K. Allard.
\newblock On the first variation of a varifold.
\newblock {\em Ann. of Math. (2)}, 95:417--491, (1972).

\bibitem{ATW}
F.~J. Almgren, Jr., J.~E. Taylor, and L.~Wang.
\newblock Curvature-driven flows: a variational approach.
\newblock {\em SIAM J. Control Optim.}, 31(2):387--438, (1993).

\bibitem{AFP}
L.~Ambrosio, N.~Fusco, and D.~Pallara.
\newblock {\em Functions of bounded variation and free discontinuity problems}.
\newblock Oxford Mathematical Monographs. The Clarendon Press, Oxford
  University Press, New York, 2000.

\bibitem{Bellettini}
G.~Bellettini and S.~Yu. Kholmatov.
\newblock Minimizing movements for mean curvature flow of partitions.
\newblock {\em SIAM J. Math. Anal.}, 50(4):4117--4148, (2018).

\bibitem{Bertini}
L.~Bertini, P.~Butt\`a, and A.~Pisante.
\newblock Stochastic {A}llen-{C}ahn approximation of the mean curvature flow:
  large deviations upper bound.
\newblock {\em Arch. Ration. Mech. Anal.}, 224(2):659--707, (2017).

\bibitem{Brakke}
K.~A. Brakke.
\newblock {\em The motion of a surface by its mean curvature}, volume~20 of
  {\em Mathematical Notes}.
\newblock Princeton University Press, Princeton, N.J., 1978.

\bibitem{Bronsard}
L.~Bronsard and F.~Reitich.
\newblock On three-phase boundary motion and the singular limit of a
  vector-valued {G}inzburg-{L}andau equation.
\newblock {\em Arch. Rational Mech. Anal.}, 124(4):355--379, (1993).

\bibitem{CGG}
Y.~G. Chen, Y.~Giga, and S.~Goto.
\newblock Uniqueness and existence of viscosity solutions of generalized mean
  curvature flow equations.
\newblock {\em J. Differential Geom.}, 33(3):749--786, (1991).

\bibitem{Depner}
D.~Depner, H.~Garcke, and Y.~Kohsaka.
\newblock Mean curvature flow with triple junctions in higher space dimensions.
\newblock {\em Arch. Ration. Mech. Anal.}, 211(1):301--334, (2014).

\bibitem{Esedoglu}
S.~Esedo\={g}lu and F.~Otto.
\newblock Threshold dynamics for networks with arbitrary surface tensions.
\newblock {\em Comm. Pure Appl. Math.}, 68(5):808--864, (2015).

\bibitem{EG}
L.~C. Evans and R.~F. Gariepy.
\newblock {\em Measure theory and fine properties of functions}.
\newblock Textbooks in Mathematics. CRC Press, Boca Raton, FL, revised edition,
  2015.

\bibitem{ES}
L.~C. Evans and J.~Spruck.
\newblock Motion of level sets by mean curvature. {I}.
\newblock {\em J. Differential Geom.}, 33(3):635--681, (1991).

\bibitem{fischer}
J.~Fischer, S.~Hensel, T.~Laux, and T.~M. Simon.
\newblock The local structure of the energy landscape in multiphase mean
  curvature flow: Weak-strong uniqueness and stability of evolutions.
\newblock {\em Preprint arXiv:2003.05478}, (2020).

\bibitem{Freire}
A.~Freire.
\newblock Mean curvature motion of triple junctions of graphs in two
  dimensions.
\newblock {\em Comm. Partial Differential Equations}, 35(2):302--327, (2010).

\bibitem{Giga-Yamauchi}
Y.~Giga and K.~Yama-uchi.
\newblock On a lower bound for the extinction time of surfaces moved by mean
  curvature.
\newblock {\em Calc. Var. Partial Differential Equations}, 1(4):417--428,
  (1993).

\bibitem{Hensel-Laux_AC}
S.~Hensel and T.~Laux.
\newblock A new varifold solution concept for mean curvature flow: Convergence
  of the {A}llen-{C}ahn equation and weak-strong uniqueness.
\newblock {\em Preprint arXiv:2109.04233}, (2021).

\bibitem{Laux2}
S.~Hensel and T.~Laux.
\newblock Weak-strong uniqueness for the mean curvature flow of double bubbles.
\newblock {\em Preprint arXiv:2108.01733}, (2021).

\bibitem{Ilm_AC}
T.~Ilmanen.
\newblock Convergence of the {A}llen-{C}ahn equation to {B}rakke's motion by
  mean curvature.
\newblock {\em J. Differential Geom.}, 38(2):417--461, (1993).

\bibitem{Ilm1}
T.~Ilmanen.
\newblock Elliptic regularization and partial regularity for motion by mean
  curvature.
\newblock {\em Mem. Amer. Math. Soc.}, 108(520):x+90, (1994).

\bibitem{Kasai-Tone}
K.~Kasai and Y.~Tonegawa.
\newblock A general regularity theory for weak mean curvature flow.
\newblock {\em Calc. Var. Partial Differential Equations}, 50(1-2):1--68,
  (2014).

\bibitem{KimTone}
L.~Kim and Y.~Tonegawa.
\newblock On the mean curvature flow of grain boundaries.
\newblock {\em Ann. Inst. Fourier (Grenoble)}, 67(1):43--142, (2017).

\bibitem{KimTone2}
L.~Kim and Y.~Tonegawa.
\newblock Existence and regularity theorems of one-dimensional {B}rakke flows.
\newblock {\em Interfaces Free Bound.}, 22(4):505--550, (2020).

\bibitem{Laux1}
T.~Laux and F.~Otto.
\newblock Convergence of the thresholding scheme for multi-phase mean-curvature
  flow.
\newblock {\em Calc. Var. Partial Differential Equations}, 55(5):Art. 129, 74,
  (2016).

\bibitem{Laux-Simon}
T.~Laux and T.~M. Simon.
\newblock Convergence of the {A}llen-{C}ahn equation to multiphase mean
  curvature flow.
\newblock {\em Comm. Pure Appl. Math.}, 71(8):1597--1647, (2018).

\bibitem{Luckhaus}
S.~Luckhaus and T.~Sturzenhecker.
\newblock Implicit time discretization for the mean curvature flow equation.
\newblock {\em Calc. Var. Partial Differential Equations}, 3(2):253--271,
  (1995).

\bibitem{Maggi_book}
F.~Maggi.
\newblock {\em Sets of finite perimeter and geometric variational problems},
  volume 135 of {\em Cambridge Studies in Advanced Mathematics}.
\newblock Cambridge University Press, Cambridge, 2012.

\bibitem{Mantegazza}
C.~Mantegazza, M.~Novaga, A.~Pluda, and F.~Schulze.
\newblock Evolution of networks with multiple junctions.
\newblock {\em Preprint arXiv:1611.08254}, (2018).

\bibitem{Menne}
U.~Menne.
\newblock Second order rectifiability of integral varifolds of locally bounded
  first variation.
\newblock {\em J. Geom. Anal.}, 23(2):709--763, (2013).

\bibitem{Mugnai}
L.~Mugnai and M.~R\"{o}ger.
\newblock The {A}llen-{C}ahn action functional in higher dimensions.
\newblock {\em Interfaces Free Bound.}, 10(1):45--78, (2008).

\bibitem{Mullins}
W.~W. Mullins.
\newblock Two-dimensional motion of idealized grain boundaries.
\newblock {\em J. Appl. Phys.}, 27:900--904, (1956).

\bibitem{R-S}
M.~R\"{o}ger and R.~Sch\"{a}tzle.
\newblock On a modified conjecture of {D}e {G}iorgi.
\newblock {\em Math. Z.}, 254(4):675--714, (2006).

\bibitem{SW}
F.~Schulze and B.~White.
\newblock A local regularity theorem for mean curvature flow with triple edges.
\newblock {\em J. Reine Angew. Math.}, 758:281--305, (2020).

\bibitem{Simon}
L.~Simon.
\newblock {\em Lectures on geometric measure theory}, volume~3 of {\em
  Proceedings of the Centre for Mathematical Analysis, Australian National
  University}.
\newblock Australian National University, Centre for Mathematical Analysis,
  Canberra, 1983.

\bibitem{stu-tone2}
S.~Stuvard and Y.~Tonegawa.
\newblock Dynamical instability of minimal surfaces at flat singular points.
\newblock {\em Preprint arXiv:2008.13728}, (2020).

\bibitem{ST19}
S.~Stuvard and Y.~Tonegawa.
\newblock An existence theorem for {B}rakke flow with fixed boundary
  conditions.
\newblock {\em Calc. Var. Partial Differential Equations}, 60(1):Paper No. 43,
  53, (2021).

\bibitem{Takasao-Tonegawa}
K.~Takasao and Y.~Tonegawa.
\newblock Existence and regularity of mean curvature flow with transport term
  in higher dimensions.
\newblock {\em Math. Ann.}, 364(3-4):857--935, (2016).

\bibitem{Ton3}
Y.~Tonegawa.
\newblock Integrality of varifolds in the singular limit of reaction-diffusion
  equations.
\newblock {\em Hiroshima Math. J.}, 33(3):323--341, (2003).

\bibitem{Ton-2}
Y.~Tonegawa.
\newblock A second derivative {H}\"{o}lder estimate for weak mean curvature
  flow.
\newblock {\em Adv. Calc. Var.}, 7(1):91--138, (2014).

\bibitem{Ton1}
Y.~Tonegawa.
\newblock {\em {B}rakke's mean curvature flow: An introduction}.
\newblock SpringerBriefs in Mathematics. Springer, Singapore, 2019.

\end{thebibliography}

\end{document}